\newif\ifstandardtemplate\standardtemplatetrue%
\newif\ifelseviertemplate\elseviertemplatefalse%
\newif\ifspringertemplate\springertemplatefalse%
\newif\ifwileytemplate\wileytemplatefalse%
\newcommand{\mytitle}{On the pure traction problem of linear elasticity: 
a regularized formulation and
its robust approximation}

\newcommand{\myabstract}{The pure traction problem of elasticity appears frequently in engineering applications, and its complexity stems from the fact that its solution is unique only up to (infinitesimal) rigid body motions. When finite elements are employed to approximate this problem, one solution is typically singled out by applying carefully selected boundary conditions on the discrete model or by imposing global constraints on the deformation. However, neither of these strategies is both simple and computationally efficient. In this work, we propose a new approach to solving the pure traction problem that overcomes existing limitations. Our method builds on a regularized form of the problem whose solution is shown to be unique, converges to the original solution of minimal norm, and can be approximated with finite elements in a straightforward way, without additional degrees of freedom. Additionally, we analyze the situation in which the approximation of the solution domain renders the loading of the discretized problem non-equilibrated, making the problem ill-posed. In this case, we propose a regularized predictor--corrector finite element formulation that handles the incompatibilities of the loading, providing a solution that converges to that of the original Neumann problem as the mesh size and the regularizing parameter tend to zero. Numerical examples illustrate the effectiveness of the proposed approach for representative problems in mechanics where pure traction boundary conditions appear.}
\newcommand{\myack}{AK and IR acknowledge the funding received, respectively, from projects 
PLEC2023-010190 and PID2021-128812OB-I00 from from the Spanish Ministry of Science and Innovation. CG would like to acknowledge the funding received from the European Research Council through the ERC Consolidator Grant ``DATA-DRIVEN OFFSHORE'' (action no. 101083157).}

\newcommand{\mypackages}{%
  \usepackage{amssymb}
  \usepackage{amsmath}
  \usepackage{amsthm}
  \usepackage{graphicx}
  \usepackage{float}
  \graphicspath{{Figures/}{figures/}}
  \usepackage[dvipsnames]{xcolor}
  \usepackage{natbib}
  \usepackage{todonotes}
  \usepackage{siunitx}
  \usepackage{booktabs}
  \usepackage{changes}
  \definechangesauthor[name=Ignacio, color=Orange]{IRO}
  \usepackage{url}
  \usepackage{comment}
  \usepackage{tabularx}
\usepackage{siunitx}
\sisetup{
  detect-all,
  table-number-alignment = center,
  round-mode = places,
  round-precision = 3
}
}

\newcommand{\mymacros}{%
  \newcommand{\concept}[1]{\textbf{\emph{##1}}}
  \newcommand{\defined}{:=}
  \newcommand{\etanorm}[1]{\triple{##1}_{\eta}}
  \newcommand{\mbs}[1]{\boldsymbol{##1}}
  \newcommand{\mbsf}[1]{\mbs{\mathsf{##1}}}
  \newcommand{\pairing}[2]{\langle{##1},{##2}\rangle}
  \newcommand{\dd}[2]{\frac{\mathrm{d} ##1}{\mathrm{d} ##2}}
  \newcommand{\pd}[2]{\frac{\partial{##1}}{\partial{##2}}}
  \newcommand{\fd}[2]{\frac{\delta{##1}}{\delta{##2}}}
  \newcommand{\set}[1]{\left\{##1\right\}}
  \newcommand{\trace}{{\mathop{\mathrm{tr}}}}
  \newcommand{\triple}[1]{{\left\vert\kern-0.25ex\left\vert\kern-0.25ex\left\vert ##1
    \right\vert\kern-0.25ex\right\vert\kern-0.25ex\right\vert}}
  \newcommand{\uptohere}{\centerline{\textcolor{blue}{\rule{6cm}{0.2cm}}}}
  \let\oldLambda=\Lambda\renewcommand{\Lambda}{\mathit{\oldLambda}}
  \let\oldGamma=\Gamma\renewcommand{\Gamma}{\mathit{\oldGamma}}
}
\ifstandardtemplate%
\documentclass[10pt,a4paper]{article}
\mypackages%
\usepackage[noblocks]{authblk}
\mymacros%
\newcommand{\mybibstyle}{unsrt}

\newtheorem{theorem}    {Theorem}[section]
\newtheorem{lemma}      [theorem]{Lemma}

\theoremstyle{definition}
\newtheorem{definition} [theorem]{Definition}

\newtheorem{examplex}[theorem]{$\triangleright\;$Ejemplo}

\theoremstyle{remark}

\newtheorem{remarks}{Remarks}

\title{\mytitle}
\author[1,2]{Ahsan Kaleem}
\author[3]{Cristian Gebhardt}
\author[1,2]{Ignacio Romero}

\affil[1]{Dept. of Mechanical Engineering, Universidad Politécnica de Madrid,
  Jos\'{e} Guti\'{e}rrez Abascal, 2, 28006 Madrid, Spain}
\affil[2]{IMDEA Materials Institute, Eric Kandel 2, 28096 Getafe, Madrid, Spain}
\affil[3]{Geophysical Institute and Bergen Offshore Wind Centre, University of Bergen, Allégaten 70, 5007 Bergen, Norway}

\newenvironment{acknowledgements}{\section*{Acknowledgements}}{}

\begin{document}
\maketitle
\begin{abstract}
  \myabstract%
\end{abstract}
\fi

%--------------------------------------------------------------------------------------------------
%                                        The document
%--------------------------------------------------------------------------------------------------

\section{Introduction}
\label{sec-intro}
In elliptic boundary value problems, the \emph{pure Neumann} problem refers to the case where no
essential (Dirichlet) boundary conditions are imposed on the boundary of the solution
domain. In elasticity, the pure Neumann problem is often referred to as the \emph{traction} problem, for it
corresponds to situations in which the boundary of the body under analysis is free of displacement
constraints, and only tractions can be applied on it. In this work we will focus on the
analysis and numerical solution of non-trivial traction problems of small strain elasticity
but note that the situation is shared by other elliptic problems such as the heat conduction,
the calculation of the warping function in Saint Venant theory \cite{roccia2024}, and others.

This class of problems is of interest from the theoretical point of view, but also from the
practical one since there are physical situations that are only correctly described as pure
Neumann problems or easily approximated as such. For example, the deformation of a body that is
submerged in the sea (without any anchoring device to the seabed) needs to be studied as a pure
Neumann problem. Also, the deformation of an elastic solid placed on a flat surface and
subject only to forces parallel to the foundation might be aptly approximated as a Neumann problem.

There are no essential difficulties for the solution of a Neumann problem, at least in theory. The
problem has been fully analyzed for a long time, and questions on its well-posedness can be found in
standard references, e.~g. \cite{fichera1972in,gurtin72tv,duvaut1976ud,ciarlet1988ux}. Among other things, it can be shown that, in addition to the standard smoothness assumptions on the problem boundary, the bounds on the Lam\'e coefficients, and the regularity of the loading, the applied volumetric and
surface loads must satisfy some global equilibrium conditions. When these are satisfied, the existence,
uniqueness, and continuous dependence of the data can be proven in the space of finite energy
displacements modulo infinitesimal rigid body motions.

Whereas dealing with quotient spaces poses no theoretical difficulties, it complicates the design of
numerical methods that can approximate the solution to the traction problem. When finite
elements or any other Galerkin-type method is employed, it is easy to find a finite-dimensional
space of functions with finite energy, but working with equivalent classes directly might not
practical or possible.

To avoid the aforementioned difficulties and still solve pure traction problems, finite element
practitioners have resorted to projection techniques, that is, schemes that encapsulate a selection method that
picks, for any given \emph{discretized} traction problem, one and only one function in the equivalent class of
solutions. The simplest technique commonly employed in finite element approximations consists of
constraining the minimum number of nodal displacements that forbid any rigid body motion while
allowing all possible deformations. In two-dimensional problems, this amounts to constraining three
nodal displacements, while in three-dimensional problems a total of six nodal displacements must be
set to zero. To allow all possible deformations of the body, these nodal constraints must be
carefully selected (see, e.g., \cite{cook1995vq}), a strategy that is not well-defined in general \cite{babuska1985tm}
and can spoil the condition number of the stiffness matrix \cite{bochev2005vk,bochev2011tv}. Instead, a robust generalization
of these ideas requires the identification of the (infinitesimal) rigid body modes of the body and their penalization \cite{papadrakakis2001qp}.

An alternative strategy for solving pure traction problems is the so-called \emph{inertial relief}
method \cite{liao2011dx,pengqiu2017kd}. This numerical technique is available in several commercial
finite element codes and is
based on replacing the quasistatic, pure traction problem with an ancillary transient problem where
inertial forces are added to all rigid body motions. This procedure assumes a split of the
displacement field into ``elastic'' and rigid body displacements that can then be used to reduce the
equilibrium equations at the expense of modifying the structure of the stiffness matrix
\cite{pengqiu2017kd}. This approach is closely related to the \emph{null space} method for the
solution of semidefinite systems of linear equations (see, e.g., \cite{benzi2005zt}). In all these
cases, the kernel of the system matrix needs to be explicitly found, leaving a well-defined system
of equations only for the relevant components of the solution.

A third approach that can yield a unique solution to the traction problem in finite element
solutions consists of employing special algebraic solvers for the discrete equations of equilibrium.
The stiffness matrix in these problems has a non-trivial kernel, and special direct solvers can be
implemented to bypass the singularities \cite{farhat1998qw}; alternatively, (iterative)
Krylov-type methods can be devised to converge to the solution of minimal norm of a singular linear
system of equations \cite{kaasschieter1988sn,elman2005uy,kuchta2018xy}.

Yet another method that can be used to obtain approximated solutions to the pure Neumann problem
involves transforming it into a saddle point problem with constraints that enforce the
projection of the quotient space as described in detail in the work by Bochev and Lehoucq
\cite{bochev2005vk,bochev2011tv}.
In addition to providing a rigorous framework to understand the effects of using point constraints for removing rigid body
modes, this work introduces a constrained and regularized formulation that cast the Neumann
problem in a form that is suitable for approximation and analysis. In particular, the regularized
formulation that they propose eliminates the need for Lagrange multipliers and shows some
resemblance to the methods advocated in the current work. In their work, however, a weighting
function needs to be introduced that may couple all the degrees of freedom of the model, thus
spoiling the sparsity pattern of the stiffness matrix leading, for the most obvious choice, to a
\emph{full} stiffness matrix. The idea of penalizing the rigid body motions has been explored
\cite{necas1981sw,kirzek1994fv,ivanov2019qf}, but in all these solutions the sparsity of the stiffness matrix
is spoiled.

Finally, a solution to the Neumann problem that is closely related to the one presented here
was provided by Day \cite{dai2007ol}. This work eliminates the singularity of the Neumann problem
by adding penalty terms involving the trace of the solution on the domain's boundary. Moreover,
an iterative solution is proposed that alleviates the ill-conditioning problem of the stiffness
matrix when the penalty parameter is small. An inconvenient drawback of this formulation --- at
least from the implementation point of view --- is that it requires the computation of
a boundary integral and special boundary elements need to be introduced in the mesh for that
purpose.

In this work we present an alternative approximation method for the pure traction problem in
elasticity. The basic idea is to regularize the energy of the original Neumann problem in such a way
that the new energy has a unique minimizer. Crucially, the minimizer of the perturbed energy can
be shown to converge to one of the minimizers of the original (non-strickly convex) Neumann energy as the regularizing
parameter goes to zero. More precisely, the regularized solution converges to the solution of minimal
$L^2$ norm of the pure traction problem. Remarkably, the regularized energy is amenable to numerical discretization, thus providing a convenient avenue for the formulation of well-posed, convergent finite element discretizations of the traction problem without the need for Lagrange multipliers and without modifying the sparsity pattern of
the stiffness matrix of the discrete Neumann problem.

When attempting to find an approximate solution of a pure traction problem, the domain of analysis
is discretized using, for example, a polyhedral mesh. If the original domain itself is not
polyhedral, it may happen that it is not exactly represented by the mesh, irrespective of the mesh
size. These errors are critical for the Neumann problem because the loads on the discrete domain
might not be compatible and the problem, strictly speaking, would not be well-posed. In the last
part of the article we analyze this situation and we show that a new predictor/corrector method can
be proposed that, based on the regularized formulation, yields a convergent approximation to the Neumann problem, even when the
loads become ``slightly'' non-equilibrated.

The remainder of the article has the following structure. In Section~\ref{sec-mot}, we present, in the simplest possible setting, the idea that justifies the method analyzed in the article. Then, starting from the general problem statement as described in Section~\ref{sec-neumann}, a regularized Neumann problem is presented in Section~\ref{sec-regularized} whose solution approaches the one of the original problem, as shown in Section~\ref{sec-analysis}. This analytical result lays the ground for simple and consistent approximation schemes presented in Section~\ref{sec-fem} for the Neumann
problem, and in Section~\ref{sec-noneq} for slightly non-equilibrated problems. Finally, Section~\ref{sec-examples} illustrates the properties of the proposed method by
means of three numerical examples. The article closes with a summary of the main results and some final comments in Section~\ref{sec-summary}.

Notation: as customary in mechanics, boldface symbols refer to vectors or tensors. The space of
functions that are (Lebesgue) square integrable over the set $\Omega\in\mathbb{R}^d$ is indicated as $L^2(\Omega)$; vector
fields $\mbs{v}:\Omega\to\mathbb{R}^d$ with components in $L^2(\Omega)$ are said to belong to $[L^2(\Omega)]^d$. Similarly,
the Hilbert space of functions that are in $L^2(\Omega)$ and whose gradient is $[L^2(\Omega)]^d$ are said
to belong to $H^1(\Omega)$. Finally, vector fields on $\Omega$ with components in $H^1(\Omega)$ belong
to the set $[H^1(\Omega)]^d$. In all the bounds, the symbol $C$ indicates a generic constant, possibly
attaining different values, but not depending on the meshsize $h$ or the regularizing
parameter~$\eta$. Similarly, the notation $a \lesssim b$ would be used to replace $a \le C\,b$,
where $C$ is a constant with the properties previously indicated.

% ---------------------------------------------------------------------------------------
\section{Motivation}
\label{sec-mot}
By way of motivation, we consider a simple mechanical system that can be analyzed as
a pure Neumann problem and review several ways of finding \emph{a} solution for it. While the
problem itself is trivial, some of the difficulties and solution strategies for more complex
problems can already be identified in this simple context.

\begin{figure}[ht]
  \centering
  \includegraphics[width=0.9\textwidth]{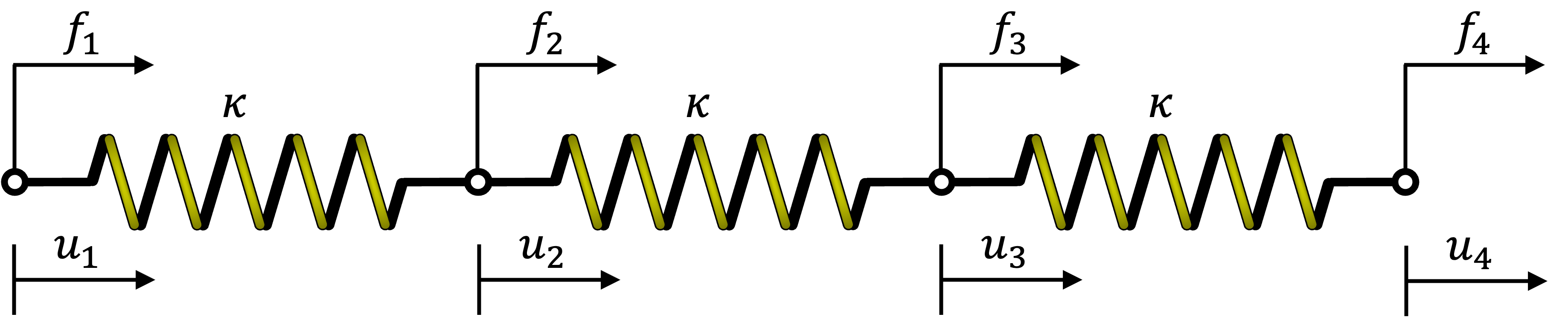}
  \caption{Simple mechanical problem.}
  \label{fig-spring3}
\end{figure}

Specifically, we consider a structure consisting of $N=4$ nodes connected by $N-1$ elastic springs
of constant $\kappa>0$, depicted in Fig.~\ref{fig-spring3}, subject to forces $\mbs{f} = \lbrace
f_1, f_2, f_3, f_4 \rbrace^T$ on the nodes. The latter are only allowed to move horizontally and we
collect their corresponding displacements in the vector $\mbs{u}=\lbrace u_1,u_2, u_{3}, u_4\rbrace^T$.

This system has four degrees of freedom, no displacement (Dirichlet) boundary conditions, and its potential energy function is
\begin{equation}
  V(\mbs{u}) \defined
  \sum_{k=1}^{N-1} \frac{\kappa}{2}\, (u_{k+1}-u_k)^2  - \mbs{u}\cdot\mbs{f}\;,
\end{equation}
where here and later we have employed a dot to indicate the scalar product of vectors of the same dimension.
If the pure traction problem is to have a solution, or rather, infinite solutions, the external loads $\mbs{f}$ must be
equilibrated, that is $\sum f_i=0$. For concreteness, let us assume that $-f_1=f_4=f, f_2=f_3=0$,
for some scalar $f$. Then, any displacement vector $\mbs{u}$ satisfying $u_{k+1}-u_k = f/\kappa$ for
$1\le k \le N-1$ is a solution to the problem.

To \emph{select} one of the infinite solutions to the problem, it suffices to pick any node $1\le J \le 4$ and impose $u_J=0$. Also, all these solutions are minimizers of
the potential energy function, hence, they satisfy the optimality condition $\nabla V(\mbs{u})=\mbs{0}$, or equivalently,
\begin{equation}
  \label{eq-matrix-k}
  \begin{bmatrix}
    \kappa & -\kappa & 0 & 0\\
    -\kappa & 2\kappa & -\kappa & 0\\
    0       & -\kappa & 2\kappa & -\kappa\\
    0       &     0   & -\kappa & \kappa\\
  \end{bmatrix}
  \begin{Bmatrix}
    u_1 \\ u_2 \\ u_3 \\ u_{4}
  \end{Bmatrix}
  =
  \begin{Bmatrix}
    -f \\ 0 \\ 0 \\ f
  \end{Bmatrix}\ .
\end{equation}
The stiffness matrix of the problem is, of course, singular, anticipating potential difficulties in
the numerical solution to this and similar problems. We can select one of the infinite solutions to the problem by imposing, for example, that
the mean displacement must be equal to zero. To solve this constrained problem, we
must find the saddle point of the Lagrangian function
\begin{equation}
    L(\mbs{u}, \lambda) \defined V(\mbs{u}) + \lambda\, \langle \mbs{u} \rangle \ ,
\end{equation}
where $\langle \mbs{u} \rangle \defined (u_1+u_2+u_3+u_4)/4$. The optimality condition of this new problem
is
\begin{equation}
  \begin{bmatrix}
    \kappa & -\kappa & 0 & 0 & 1/4\\
    -\kappa & 2\kappa & -\kappa & 0& 1/4\\
    0       & -\kappa & 2\kappa & -\kappa& 1/4\\
    0       &     0   & -\kappa & \kappa& 1/4\\
    1/4     & 1/4     & 1/4     & 1/4   & 0
  \end{bmatrix}
  \begin{Bmatrix}
    u_1 \\ u_2 \\ u_3 \\ u_{4} \\ \lambda
  \end{Bmatrix}
  =
  \begin{Bmatrix}
    -f \\ 0 \\ 0 \\ f \\ 0
  \end{Bmatrix}\ ,
\end{equation}
whose unique solution is
\begin{equation}
  \label{eq-mot-lambda}
    u_1^{\lambda} = - \frac{3f}{2 \kappa}, \qquad
    u_2^{\lambda} = - \frac{f}{2 \kappa}, \qquad
    u_3^{\lambda} = \frac{f}{2 \kappa}, \qquad
    u_4^{\lambda} = \frac{3f}{2 \kappa}, \qquad
    \lambda = 0\ .
\end{equation}
Following this approach, we can transform a problem
with infinite solutions into another one, with an additional Lagrange multiplier, but with a unique
solution. In fact, we could have constrained the mean value $\langle \mbs{u} \rangle$ to any other
value and still obtain a unique solution.

Alternatively, and to avoid the new Lagrange multiplier, we could search for a penalty solution that
minimizes the function
\begin{equation}
    V_\alpha(\mbs{u}) \defined V(\mbs{u}) + \frac{\alpha}{2}\langle \mbs{u} \rangle^2\ ,
\end{equation}
for some large positive parameter~$\alpha$. The optimality condition for this function reads
\begin{equation}
  \label{eq-matrix}
  \begin{bmatrix}
    \kappa + \alpha& -\kappa+ \alpha & \alpha & \alpha\\
    -\kappa + \alpha & 2\kappa + \alpha & -\kappa + \alpha & \alpha\\
    \alpha   & -\kappa + \alpha& 2\kappa + \alpha& -\kappa + \alpha\\
    \alpha   &  \alpha   & -\kappa + \alpha& \kappa + \alpha\\
  \end{bmatrix}
  \begin{Bmatrix}
    u_1 \\ u_2 \\ u_3 \\ u_{4}
  \end{Bmatrix}
  =
  \begin{Bmatrix}
    -f \\ 0 \\ 0 \\ f
  \end{Bmatrix}\ ,
\end{equation}
that employs a regular stiffness matrix and whose solution is again
\begin{equation}
  \label{eq-mot-alpha}
    u_1^{\alpha} = - \frac{3f}{2 \kappa}, \qquad
    u_2^{\alpha} = - \frac{f}{2 \kappa}, \qquad
    u_3^{\alpha} = \frac{f}{2 \kappa}, \qquad
    u_4^{\alpha} = \frac{3f}{2 \kappa} .
\end{equation}

We have seen that both the Lagrange multiplier and the penalty methods remove the singularity of the
stiffness matrix, effectively providing a route for the solution of the pure traction problem. Both
strategies, however, spoil the sparsity of the stiffness matrix. In fact, the band of the matrix in
the Lagrange multipliers method changes from $2$ to $N$, and the matrix in the penalty method is
completely full. Whereas for this illustrative example the band size and/or the sparsity of the
stiffness matrix is completely irrelevant, these two solutions prove extremely expensive --- and thus
impractical --- for Engineering problems with many degrees of freedom.

There is, however, another path that relies on a perturbed potential of the form
\begin{equation}
\label{eq-mot-veta}
    V_\eta(\mbs{u}) \defined V(\mbs{u}) + \frac{\eta}{2} \|\mbs{u}\|^2_2\ ,
\end{equation}
where $\eta$ is a \emph{small} positive parameter and $\|\cdot\|_2$ stands for the Euclidean norm. The optimality condition for this regularized function reads
\begin{equation}
  \label{eq-matrix-eta}
  \begin{bmatrix}
    \kappa+\eta & -\kappa & 0 & 0\\
    -\kappa & 2\kappa+\eta & -\kappa & 0\\
    0       & -\kappa & 2\kappa+\eta & -\kappa\\
    0       &     0   & -\kappa & \kappa+\eta\\
  \end{bmatrix}
  \begin{Bmatrix}
    u_1 \\ u_2 \\ u_3 \\ u_{4}
  \end{Bmatrix}
  =
  \begin{Bmatrix}
    -f \\ 0 \\ 0 \\ f
  \end{Bmatrix}\ ,
\end{equation}
and it is trivial to verify that the stiffness matrix is now regular. The unique solution to this
matrix equation is
\begin{subequations}
  \begin{align}
    u_1^\eta &=-\frac{3f}{2\,\kappa}\; \frac{ 1 + \frac{\eta}{3\kappa}}{1 + \frac{2\eta}{\kappa} + \frac{\eta^2}{2\kappa^2}}\ ,\\
    u_2^\eta &=-\frac{f}{\kappa}\; \frac{ 1 + \frac{\eta}{3\kappa}}{1 + \frac{2\eta}{\kappa} + \frac{\eta^2}{2\kappa^2}}\ ,\\
    u_3^\eta &= \frac{f}{\kappa}\; \frac{ 1 + \frac{\eta}{3\kappa}}{1 + \frac{2\eta}{\kappa} + \frac{\eta^2}{2\kappa^2}}\ ,\\
    u_4^\eta &= \frac{3f}{2\,\kappa}\; \frac{ 1 + \frac{\eta}{3\kappa}}{1 + \frac{2\eta}{\kappa} + \frac{\eta^2}{2\kappa^2}}\ .
  \end{align}
\end{subequations}
The solution corresponding to the regularized potential energy consistently approximates the
solutions $\mbs{u}^\lambda\equiv\mbs{u}^\alpha$, i. e.,
\begin{equation}
    \lim_{\eta \downarrow 0} \mbs{u}^\eta = \mbs{u}^\lambda = \mbs{u}^\alpha \ ,
\end{equation}
while also satisfying $\langle \mbs{u}^\eta \rangle = 0$. For this simple problem, the error of the regularized solution relative to the exact solution with zero mean can be explicitly calculated
\begin{equation}
    \|\mbs{u}^\eta-\mbs{u}^\lambda \|_2 =
    \frac{\eta \,\left|f\right|\,\sqrt{5\,\eta ^2+34\,\eta \,\kappa+58\,\kappa^2}}{\kappa\,\left(\eta ^2+4\,\eta \,\kappa+2\,\kappa^2\right)}\ .
    %\leq C\; \eta\; \|\mbs{f}\|_2\ .
\end{equation}

This result shows that the displacements $\mbs{u}^{\eta}$ may be as close to $\mbs{u}^\lambda$ as desired.
Crucially, and in contrast with the approaches based on Lagrange multipliers or penalization, the
regular stiffness matrix of Eq.~\eqref{eq-matrix-eta} has the same sparsity pattern as the
(singular) stiffness matrix of Eq.~\eqref{eq-matrix-k}. This fact suggests that the
regularization~\eqref{eq-mot-veta} can be
used as a means to finding approximations to the pure traction problems with a negligible additional cost.
In the remainder of this work, we will prove that,
indeed, the regularization idea can be generalized to (small strain) elasticity problems, and that
it can be used for efficient finite element approximations.

%--------------------------------------

\section{The pure Neumann problem for elasticity}
\label{sec-neumann}
In this article, we study the solution of the pure Neumann problem in small strain elasticity. To
define it, consider a domain $\Omega\in \mathbb{R}^d$, with $d=2$ or $3$, and boundary $\partial\Omega$.
Given a field of
volume forces $\mbs{f}:\Omega\to \mathbb{R}^d$ and a field of surface tractions
$\mbs{t}:\partial\Omega\to \mathbb{R}^d$, we seek for the displacement field $\mbs{u}:\Omega\to
\mathbb{R}^d$ that minimizes the potential energy
\begin{equation}
  \label{eq-potential}
  V(\mbs{u}) = \int_{\Omega} W(\nabla^s\mbs{u}(\mbs{x}))\; dV
  - \int_{\Omega} \mbs{f}(\mbs{x})\cdot \mbs{u}(\mbs{x}) \; dV
  - \int_{\partial\Omega} \mbs{t}(\mbs{x})\cdot \mbs{u}(\mbs{x}) \; dA .
\end{equation}
In this expression, $\mbs{x}$ denotes points on $\Omega$, $\nabla^s$ is the symmetric gradient
operator and $W$ is the stored energy density function.

The infinitesimal rigid body motions play a crucial role in the development that follows. For
completeness, let us review this basic concept and the related \emph{compatibility} condition.

\begin{definition}
  Let $\mathcal{R}$ be the space of infinitesimal rigid body motions, that is, vector fields $\mbs{u}\in [C^{\infty}(\Omega)]^d$ of the form
  \begin{equation}
  \label{eq-irbm}
  \mbs{\rho}(\mbs{x}) = \mbs{\alpha} + \mbs{\beta}\, \mbs{x}\ ,
\end{equation}
where $\mbs{\alpha}\in \mathbb{R}^d$ and $\mbs{\beta}$ is a skew symmetric second-order tensor on $\mathbb{R}^{d}$.
\end{definition}

\begin{definition}
Let $\mbs{u},\mbs{v}$ be two vector fields in $[L^2(\Omega)]^d$. We say that they are (infinitesimally) rigid body
equivalent, and we indicate it as $\mbs{u}\sim \mbs{v}$, if there exists a rigid body motion~$\mbs{\rho}$  such that
  \begin{equation}
    \| \mbs{u} + \mbs{\rho}  - \mbs{v} \|_{[L^2(\Omega)]^d} = 0\ .
  \end{equation}
This is an equivalence relation that partitions functional spaces defined on $\Omega$ into equivalence classes. We define $\mathcal{S}$ as the quotient space of $[H^1(\Omega)]^d$ modulo rigid body motions and we use the
notation $[\mbs{u}]\in\mathcal{S}$ to denote the vector fields in $[H^1(\Omega)]^d$ that are rigid
body equivalent to~$\mbs{u}$.
\end{definition}

\begin{remarks}
  \item Since the space of all infinitesimal rigid body motions is a closed subspace of  $[H^1(\Omega)]^d$, $\mathcal{S}$ becomes a normed space with respect to the norm
\begin{equation}
    \| [\mbs{u}] \|_{\mathcal{S}} = \inf_{\mbs{v} \in [\mbs{u}]} \| \mbs{v} \|_{[H^1(\Omega)]^d}.
\end{equation}
% (Zeidler, 1995, vol. 109, p. 185, Prop. 2)

\item The space $\mathcal{R}$ has finite dimension. Hence, all the norms on $\mathcal{R}$ are equivalent.
\end{remarks}

\begin{definition}
  Let $\mbs{f}:\Omega\to \mathbb{R}^d$ be a vector field of volumetric forces and
  $\mbs{t}:\partial\Omega\to \mathbb{R}^d$ a vector field of tractions defined on the boundary
  of~$\Omega$. This loading system is \emph{equilibrated} or  \emph{compatible} if
  \begin{equation}
    \label{eq-equilibrium2}
    \int_{\Omega} \mbs{f}(\mbs{x})\cdot \mbs{\rho}(\mbs{x}) \;dV +
    \int_{\partial\Omega} \mbs{t}(\mbs{x})\cdot \mbs{\rho}(\mbs{x}) \;dA = \mbs{0}\ ,
  \end{equation}
  for all rigid body motions~$\mbs{\rho}$. If $d=3$, the previous condition is equivalent to
    \begin{equation}
    \label{eq-equilibrium}
    \begin{split}
    \int_{\Omega} \mbs{f}(\mbs{x}) \;dV + \int_{\partial\Omega} \mbs{t}(\mbs{x}) \;dA &= \mbs{0}\ ,\\
    \int_{\Omega} \mbs{x} \times \mbs{f}(\mbs{x}) \;dV +
    \int_{\partial\Omega} \mbs{x} \times\mbs{t}(\mbs{x}) \;dA &= \mbs{0}\ .\\
    \end{split}
  \end{equation}
\end{definition}

The classical elasticity problem has been thoroughly analyzed in the
literature~\cite{fichera1972in,ciarlet1988ux}. We summarize next the main result regarding the
well-posedness of the Neumann problem:

\begin{theorem}[\cite{ciarlet1978tm}]
  Let $\Omega$ be an open, bounded set in $\mathbb{R}^d$. Let $\partial\Omega$,
  the boundary of $\Omega$ be Lipschitz, and consider a homogeneous, isotropic elastic body with
  stored energy function $W:\text{Sym}^{d\times d}\to \mathbb{R}$
  \begin{equation}
    \label{eq-stored-energy}
    W(\mbs{\epsilon}) = \frac{\lambda}{2} (\mathrm{tr}[\mbs{\epsilon}])^2 + \mu\, \mbs{\epsilon}: \mbs{\epsilon}\ ,
  \end{equation}
  where $\mathrm{tr}[\cdot]$ is the trace operator,  $\lambda$ and $\mu$ are the Lamé constants that
  satisfy $\lambda,\mu>0$, and the colon product denotes the contraction of second-order
  tensors. 
  %Let $\mbs{f}\in [L^{p}(\Omega)]^d$ and $\mbs{t}\in [L^{q}(\partial\Omega)]^d$ be compatible
  %loads with $(p,q)=(6/5,4/3)$ if $d=3$, and $(p,q)=(2,2)$ if $d=2$.
  Let $\mbs{f}$ and $\mbs{t}$ be compatible loads.
  Then, the pure traction problem has a unique solution $[\mbs{u}]\in \mathcal{S}$ that depends
  continuously on the loading and corresponds to a minimizer of the energy~$V:\mathcal{S}\to \mathbb{R}$.
\end{theorem}

The weak form that corresponds to the minimization of~\eqref{eq-potential} consists in finding
$[\mbs{u}]\in \mathcal{S}$ such that
\begin{equation}
\label{eq-weak-neumann}
    a(\mbs{u}, \mbs{v}) = \ell(\mbs{v})\ ,
\end{equation}
for all $[\mbs{v}]\in \mathcal{S}$, where
\begin{subequations}
  \label{eq-forms}
  \begin{align}
a(\mbs{u}, \mbs{v})
  &\defined
\int_{\Omega} (\mu \nabla^s \mbs{u}(\mbs{x}): \nabla^s \mbs{v}(\mbs{x}) +
    \lambda \nabla\cdot \mbs{u}(\mbs{x})\;\nabla\cdot \mbs{v}(\mbs{x})) \; dV\ ,
    \label{eq-forms1}\\
\ell(\mbs{v}) &\defined
\int_{\Omega} \mbs{f}(\mbs{x})\cdot \mbs{v}(\mbs{x}) \; dV
+
                \int_{\partial\Omega} \mbs{t}(\mbs{x})\cdot \mbs{v}(\mbs{x}) \; dA\; .
                \label{eq-forms2}
  \end{align}
\end{subequations}
In view of Eq.~\eqref{eq-equilibrium}, $\ell(\mbs{\rho}) = \mbs{0}$, for
any $\mbs{\rho}\in \mathcal{R}$. Also, we note that $\nabla^s\mbs{\rho}=\mbs{0}$, thus
\begin{equation}
\label{eq-a-orthogonal}
a(\mbs{v},\mbs{\rho}) = a(\mbs{\rho},\mbs{v}) = 0\ ,
\end{equation}
for arbitrary $\mbs{\rho}\in \mathcal{R}$ and displacement vector fields~$\mbs{v}\in [H^1(\Omega)]^d$.

The main difficulty for finding a numerical approximation to the solution of the Neumann problem is precisely
factoring out $\mathcal{R}$ from the space where the solution is sought. In finite element analyses, for
example, this projection has often been accomplished by introducing, for each Neumman problem, carefully selected
boundary conditions that remove the rigid body motions from the solution, effectively selecting a
single displacement field from the quotient space $[H^1(\Omega)]^d/\mathcal{R}$. These constraints have to
be imposed without inadvertently introducing any undesired strains on the solution, hence the
procedure is non-systematic and prone to error.

\section{A regularized Neumann problem}
\label{sec-regularized}
The Neumann problem of elasticity, as formulated in Section~\ref{sec-neumann}, has a unique solution
that is an equivalent class of rigid-body-equivalent functions in a Hilbert space. Anticipating the formulation of a
finite element approximation of the problem, we postulate a new boundary value problem obtained by
regularizing the original Neumann problem. This new formulation is posed on a standard Hilbert
space, so its Galerkin discretization is straightforward.

The sought vector field $\mbs{u}_{\eta}\in [H^1(\Omega)]^d$ that minimizes
the regularized potential
\begin{equation}
\label{eq-veta}
V_{\eta}(\mbs{v}) \defined V(\mbs{v}) + \frac{\eta}{2} \| \mbs{v} \|^2_{[L^2(\Omega)]^d}\ ,
\end{equation}
for some \textit{small} $\eta>0$. Alternatively, $\mbs{u}_{\eta}\in[H^1(\Omega)]^d$ is the function that satisfies
\begin{equation}
\label{eq-weak-regularized}
  a_{\eta}(\mbs{u}_{\eta}, \mbs{v}) = \ell(\mbs{v})\ ,
\end{equation}
for all $\mbs{v}\in[H^1(\Omega)]^d$, where $a_{\eta}:[H^1(\Omega)]^d \times [H^1(\Omega)]^d\to \mathbb{R}$ is
defined
as
\begin{equation}
  \label{eq-a-eta}
  a_{\eta}(\mbs{v},\mbs{w}) \defined
  a(\mbs{v},\mbs{w}) + \eta\, (\mbs{v},\mbs{w})\ ,
\end{equation}
and we have introduced the $L^2$-product
\begin{equation}
\label{eq-L2}
   (\mbs{u},\mbs{v}) \defined  \int_{\Omega} \mbs{u}(\mbs{x})\cdot \mbs{v}(\mbs{x})\; dV\
\end{equation}
between vector fields in $[L^2(\Omega)]^d$. The regularized Neumann problem is clearly similar to
the pure traction problem of interest in this work. Prior to studying the relationship between their
two solutions, we note the following fundamental result, whose proof is postponed to Section~\ref{sec-analysis}

\begin{theorem}
  \label{theo-reg}
  For every $\eta>0$, the functional $V_{\eta}$ has a unique minimizer $\mbs{u}_{\eta}\in [H^1(\Omega)]^d$
  that depends continuously on the loading. Equivalently, the variational problem
  \eqref{eq-weak-regularized} is   well-posed.
\end{theorem}

\begin{remarks}
\ \newline
\begin{enumerate}

\item The idea of studying the solution to the pure traction problem by regularization is not new. This technique appears, for example, in the classical work of Fichera \cite{fichera1972in}.

\item The key aspect of Theorem~\ref{theo-reg} is that by shifting the problem from the pure Neumann case to the regularized formulation, the functional space where solutions are sought is now the  Hilbert space $[H^1(\Omega)]^d$, not a quotient space.  This seemingly minor change has a dramatic impact from the point of view of approximation: whereas there is no simple way to approximate classes in $\mathcal{S}$, the approximation of the regularized problem is trivial.

\item Let us note that Theorem~\ref{theo-reg} says nothing about the dependency of the solution on~$\eta$. In fact,
for general loads, the bound in the solution is not uniformly in $\eta$. We will refine this result
later. In terms of the motivation example of Section~\ref{sec-mot}, this theorem simply generalizes the fact
that the stiffness in Eq.~\eqref{eq-matrix-eta} is regular without claiming anything about its
condition number.

\item Let us note that, unless all quantities appearing in the problem have been normalized, the
regularizing parameter $\eta$ must have dimensions of $F/L^{d+1}$. In numerical examples, it might be
convenient to write
\begin{equation}
\label{eq-char-length}
\eta = \bar{\eta} \frac{\mu}{L^{d-1}}\ ,
\end{equation}
where $L$ is a characteristic length of the body and $\bar{\eta}$ is nondimensional. In Section~\ref{sec-analysis}
we will assume that all the equations are nondimensional. This simplifies the derivation of bounds and estimates.
\end{enumerate}

\end{remarks}

We claimed that the regularized problem~\eqref{eq-weak-regularized} is well-posed, but its relationship to the
original Neumann problem has not been elucidated yet. As we will show in Section~\ref{sec-analysis},
the solutions to the two problems are closely related, and when $\eta\to0$ the regularized solution
converges to a solution of the original problem. In effect, we have proposed a modified problem that
is simple to approximate numerically, and whose solution is as close as desired to the
solution to the pure traction problem.

\section{Analysis}
\label{sec-analysis}
We show in this section that the regularized problem is well-posed and that its solution is
related to the solution to the pure traction problem. Specifically, the limit solution when the regularizing
parameter goes to zero coincides with a particular solution of the Neumann problem. This important
result shows that solving the regularized problem with a small perturbation might be used to solve
the original problem, as closely as needed. For the remainder of this article, we assume that $d=3$. This
choice simplifies the algebra so that we can use the standard vector product in $\mathbb{R}^3$. All the
results can be extended to the two-dimensional case be redefining appropriately the operation of skew
tensors on vectors.

To start, let $\mbs{v}\in [L^2(\Omega)]^d$ be an arbitrary vector field and define
\begin{equation}
  \label{eq-means}
  \langle \mbs{v} \rangle \defined
  |\Omega|^{-1} \int_{\Omega} \mbs{v}(\mbs{x})
  \;dV\ ,
  \qquad
  \left\{ \mbs{v} \right\} \defined
  \mbs{J}^{-1} \int_{\Omega} (\mbs{x}-\mbs{x}_{\Omega})\times \mbs{v}(\mbs{x})\;dV\ ,
\end{equation}
with
\begin{equation}
  \label{eq-j}
  \mbs{x}_{\Omega}  \defined
  |\Omega|^{-1} \int_{\Omega} \mbs{x} \;dV\ ,
  \qquad
  \mbs{J} \defined
  \int_{\Omega} (\mbs{x}-\mbs{x}_{\Omega})\otimes (\mbs{x}-\mbs{x}_{\Omega})
  \;dV\ .
\end{equation}
For convenience, we introduce the following concept:

\begin{definition}
  A vector field $\mbs{v}\in[L^2(\Omega)]^d$ belongs to $\mathcal{C}$, the space of \emph{centered} displacements
  if
  \begin{equation}
    \label{eq-centered}
    \langle \mbs{v} \rangle = \{ \mbs{v} \} = \mbs{0}\ .
  \end{equation}
\end{definition}
This concept is closely related to the definition of the space of infinitesimal rigid body
motions as the following result shows:

\begin{lemma}
  \label{lem-brackets}
  A vector field $\mbs{u}\in [L^2(\Omega)]^d$ is orthogonal to the space of
  infinitesimal rigid body motions if and only if it is centered. Thus, the space $[L^2(\Omega)]^d$ is the orthogonal direct sum
  \begin{equation}
    \label{eq-direct-sum}
    [L^2(\Omega)]^d = \mathcal{C} \oplus \mathcal{R}\ ,
  \end{equation}
  and every $\mbs{u}\in [L^2(\Omega)]^d$ can be uniquely expressed as
  \begin{equation}
    \label{eq-direct-sum2}
    \mbs{u}(\mbs{x}) = \mbs{u}_0(\mbs{x}) + \mbs{u}_{\perp}(\mbs{x})\ ,
  \end{equation}
  with $\mbs{u}_0\in \mathcal{C}$, $\mbs{u}_{\perp}\in \mathcal{R}$, and
  $(\mbs{u}_0,\mbs{u}_\perp)=0$.
\end{lemma}

\begin{proof}
A vector field $\mbs{u}\in [L^2(\Omega)]^d$ is orthogonal to an arbitrary infinitesimal rigid body
motion if
\begin{equation}
  0 = (\mbs{u},\mbs{\alpha} + \mbs{\beta}\;\mbs{x})
\end{equation}
for arbitrary $\mbs{\alpha} + \mbs{\beta}\;\mbs{x}\in \mathcal{R}$. Selecting $\mbs{\beta}\equiv\mbs{0}$,
we get
\begin{equation}
  0 = \int_{\Omega} \mbs{u}(\mbs{x})\cdot \mbs{\alpha}\; dV = \mbs{\alpha}\cdot \int_{\Omega} \mbs{u}(\mbs{x})\; dV
    = \mbs{\alpha} \cdot |\Omega|\,\langle \mbs{u} \rangle\ .
\end{equation}
Since $\mbs{\alpha}$ is arbitrary and $|\Omega|>0$, $\langle \mbs{u} \rangle$ must identically vanish. A similar
argument shows that $\{ \mbs{u} \}$ must also vanish. The reverse implication is obtained by noting
that $0=\mbs{\alpha}\cdot \langle \mbs{u} \rangle + \mbs{\beta}\cdot \mbs{x}\times \{ \mbs{u} \}$
for any $\mbs{\alpha},\mbs{\beta}$.
\end{proof}

\begin{lemma}
  \label{lem-u0}
  Let $\mbs{u}\in [H^1(\Omega)]^d$ be any solution to the Neumann
  problem~\eqref{eq-weak-neumann}. Then, the function $\mbs{u}_{0}\in [H^1(\Omega)]^d$ defined as
  \begin{equation}
  \label{eq-u-split}
  \mbs{u}_0(\mbs{x}) \defined \mbs{u}(\mbs{x}) - \langle \mbs{u} \rangle -
  (\mbs{x}-\mbs{x}_{\Omega})\times \left\{ \mbs{u} \right\}
\end{equation}
belongs to $[\mbs{u}]$ and thus is also a solution to the Neumann problem. Moreover, it is centered and has the minimum $L^2$ norm among all functions in $[\mbs{u}]$.
\end{lemma}
\begin{proof} The function $\mbs{u}_0$ satisfies Eq.~\eqref{eq-centered} and the difference
  $\mbs{u}-\mbs{u}_0= \langle \mbs{u} \rangle + (\mbs{x}-\mbs{x}_{\Omega})\times \left\{
    \mbs{u} \right\}$  is an infinitesimal rigid body motion.
  Hence $\mbs{u}\sim \mbs{u}_0$, and $\mbs{u}_0$ is a
  solution to the Neumann problem. Also, a simple manipulation gives
  \begin{equation}
    \| \mbs{u} \|_{[L^2(\Omega)]^d}^2 =
    \| \mbs{u}_0 \|_{[L^2(\Omega)]^d}^2 +
     |\langle \mbs{u} \rangle|^{2}\, |\Omega|^2 +
    \left\{ \mbs{u} \right\}\cdot \mbs{J} \left\{ \mbs{u}^{*} \right\}
    \ge
  \| \mbs{u}_0 \|_{[L^2(\Omega)]^d}^2\ ,
\end{equation}
since $|\Omega|>0$ and the inertia $\mbs{J}$ is positive definite.
\end{proof}

The key result for proving the well-posedness of the regularized problem is provided by
a Korn-type inequality \cite{duvaut1976ud}:

\begin{lemma}
  \label{lem-KPW}
  Let $\mbs{v}\in [H^1(\Omega)]^d$ be centered. Then,
\begin{equation}
\label{eq-KPW}
\| \nabla^s\mbs{v} \|_{[L^2(\Omega)]^d}
\ge C\;
\| \mbs{v}  \|_{[H^1(\Omega)]^d}
\ ,
\end{equation}
for some constant $C$.
\end{lemma}

\begin{proof}
Let us introduce the map $\triple{\mbs{v}} \defined \|\nabla^s \mbs{v}\|_{[L^2(\Omega)]^d}$.
Korn's second inequality shows that the map $\mbs{v}\mapsto (\triple{\mbs{v}}^2 + \| \mbs{v}
\|^2_{[L^2(\Omega)]^d})^{1/2}$ is a norm on
the space $[H^1(\Omega)]^d$, and equivalent to the standard $H^1-$norm. Hence, to prove
bound~\eqref{eq-KPW}, it suffices to show that
\begin{equation}
  \label{eq-kk-p1}
\triple{\mbs{v}} \ge C\; \| \mbs{v} \|_{[L^2(\Omega)]^d}\ ,
\end{equation}
for all $\mbs{v}\in[H^1(\Omega)]^d$ that are centered. For that, we note first that
centered functions are orthogonal to the space $\mathcal{R}$, which is precisely
the kernel of the operator $\nabla^s$, hence the only centered function that
verifies $\triple{\mbs{v}}=0$ is $\mbs{v}\equiv\mbs{0}$. Then, if \eqref{eq-kk-p1} were not true, there
would be a sequence of centered functions $\{\mbs{v}_n\}_{n=1}^{\infty}$ with
\begin{equation}
  \triple{\mbs{v}_n}\to 0 \ , \qquad
  \|\mbs{v}_n\|_{[L^2(\Omega)]^d}=1\ .
\end{equation}
Again, by Korn's second inequality the sequence $\{\mbs{v}\}_{n=1}^{\infty}\in [H^1(\Omega)]^d$ is
bounded and thus there is a subsequence, still denoted $\{\mbs{v}_n\}_{n=1}^{\infty}$, and a centered
function $\mbs{v}$ such that $\mbs{v}_n \rightharpoonup \mbs{v}$ in $\in [H^1(\Omega)]^d$.
Using the properties of weak limits,
\begin{equation}
  0 = \lim_{n\to\infty} \triple{\mbs{v}_{n}}
  \ge \liminf_{n\to\infty} \triple{\mbs{v}_n} \ge \triple{\mbs{v}} \ ,
\end{equation}
thus, $\triple{\mbs{v}}=0$ and, by the first result of this proof, $\mbs{v}\equiv \mbs{0}$. By Rellich's embedding theorem (see \cite[Prop. 4.4]{taylor2011}) and the fact that finite Cartesian products preserve compact embeddings, $[H^1(\Omega)]^d$ is compactly embedded in $[L^2(\Omega)]^d$ 
and thus $\mbs{v}_n\to\mbs{0}$ in $[L^2(\Omega)]^d$, contradicting the assumption that
$\| \mbs{v}_n\|_{[L^2(\Omega)]^d} = 1$, hence proving the first result of the lemma. With it, the
proof the second bound is immediate. See \cite{acosta2017vu} for further details.
\end{proof}

Lemma~\ref{lem-KPW} is the central result required to prove the stability of the regularized Neumann
formulation. Using it, we can prove the following result that will simplify the analysis

\begin{lemma}
  \label{lem-norms}
  The function $\triple{\cdot}$ is a norm on $[H^1(\Omega)]^d\cap \mathcal{C}$, equivalent to the
  standard norm on $[H^1(\Omega)]^d$. Moreover, the map
  \begin{equation}
    \label{eq-eta-norm}
    \etanorm{\mbs{v}} \defined
    \left( a_{\eta}(\mbs{v},\mbs{v}) \right)^{1/2}
  \end{equation}
  defined over $[H^1(\Omega)]^d$ satisfies
  \begin{equation}
    \label{eq-eta-bound2}
    \etanorm{\mbs{v}} \lesssim \max(1,\sqrt{\eta})\, \| \mbs{v} \|_{[H^1(\Omega)]^d}\ .
  \end{equation}
  Also, if $\mbs{v}\in [H^1(\Omega)]^d\cap \mathcal{C}$, then
  \begin{equation}
    \label{eq-eta-bound1}
    \| \mbs{v} \|_{[H^1(\Omega)]^d} \lesssim \etanorm{\mbs{v}}\ .
  \end{equation}
  Hence, the weighted norm $\etanorm{\cdot}$ is equivalent to the $[H^1(\Omega)]^d-$norm, but not
  uniformly on~$\eta$.
\end{lemma}

The solution of the regularized problem~\eqref{eq-weak-regularized} is centered because,
due to relations~\eqref{eq-equilibrium2} and \eqref{eq-a-orthogonal}, for all $\mbs{\rho}\in\mathcal{R}$ and $\eta>0$,
\begin{equation}
\label{eq-ueta-ortho}
(\mbs{u}_\eta, \mbs{\rho}) =
\frac{1}{\eta}  \left( \ell(\mbs{\rho}) -  a(\mbs{u}_\eta, \mbs{\rho}) \right) = 0\ .
\end{equation}
As a result, we can use bound~\eqref{eq-KPW} to prove Theorem~\ref{theo-reg}:
\begin{proof}[Proof of Theorem~\ref{theo-reg}]
  The solution $\mbs{u}_\eta$ of the variational equation~\eqref{eq-weak-regularized} is centered and thus
  \begin{equation}
  \| \mbs{u}_\eta \|^2_{[H^1(\Omega)]^d} \lesssim
  \triple{\mbs{u}_{\eta}}_{\eta}^2
  \lesssim
  \|\ell\|_{[H^{-1}(\Omega)]^d} \| \mbs{u}_\eta \|_{[H^1(\Omega)]^d}\ ,
  \end{equation}
  hence
  \begin{equation}
  \label{eq-ueta-bound}
  \| \mbs{u}_\eta \|_{[H^1(\Omega)]^d} \lesssim \|\ell\|_{[H^{-1}(\Omega)]^d}\ .
  \end{equation}
\end{proof}
As indicated in the statement of Theorem~\ref{theo-reg}, Eq.~\eqref{eq-ueta-bound}
shows that the regularized solution is uniformly stable in $\eta$. This pivots
on the fact that the loads are equilibrated.

Next, we prove the key result of the continuous problem under consideration. We will show that the solution of the
regularized variational problem is close to one member of the equivalent class of functions that
solve the Neumann problem. More precisely, the regularized solution $\mbs{u}_{\eta}$ is close to
$\mbs{u}_{0}\in [\mbs{u}]$, the solution function of minimal $L^2$ norm, and the distance between
the two solutions goes to zero with the regularization parameter.

\begin{theorem}\label{thm-main}
  Let $[\mbs{u}]\in \mathcal{S}$ be the solution to the Neumann problem~\eqref{eq-weak-neumann}
  and $\mbs{u}_{\eta}\in [H^1(\Omega)]^d$ be the solution
  to the regularized Neumann problem~\eqref{eq-weak-regularized}. Then, there exists a constant $C$
  independent of $\eta$ such that
  \begin{equation}
  \label{eq-th1}
  \| \mbs{u}_{\eta} - \mbs{u}_0 \|_{[H^1(\Omega)]^d}
  \lesssim
  \eta\; \| \ell \|_{[H^{-1}(\Omega)]^d}\ ,
\end{equation}
where $\mbs{u}_0$ is defined in Eq.~\eqref{eq-u-split}.
\end{theorem}
\begin{proof}
  Let $\mbs{u}_0$ be the displacement field defined in
  Lemma~\ref{lem-u0}. Subtracting Eq.~\eqref{eq-weak-neumann} from~\eqref{eq-weak-regularized}, and
  setting $\mbs{v} = \mbs{u}_{\eta}-\mbs{u}_0$ we get
  \begin{equation}
  a(\mbs{u}_{\eta} - \mbs{u}_0,\mbs{u}_{\eta}-\mbs{u}_0) +
  \eta (\mbs{u}_{\eta},\mbs{u}_{\eta}-\mbs{u}_0) = 0 \ .
\end{equation}
Since both $\mbs{u}_{\eta}$ and $\mbs{u}_0$ are centered, their difference is
too. Therefore, we can apply Lemma~\ref{lem-KPW} to obtain
\begin{equation}
  \| \mbs{u}_{\eta} - \mbs{u}_0 \|_{[H^1(\Omega)]^d}^2
  \lesssim
  \triple{\mbs{u}_{\eta}-\mbs{u}_0}^2
  \lesssim
  \eta\;\| \mbs{u}_{\eta} \|_{[L^2(\Omega)]^d} \;
  \| \mbs{u}_{\eta} - \mbs{u}_0 \|_{[L^2(\Omega)]^d}\ .
\end{equation}
We can simplify this equation and get
\begin{equation}
\label{eq-mt-p1}
\| \mbs{u}_{\eta} - \mbs{u}_0 \|_{[H^1(\Omega)]^d}
\lesssim
\eta\;
  \| \mbs{u}_{\eta} \|_{[L^2(\Omega)]^d}\ .
\end{equation}
Using bound~\eqref{eq-ueta-bound}, we conclude the proof.
\end{proof}

Theorem~\ref{thm-main} shows that, instead of solving the Neumann problem, whose solution is a
class of rigid-body equivalent functions in $H^1$, one can solve the perturbed problem --- whose solution is now a function in $H^1$. By replacing the original problem with the perturbed one, an error is made in the solution which is of the order of the perturbation parameter. Remarkably,
the regularized problem is well-posed, independently of $\eta$. Since the perturbed problem
is posed on a Hilbert space, it is straightforward to formulate a finite element approximation of the latter, as we will now see.

It bears emphasis that bound~\eqref{eq-th1} depends on the fact that
the external loading satisfies the compatibility conditions~\eqref{eq-equilibrium}, as we have shown
in the proof. Without this property, the solutions $\mbs{u}$ and $\mbs{u}_{\eta}$ need not be close
and the approach would fail.

\section{Finite element approximation}
\label{sec-fem}
Once we have shown that the regularized problem~\eqref{eq-weak-regularized} is well-posed in
$[H^1(\Omega)]^d$, we analyze its finite element approximation. Using
Céa's lemma \cite{ciarlet1978tm} it is straightforward to show that this Galerkin method
approximates the regularized solution $\mbs{u}_{\eta}$ optimally and inherits the stability of the
continuous problem. The result that closes the article shows that the finite element solution
converges to the solution of the Neumann problem with minimum $L^2$ norm when both the mesh size and the
regularizing parameter go to zero.

To show this, we start by introducing a finite element approximation of the pure traction problem. For that, let
$\mathcal{T}=\left\{ \Omega_e,\; e\in \mathcal{E} \right\}$ be a regular partition of the domain $\Omega$ into finite
elements $\Omega_e$ and define a finite element space of functions
\begin{equation}
  \label{eq-fem-space}
  \mathcal{V}^h =
  \left\{
    \mbs{u}^h\in H^1(\Omega),\
    \left. \mbs{u}^h\right|_{\Omega_e} \in R_k
  \right\}\ ,
\end{equation}
with $R_k$ being a space of polynomials or affine transformations of tensor-product polynomials of
order $k$ defined on a reference element $\hat{K}$. The finite element solution of the regularized problem can
be obtained as
\begin{equation}
\label{eq-fem}
\mbs{u}^h_{\eta} = \arg\min_{\mbs{w}^h\in \mathcal{V}^h} V_{\eta}(\mbs{w}^h)\ ,
\end{equation}
but also as the unique $\mbs{u}^h_{\eta}\in \mathcal{V}^h$ such that
\begin{align}
\label{eq-fem2}
a(\mbs{u}^h_{\eta},\mbs{w}^h) + \eta\; (\mbs{u}^h,\mbs{w}^h) = \ell(\mbs{w}^h)\ ,
\end{align}
for all $\mbs{w}^h\in \mathcal{V}^h$.
The space of finite element functions contains the infinitesimal rigid body
motions and, just as in Eq.~\eqref{eq-ueta-ortho}, we have that
\begin{equation}
 (\mbs{u}^h_\eta, \mbs{\rho}^h) =
 \frac{1}{\eta} \left( \ell(\mbs{\rho}^h) - a(\mbs{u}^h_\eta , \mbs{\rho}^h) \right) = 0\ .
\end{equation}
Then, we can replicate the proof of Theorem~\ref{theo-reg} to conclude that the finite
element regularized problem is well-posed, uniformly in $\eta$, i. e.,
\begin{equation}
\label{eq-fem-basic}
\| \mbs{u}^h_{\eta} \|_{[H^1(\Omega)]^d}
\lesssim
\| \ell \|_{[H^{-1}(\Omega)]^d} .
\end{equation}

The main result of this work is provided by the next theorem, which guarantees the convergence of the
proposed finite element discretization to \emph{a} solution of the traction problem, precisely the
one of minimal $L^2$ norm:

\begin{theorem}
  \label{theo-fem}
  Let $[\mbs{u}]$ be the solution to the Neumann problem~\eqref{eq-weak-neumann} and $\mbs{u}^h_{\eta}$ be the
  solution to~\eqref{eq-fem}. Assume, furthermore, that the solution $\mbs{u}_{\eta}$ satisfies the
  regularity estimate
  \begin{equation}
  \label{eq-regularity}
  | \mbs{u}_\eta|_{[H^{k+1}(\Omega)]^d}
  \lesssim
  \| \ell \|_{[H^{-1}(\Omega)]^d}
  \end{equation}
  for some constant $C$. Then, the following error estimate holds
  \begin{equation}
  \label{eq-end}
  \| \mbs{u}_0 - \mbs{u}_{\eta}^h \|_{[H^1(\Omega)]^d}
  \lesssim
  (\eta + h^k) \|l\|_{[H^{-1}(\Omega)]^d}\ ,
\end{equation}
where $\mbs{u}_0$ is the member of $[\mbs{u}]$ with minimal~$L^2$ norm.
\end{theorem}

\begin{proof}
To show this, we start by noticing that the triangle inequality gives
\begin{equation}
\| \mbs{u}_0 - \mbs{u}_{\eta}^h \|_{[H^1(\Omega)]^d}
\le
\| \mbs{u}_0 - \mbs{u}_{\eta} \|_{[H^1(\Omega)]^d}
+
\| \mbs{u}_\eta - \mbs{u}_{\eta}^h \|_{[H^1(\Omega)]^d}\,.
\end{equation}
The first term on the right-hand side can be bounded using Eq.~\eqref{eq-th1}. The standard
finite element \emph{a priori} estimate is
\begin{equation}
\label{eq-fem-estimate}
\| \mbs{u}_\eta  - \mbs{u}^h_\eta \|_{[H^1(\Omega)]^d}
\lesssim
h^k\; | \mbs{u}_{\eta} |_{[H^{k+1}(\Omega)]^d}\ .
\end{equation}
Combining Eqs.~\eqref{eq-regularity} and~\eqref{eq-fem-estimate}, the theorem follows.
\end{proof}

This result shows that for any fixed value of the relaxation parameter $\eta$, the finite element
method converges optimally to the solution of the relaxed problem. In addition,
the finite element discretization preserves the convergence of the regularized solution to the exact solution of
the traction problem and choosing $h$ and $\eta$ as small as desired, the error between the
finite element solution and the one sought goes to zero. Note that if we wish that
the finite element discretization converges to the Neumann solution $\mbs{u}_0$ when the mesh
is refined, we must insist that
\begin{equation}
    \lim_{\eta,h\to0} \frac{\eta}{h^k} = 0\ .
  \end{equation}

\subsection{An iterative solution}
\label{subs-iterative}

From the algebraic point of view, the regularized solution proposed in this section has the effect
of replacing the singular stiffness of the pure Neumann formulation with a regular one, thus turning
the Neumann problem solvable with a unique solution. The
condition number of stiffness matrix in the regularized problem, however, deteriorates with
$\eta\to0$. This is a concern that should be taken into consideration since floating point
computations are of limited precision and badly-conditioned matrices can lead to erroneous solutions
in practice.

This aspect has been carefully analyzed by Dai \cite{dai2007ol}, where a regularized method
was proposed that is similar to the one introduced in Section~\ref{sec-regularized}. Adapting the
ideas of that work to our method, we can propose an iterative solution that can avoid the
ill-conditioned problems of regularized solutions.

\begin{theorem}
  Consider a pure Neumann problem and the sequence $\{\mbs{u}_h^k\}_{k=0}^{\infty}$ of finite element
  solutions obtained by solving
  \begin{equation}
    \label{eq-iterative}
    \begin{split}
      a_{\eta}(\mbs{u}_h^0,\mbs{v}_h) &= \ell(\mbs{v}_h)\ , \\
      a_{\eta}(\mbs{u}_h^{k+1},\mbs{v}_h) &= \ell(\mbs{v}_h) + \eta(\mbs{u}_h^k,\mbs{v}_h)\ ,
   \end{split}
 \end{equation}
 for all $\mbs{v}_h\in \mathcal{V}^h$. Let $\alpha$ denote the coercivity constant of the bilinear form $a_{\eta}(\cdot,\cdot)$. Then the iterative method~\eqref{eq-iterative} converges when $k\to\infty$ 
 to the solution $\mbs{u}_h$ of the discretized Neumann problem for all $0<\eta < 1/\alpha$.
\end{theorem}
\begin{proof}
  Consider the solution of two consecutive iterative steps $k,k+1$ with $k>0$. Then, by subtracting
  the equations of the method it follows that
  \begin{equation}
   a_{\eta}(\mbs{u}_h^{k+2}-\mbs{u}_h^{k+1}, \mbs{v}_h) = \eta\, (\mbs{u}_h^{k+1}-\mbs{u}_h^k,\mbs{v}_h)\ .
  \end{equation}
  Then, by the coercivity of the bilinear form $a_{\eta}(\cdot,\cdot)$ we have that
  \begin{equation}
    \begin{split}
    \| \mbs{u}_h^{k+2}-\mbs{u}_h^{k+1}\|_{[H^1(\Omega)]^d} &\le
                                                 \alpha\; \etanorm{\mbs{u}_h^{k+2}-\mbs{u}_h^{k+1}}^2 \\
      &= \alpha\; a_{\eta}(\mbs{u}_h^{k+2}-\mbs{u}_h^{k+1}, \mbs{u}_h^{k+2}-\mbs{u}_h^{k+1})\\
      &=
        \alpha\; \eta\; (\mbs{u}_h^{k+1}-\mbs{u}_h^{k}, \mbs{u}_h^{k+2}-\mbs{u}_h^{k+1}) \\
      &\le
    \alpha; \eta\; \|\mbs{u}_h^{k+1}-\mbs{u}_h^{k}\|_{[H^1(\Omega)]^d}\; \|\mbs{u}_h^{k+2}-\mbs{u}_h^{k+1}\|_{[H^1(\Omega)]^d}\ .
    \end{split}
  \end{equation}
  Simplifying this relation and selecting $\eta < 1/\alpha$, we have that
  \begin{equation}
    \| \mbs{u}_h^{k+2} - \mbs{u}_h^{k+1}\|_{[H^1(\Omega)]^d}
    <
   \| \mbs{u}_h^{k+1} - \mbs{u}_h^{k}\|_{[H^1(\Omega)]^d}\ ,
  \end{equation}
  proving that the map $\mbs{u}_h^k\mapsto \mbs{u}_h^{k+1}$ is contractive. By Banach's fixed point theorem,
  there exists a unique $\mbs{u}_h$ such that $\lim_{k\to\infty} \|\mbs{u}_h- \mbs{u}_h^k\|_{[H^1(\Omega)]^d}=0$.
\end{proof}

Notice that the iterative method~\eqref{eq-iterative} can be alternatively be written as:
\begin{equation}
    \label{eq-iterative2}
    \begin{split}
      a_{\eta}(\mbs{u}_h^0,\mbs{v}_h) &= \ell(\mbs{v}_h)\ , \\
      (\frac{\mbs{u}_h^{k+1}-\mbs{u}_h^k}{\eta^{-1}},\mbs{v}_h)
      +
      a(\mbs{u}_h^{k+1},\mbs{v}_h) 
      &= \ell(\mbs{v}_h)\ ,
   \end{split}
 \end{equation}
 which is the backward-Euler discretization of a parabolic problem. Thus, the iterative method can be
 interpreted as a discrete evolution that, due to dissipation, converges to the equilibrated Neumann
 solution.

\section{Non-equilibrated loads}
\label{sec-noneq}
Sections~\ref{sec-analysis} and \ref{sec-fem} have addressed the pure Neumann problem, where a
solution or an approximation of it is sought to a problem with perfectly equilibrated loading. In
practice, it often happens that the spatial discretization of the solution domain introduces errors
that spoil the load compatibility. Strictly speaking, these problems with (slightly) non-equilibrated loading are not pure traction problems and therefore, they do not have a solution. However, it may turn out that the only way to
obtain an approximate solution to a pure traction problem is through a discretization. We will study
next how the method of Section~\ref{sec-regularized} can be used to ensure convergent approximations
to Neumann problems, even if the discretization upsets the compatibility of the loading.

A first, a very common, situation where the discretization spoils the compatibility of a Neumann
problem is due simply by the use of a mesh. By doing so, the domain $\Omega$ is often replaced
by a polyhedral domain. There might be situations where a statically equilibrated loading fails to
be so when the exact domain is replaced with an approximated one. An ever more subtle side effect of
the tesselation of the solution domain is that, in general, meshes introduce a certain anisotropy in
the approximation spaces. Although such an undesirable artifact vanishes when the mesh size tends to zero,
it nevertheless might spoil the compatibility of the loading.

A second illustrative example where these difficulties might arise is related to finding
displacements of a deformable body that is rotating with uniform angular velocity. These
\emph{relative equilibria} are paradigmatic examples of Neumann problems
\cite{simo1991ti,armero2001vc,romero2001th}. To calculate all the relevant quantities, it suffices to
impose onto the body a field of central forces with its origin at its center of mass
and modulus, at each point, proportional to the distance to the origin. However, a slight
misrepresentation of the coordinates of this point (be it due to a calculation error or just
truncation) will produce a load that is no longer compatible yielding, hence, an ill-posed problem.

Given its practical importance, we examine next the impact that ``small'' load incompatibilities
produce in the numerical solution of pure traction problems. Before that, let us point out that the answer
depends on the method employed for its solution. For example, if the Neumann problem is solved using
carefully selected boundary conditions to preclude rigid body motions, then small incompatibilities
would only induce some reaction forces on those constraints, which are supposed to vanish exactly
for equilibrated loading. In the example of the relative equilibrium, however, we note that all the
displacements of the center of mass of the solid should be constrained, a condition that is only
easily implemented if this point coincides with a node, which most likely would not occur for bodies
with complex geometries.

\subsection{Perturbed Neumann problems}

Consider again the finite element solution to a Neumann problem. The starting point of the discretization is the partition of the domain $\Omega$ onto a finite collection of disjoint elements
$\Omega_e$ such that $\Omega_h\defined \cup_e  \Omega_e \approx \Omega$. This error in the
geometrical representation of the domain depends on the mesh size and we thus assume that
\begin{equation}
  \label{eq-omega-error}
  |\Omega - \Omega_h| \lesssim h^d\ .
\end{equation}
As a result of this misrepresentation, instead of the variational problem~\eqref{eq-weak-neumann},
the discretized solution becomes the function $\mbs{u}_h\in \mathcal{V}_h$ that satisfies the perturbed variational problem
\begin{equation}
  \label{eq-discrete-alh}
  a^h(\mbs{u}_h,\mbs{w}_h) = \ell^h(\mbs{w}_h)\ ,
\end{equation}
for all $\mbs{w}_h\in \mathcal{V}_h$ and
\begin{subequations}
  \label{eq-formsh}
  \begin{align}
a^h(\mbs{u}, \mbs{v})
  &\defined
\int_{\Omega_h} (\mu \nabla^s \mbs{u}(\mbs{x}): \nabla^s \mbs{v}(\mbs{x}) +
    \lambda \nabla\cdot \mbs{u}(\mbs{x})\;\nabla\cdot \mbs{v}(\mbs{x})) \; dV\ ,
    \label{eq-formsh1}\\
\ell^h(\mbs{v}) &\defined
\int_{\Omega_h} \mbs{f}(\mbs{x})\cdot \mbs{v}(\mbs{x}) \; dV
+
                \int_{\partial\Omega_h} \mbs{t}(\mbs{x})\cdot \mbs{v}(\mbs{x}) \; dA\; .
                \label{eq-formsh2}
  \end{align}
\end{subequations}
This approximation of the boundary value problems is known as a \emph{variational crime} and its
effects were already studied by Strang decades ago \cite{strang1973vc,ciarlet1978tm}. In the Neumann
problem, however, it is not enough to use this classical results because the convergence results of
Section~\ref{sec-analysis} depend on the bilinear form $a_{\eta}(\cdot,\cdot)$ vanishing on $\mathcal{R}$ and the linear
form $\ell(\cdot)$ being equilibrated. These two critical conditions need to be revisited in light of the
differences between Eqs.~\eqref{eq-weak-neumann} and~\eqref{eq-discrete-alh}.

We first characterize the class of incompatible or non-equilibrated loads.
\begin{definition}
  \label{def-ell-split}
  Let $\ell$ be a load, a continuous functional on the dual space of $[H^1(\Omega)]^d$ which we
  identify with $[H^{-1}(\Omega)]^d$. This load admits the unique decomposition
  \begin{equation}
    \label{eq-dual-load}
    \ell = \ell_0 + \ell_{\perp}
  \end{equation}
  where, for all $\mbs{w}\in [H^1(\Omega)]^d$ we define

  \begin{equation}
    \label{eq-load-split}
    \ell_0(\mbs{w}) \defined \ell(\mbs{w}_0)\ ,
    \qquad
    \ell_{\perp}(\mbs{w}) \defined \ell(\mbs{w}_{\perp})\ ,
  \end{equation}
  where $\mbs{w}_0$ and $\mbs{w}_{\perp}$ are as in Eq.~\eqref{eq-direct-sum2}, with the
  integrals in the definition of centered fields evaluated over $\Omega_h$.

  Since $\ell_0(\mbs{\rho})=0$ for all $\mbs{\rho}\in \mathcal{R}$, we refer to $\ell_0$ as the
  compatible or equilibrated component of $\ell$; respectively, $\ell_{\perp}$ is the incompatible
  or non-equilibrated part of the load.
  \end{definition}

As explained, when discretizing a Neumann problem, it might happen that the loading functional
$\ell$ is replaced by an approximated loading $\ell^h$ that is not equilibrated, invalidating the convergence
analysis of Section~\ref{sec-fem}. Likewise, the bilinear form $a_{\eta}(\cdot,\cdot)$ of the
continuous problem might be replaced with an approximate bilinear form $a_{\eta}^h(\cdot,\cdot)$
such that
\begin{equation}
  \label{eq-approx-aheta}
  a_{\eta}^h(\mbs{v},\mbs{w}) \defined a^h(\mbs{v},\mbs{w}) + \eta\,(\mbs{v},\mbs{w})^h\ ,
\end{equation}
where $a^h(\cdot,\cdot)$ and $\ell^h(\cdot)$. are defined in Eq.~\eqref{eq-formsh} and
\begin{equation}
  \label{eq-approx-l2}
  (\mbs{v}, \mbs{w})^h
  \defined
  \int_{\Omega^h} \mbs{v}(\mbs{x})\cdot \mbs{w}(\mbs{x}) \; \mathrm{d} V
\end{equation}
is the $L^2$ inner product of functions defined on $\Omega^h$.

Under these premises, we hereafter consider the following problem: find $\mbs{u}_{\eta}\in [H^1(\Omega)]^d$ such that
\begin{equation}
  \label{eq-discrete-h}
  a_{\eta}^h(\mbs{u}^{h}_{\eta}, \mbs{w}^h) = \ell^h(\mbs{w}^h)\ ,
\end{equation}
for all $\mbs{w} \in [H^1(\Omega)]^d$. In what follows, we will only consider symmetric bilinear forms $a^h(\cdot,\cdot)$ that verify
\begin{equation}
  \label{eq-bilin-approx-ker}
  a^h(\mbs{w},\mbs{\rho}) = a^h(\mbs{\rho},\mbs{w}) = 0\ ,
\end{equation}
for all $\mbs{w}\in [H^1(\Omega)]^d$ and $\mbs{\rho}\in \mathcal{R}$. This condition is always
satisfied for bilinear forms defined as in Eq.~\eqref{eq-formsh1} since the symmetric gradient of
functions on $\mathcal{R}$ vanish identically.

For all $\mbs{w}\in [H^1(\Omega)]^d$, let us define
\begin{equation}
  \label{eq-discrete-norms}
  \triple{\mbs{w}}^h \defined \left( a^h(\mbs{w}, \mbs{w}) \right)^{1/2}\ ,
  \qquad
  \etanorm{\mbs{w}}^h \defined \left( a^h_{\eta}(\mbs{w},\mbs{w}) \right)^{1/2}\ ,
\end{equation}
and assume that, extending Lemma~\ref{lem-norms}, the function $\triple{\cdot}^h$ is a norm on
$[H^1(\Omega)]^d\cap \mathcal{C}$. Also, for $\mbs{w}\in [H^1(\Omega)]^d$
\begin{equation}
  \label{eq-hhh}
  \triple{\mbs{w}}_{\eta}^h \lesssim \max(1,\sqrt{\eta})\; \| \mbs{w} \|_{[H^1(\Omega)]^d}\ ,
\end{equation}
whereas if $\mbs{w}\in [H^1(\Omega)]^d\cap \mathcal{C}$, then
\begin{equation}
  \label{eq-hhhi}
  \| \mbs{w} \|_{[H^1(\Omega)]^d}
  \lesssim
  \triple{\mbs{w}}_{\eta}^h\ .
\end{equation}

\subsection{The continuous regularized Neumann problem with non-equilibrated loading}

We study first the effects of non-equilibrated loading on a continuous solution of the regularized
problem defined on the approximate domain. Note that a similar analysis can not be performed to the
standard Neumann problem since, since the load is not equilibrated, it has no solution.

\begin{lemma}
\label{lem-prob-split}
Let $\ell\in [H^{-1}(\Omega)]$ be a loading with equilibrated and non-equilibrated components as in Eq.~\eqref{eq-load-split}. Then, the centered and rigid body components of the solution $\mbs{u}_\eta$ to the regularized problem~\eqref{eq-discrete-h} are the solutions to
\begin{equation}
\begin{split}
  a_{\eta}(\mbs{u}_{\eta,0}, \mbs{w}) &= \ell_0(\mbs{w})\ , \\
  a_{\eta}(\mbs{u}_{\eta,\perp}, \mbs{w})  &= \ell_\perp(\mbs{w})\ , \\
\end{split}
\end{equation}
for all $\mbs{w}\in [H^1(\Omega)]^d$. The same result holds for the Galerkin approximation~\eqref{eq-fem2}.
\end{lemma}

\begin{proof}
  By linearity, the solution to the regularized Neumann problem is the displacement
  $\mbs{u}_{\eta}=\mbs{u}_1+\mbs{u}_2$, with
\begin{subequations}
  \begin{align}
  a_{\eta}(\mbs{u}_{1}, \mbs{w}) &= \ell_0(\mbs{w})\ , \label{eq-pfs-1}\\
  a_{\eta}(\mbs{u}_{2}, \mbs{w}) &= \ell_\perp(\mbs{w})\ . \label{eq-pfs-2}
\end{align}
\end{subequations}
Choosing $\mbs{w}\in \mathcal{R}$ in Eq.~\eqref{eq-pfs-1} we show that $\mbs{u}_1$ is centered.
Next, we select $\mbs{w}$ in Eq.~\eqref{eq-pfs-2} to be $\mbs{u}_{2,0}$, the centered part of
$\mbs{u}_2$. Using Lemma~\ref{lem-KPW} and Eq.~\eqref{eq-pfs-2}, we have that
\begin{equation}
  \| \mbs{u}_{2,0} \|^2_{[H^1(\Omega)]^d} \lesssim
  a_{\eta}(\mbs{u}_{2,0},\mbs{u}_{2,0})
  = \ell_{\perp}(\mbs{u}_{2,0})
  = 0\ .
\end{equation}
Hence the centered component of $\mbs{u}_2$ vanishes and thus $\mbs{u}_2\in \mathcal{R}$. Since the
split~\eqref{eq-u-split} is unique, we conclude that
\begin{equation}
  \mbs{u}_1 = \mbs{u}_{\eta,0}\ ,
  \qquad
  \mbs{u}_2 = \mbs{u}_{\eta,\perp}\ ,
\end{equation}
as claimed in the Lemma.
\end{proof}

For the method advocated in the current work, the next result summarizes its performance for
non-compatible loading.

\begin{theorem}
  \label{theo-neq}
  Consider the finite element solution to a pure traction problem with non-compatible external loading~$\ell^h$. Then, the solution $\mbs{u}^h_{\eta}$ to the regularized formulation~\eqref{eq-discrete-h} can be uniquely split as
  \begin{equation}
  \label{eq-neq-usplit}
  \mbs{u}_{\eta}^h =
  \mbs{u}^h_{\eta,0}
  +
  \mbs{u}^h_{\eta,\perp}\,,
\end{equation}
where $\mbs{u}^h_{\eta,0},\ \mbs{u}^h_{\eta,\perp} $ are the displacement fields due
to, respectively, the compatible and incompatible parts of the loads. Moreover,
\begin{subequations}
  \label{eq-neq-bounds}
  \begin{align}
    \| \mbs{u}^h_{\eta,0}\|_{[H^1(\Omega)]^d} &\lesssim \| \ell^h_0\|_{[H^{-1}(\Omega)]^d}\ ,
                                                \label{eq-neq-b1}
    \\
    \| \mbs{u}^h_{\eta,\perp}\|_{[H^1(\Omega)]^d} &\lesssim
                                                    {\max(1,\eta^{-1})}\;
     \| \ell^h_\perp\|_{[H^{-1}(\Omega)]^d}\ ,
    \label{eq-neq-b2}
  \end{align}
\end{subequations}
hence,
\begin{equation}
  \label{eq-neq-bg}
    \| \mbs{u}^h_{\eta}\|_{[H^1(\Omega)]^d} \lesssim
                                    \max(1,\eta^{-1}) \| \ell^h\|_{[H^{-1}(\Omega)]^d} .
\end{equation}
\end{theorem}

\begin{proof}
  Lemma~\ref{lem-prob-split} proves that the split~\eqref{eq-neq-usplit} corresponds to the sum of
  the solutions for the equilibrated and non-equilibrated loading.

We can use the results of Section~\ref{sec-fem} to bound the norm of the centered part of the solution. In particular, Lemma~\ref{lem-KPW} applies and we obtain
\begin{equation}
  \begin{split}
  \|  \mbs{u}^h_{\eta,0} \|^2_{[H^1(\Omega)]^d}
  &\lesssim
    a(\mbs{u}^h_{\eta,0}, \mbs{u}^h_{\eta,0})
    \\
  &\lesssim
    a_{\eta}(\mbs{u}^h_{\eta,0}, \mbs{u}^h_{\eta,0})  \\
  &=
    \ell^h_0(\mbs{u}^h_{\eta,0})
    \\
  &\lesssim \|\mbs{u}^h_{\eta,0}\|_{[H^1(\Omega)]^d}\;
    \| \ell^h_0 \|_{[H^{-1}(\Omega)]^d}\ .
    \end{split}
\end{equation}
This is precisely the bound~\eqref{eq-neq-b1}. As for
$\mbs{u}^h_{\eta,\perp}$, since the loading $\ell_\perp$ is not compatible, this displacement field need not be centered and the results of Section~\ref{sec-analysis} can not be employed. However, using Korn's inequality we obtain
\begin{equation}
  \begin{split}
    \|  \mbs{u}^h_{\eta,\perp} \|^2_{[H^1(\Omega)]^d}
    &\lesssim \left( \; \triple{\mbs{u}^h_{\eta,\perp}}^2 + \|\mbs{u}^h_{\eta,\perp}\|^2_{[L^2(\Omega)]^d} \right)
    \\
    &= \frac{1}{\min(1,\eta)} \;
    \left( \; \triple{\mbs{u}^h_{\eta,\perp}}^2 + \eta\, \|\mbs{u}^h_{\eta,\perp}\|^2_{[L^2(\Omega)]^d} \right)
    \\
    &=
    {\max(1,\eta^{-1})} \;
    a_\eta(\mbs{u}^h_{\eta,\perp},\mbs{u}^h_{\eta,\perp})
    \\
    &\lesssim {\max(1,\eta^{-1})} \;
    \|  \mbs{u}^h_{\eta,\perp} \|_{[H^1(\Omega)]^d}
    \| \ell^h_\perp \|_{[H^{-1}(\Omega)]^d}\ ,
  \end{split}
\end{equation}
and bound~\eqref{eq-neq-b2} follows. The global bound~\eqref{eq-neq-bg} is obtained combining Eqs.~\eqref{eq-neq-b1} and~\eqref{eq-neq-b2}.
\end{proof}

This theorem shows that, while equilibrated loading produces displacements whose norms are uniformly
bounded in $\eta$, incompatible loads contribute to the displacement solution with a component that
can only be bounded by the size of the non-equilibrated loading, and it is inversely proportional to the regularizing parameter.

The practical consequence of Theorems~\ref{theo-fem} and~\ref{theo-neq} is that the regularized
method approximates the minimum norm solution of the pure traction problem when $h,\eta\to0$, as long as
the loading is exactly equilibrated. When the loading is slightly incompatible --- a case for which
the Neumann problem has no solution --- the regularized formulation still possesses
a solution. In this situation, however, the solution would not be centered
and the non-centered part might be large relative to the centered one if the regularizing parameter is small.

The traction problem with non-compatible loads splits into two independent
problems. In the first one, which is uniformly well-posed, compatible loads produce centered displacements; in the second one, incompatible loads give rise to non-centered
displacements. The second problem has a unique solution that depends continuously on the loading,
but not uniformly on $\eta$. Even if the non-equilibrated component of the loading goes to zero as
the mesh is refined, the method would be stable only if:
\begin{equation}
  \label{eq-ne-stable}
  \lim_{\eta,h\to0} \frac{\| \ell^h_{\perp} \|_{[H^{-1}(\Omega)]^d}}{\eta}\to 0\ .
\end{equation}
This is a rather stringent constraint on the regularizing parameter $\eta$.

To further analyze this situation, consider now the algebraic form of the finite element problem, namely,
\begin{equation}
\mbsf{K}\, \mbsf{U} = \mbsf{F}\ ,
\end{equation}
where $\mbsf{K}$ is the stiffness matrix, and $\mbsf{U}$, $\mbsf{F}$ the vectors of displacements and external forces, respectively. Theorem~\ref{theo-neq} shows that the displacement, forces, and stiffness matrix can be decomposed as in
\begin{equation}
\begin{bmatrix}
\mbsf{K}_{00} & \mbsf{0} \\
\mbsf{0} & \mbsf{K}_{\perp\perp} \\
\end{bmatrix}
\,
\begin{Bmatrix}
\mbsf{U}_{0}  \\
\mbsf{U}_{\perp}\\
\end{Bmatrix}
=
\begin{Bmatrix}
\mbsf{F}_{0}  \\
\mbsf{F}_{\perp}\\
\end{Bmatrix}\ ,
\end{equation}
where, following the same notation as before, $\mbsf{U}_0,\mbsf{F}_0$ correspond to the centered displacements and equilibrated forces, and $\mbsf{U}_\perp,\mbsf{F}_\perp$ to the remaining parts.

The condition number of $\mbsf{K}_{00}$, denoted as $\kappa(\mbsf{K}_{00})$, is similar to the one in a (regular) finite element discretization. However,
the condition number of $\mbsf{K}_{\perp\perp}$ will depend on $\eta^{-1}$. For a regular mesh of element size~$h$ we will expect \cite{ern2004wx}:
\begin{equation}
\kappa(\mbsf{K}_{00})\sim h^{-2}\ ,\qquad
\kappa(\mbsf{K}_{\perp\perp})\sim h^{-2}\,\eta^{-1}\ .
\end{equation}
If the non-centered solution $\mbsf{U}_{\perp}$ should vanish when the mesh is refined, then
\begin{equation}
  \| \mbsf{U}_{\perp} \| \leq \| \mbsf{K}_{\perp\perp} \|\; \|\mbsf{F}_{\perp} \|
  \leq h^{-2} \eta^{-1}\; \|\mbsf{F}_{\perp} \|
\end{equation}

\subsection{A projection solution for non-equilibrated loading}
\label{subs-two-step}
We explore a numerical strategy that can be used to obtain the centered part
of the solution in a pure Neumann problem with incompatible loads. An obvious way to
achieve this goal would be to project the loading onto its equilibrated
part and use the latter to solve for the centered displacement. This approach is impractical from
the numerical point of view since it requires finding a basis of the space~$\mathcal{R}$.

Instead, we propose a predictor-corrector scheme that consists in finding $\mbs{u}^{h}_p, \mbs{u}^h_\eta\in \mathcal{V}^h$ that satisfy
\begin{subequations}
  \label{eq-two-step}
  \begin{align}
    a_{\eta}^h(\mbs{u}^h_p, \mbs{w}^h) &= \ell^h(\mbs{w}^h)\ ,
    \label{eq-two-step1}
    \\
    a_{\eta}^h(\mbs{u}^h_\eta, \mbs{w}^h) &= \ell^h(\mbs{w}^h)
    - \eta\, (\mbs{u}_p^h,\mbs{w}^h)^h\ ,
    \label{eq-two-step2}
  \end{align}
\end{subequations}
for all $\mbs{w}^h\in\mathcal{V}^h$. The following result shows that indeed the function $\mbs{u}^h_\eta$ is the sought displacement field.

\begin{theorem}
Let $\ell^h:\mathcal{V}^h\to\mathbb{R}$ be a continuous linear form, not necessarily equilibrated
and $a^h_{\eta}:\mathcal{V}^h\times :\mathcal{V}^h\to \mathbb{R}$ is the symmetric bilinear form
defined in Eq.~\eqref{eq-discrete-h} that satisfies
Eqs.~\eqref{eq-bilin-approx-ker}-\eqref{eq-hhhi}.
If $\mbs{u}^h_p$ and $\mbs{u}^h_\eta$ solve Eqs.~\eqref{eq-two-step1} and~\eqref{eq-two-step2}, respectively, then  $\mbs{u}^h_\eta$ is centered and satisfies
\begin{equation}
  \label{eq-th-74}
  \| \mbs{u}_0 - \mbs{u}_\eta^h \|_{[H^1(\Omega)]^d}
  \lesssim
  h \; |\mbs{u}|_{[H^{k+1}(\Omega)]^d} +
  \sup_{\mbs{w}^h\in \mathcal{V}^h} \frac{|a_{\eta}^h(\mbs{u},\mbs{w}^h)-\ell^h(\mbs{w}^h)|}{\| \mbs{w}^h \|_{[H^1(\Omega)]^d}}
  +\eta\; \|\ell^h\|_{[H^{-1}(\Omega)]^d}\ .
\end{equation}
Thus, the predictor-corrector method~\eqref{eq-two-step} is convergent in the sense that
\begin{equation}
  \label{eq-pc-convergence}
  \lim_{h,\eta\to0} \| \mbs{u}_0 - \mbs{u}_\eta^h \|_{[H^1(\Omega)]^d} = 0\ .
\end{equation}
\end{theorem}
\begin{proof}
The fact that $\mbs{u}_\eta^h$ is centered follows from Eq.~\eqref{eq-bilin-approx-ker}, irrespective of
whether $\ell^h$ is equilibrated or not. Then, we can use Lemma~\ref{lem-KPW} to obtain
\begin{equation}
  \label{eq-pf74a}
  \| \mbs{u}_0 - \mbs{u}_\eta^h \|_{[H^1(\Omega)]^d}
  \le
  \| \mbs{u}_0 - \mbs{u}_\eta \|_{[H^1(\Omega)]^d}
  +
  \| \mbs{u}_\eta - \mbs{u}_\eta^h \|_{[H^1(\Omega)]^d}\ .
\end{equation}

To calculate the second error, let us define
\begin{equation}
  \label{eq-errors}
  \mbs{e} \defined \mbs{u}_{\eta} - \mbs{u}_{\eta}^h\ ,
  \qquad
  \mbs{e}_i \defined \mbs{u}_{\eta} - \Pi^h \mbs{u}_{\eta}\ ,
  \qquad
  \mbs{e}_h \defined \mbs{u}_{\eta}^h - \Pi^h \mbs{u}_{\eta}\ .
\end{equation}
where $\Pi^h\mbs{u}_{\eta}\in \mathcal{V}^h$ is the $L^2$ projection
of the function $\mbs{u}_{\eta}$ onto $\mathcal{V}^h$. Then, we can split the error in the finite
element solution of the regularized problem into two contributions, namely,
\begin{equation}
  \label{eq-pppt}
    \| \mbs{e} \|_{[H^1(\Omega)]^d}
    \le
    \| \mbs{e}_i\|_{[H^1(\Omega)]^d}
    +
    \| \mbs{e}_h \|_{[H^1(\Omega)]^d}\ .
\end{equation}
The first term is the interpolation error
\begin{equation}
  \label{eq-ppp0}
   \| \mbs{e}_i \|_{[H^1(\Omega)]^d}
   \lesssim
   h^k \,
   |\mbs{u}_{\eta} |_{[H^{k+1}(\Omega)]^d}\ .
\end{equation}
For the second term, we write
\begin{equation}
  \label{eq-ppp1}
    \| \mbs{e}_h \|_{[H^1(\Omega)]^d}^2
      \lesssim (\triple{\mbs{e}^h}^h_{\eta})^{2}
      = a_\eta^h(\mbs{e}^h,\mbs{e}^h)
     = a_{\eta}^h(\mbs{e},\mbs{e}^h) + a_{\eta}^h(-\mbs{e}_i,\mbs{e}^h)
\end{equation}
Moreover,
\begin{equation}
  \label{eq-ppp2}
  a_{\eta}^h(\mbs{e},\mbs{e}^h) =
  a_{\eta}^{h}(\mbs{u}_{\eta}, \mbs{e}^h) - a_{\eta}^h(\mbs{u}^h_{\eta}, \mbs{e}^h) =
  a_{\eta}^{h}(\mbs{u}_{\eta}, \mbs{e}^h) - \ell^h(\mbs{e}^h) + \eta\,(\mbs{u}_{p,0}^h,\mbs{e}^h)
\end{equation}
and
\begin{equation}
  \label{eq-ppp3}
  a_{\eta}^h(-\mbs{e}_i,\mbs{e}^h)
  \lesssim
  \triple{\mbs{e}_i}_{\eta}^h\, \triple{\mbs{e}_h}_{\eta}^h\ .
\end{equation}
Combining Eqs.~\eqref{eq-ppp1} to~\eqref{eq-ppp3} we obtain that the norm of the error $\mbs{e}_h$
can be bound as
\begin{equation}
  \label{eq-ppp4}
  \| \mbs{e}_h \|_{[H^1(\Omega)]^d} \lesssim
  \sup_{\mbs{w}_h\in \mathcal{V}_h} \frac{|a_{\eta}^h(\mbs{u}_{\eta}, \mbs{w}_h) - \ell(\mbs{w}^h)}{\| \mbs{w}_h \|_{[H^1(\Omega)]^d}}
  + \eta\,\| \mbs{u}_{p,0}^h \|_{[L^2(\Omega)]^d} + \triple{\mbs{e}_i}_{\eta}^h\ .
\end{equation}
Combining this bound with Eqs~\eqref{eq-pf74a} and \eqref{eq-ppp0}, the main result is proven. Note that
when the approximate bilinear form coincides with exact one the error bound simplifies to
\begin{equation}
  \label{eq-}
  \| \mbs{u}_0 - \mbs{u}_\eta^h \|_{[H^1(\Omega)]^d}
  \lesssim
  h^k \; |\mbs{u}|_{[H^{k+1}(\Omega)]^d} +\eta\; \|\ell^h\|_{[H^{-1}(\Omega)]^d}\ .
\end{equation}
The previous theorem shows that, as long as the approximate bilinear form vanishes for infinitesimal
rigid body motions, the error in the two-step solution is controlled by the interpolation error, the
error of the bilinear form, and an error that is proportional to the regularizing parameter~$\eta$.
This result should be compared with bound~\eqref{eq-neq-bg} that shows that if the two-step method
is not used, the size of the finite element solution grows with $\eta^{-1}$.

\end{proof}

\section{Numerical examples}
\label{sec-examples}
We show four examples of pure traction problems whose solutions are approximated with the method
proposed in this work. All equations are assumed to be nondimensionalized.

\subsection{Deformable sphere under a central force field}
\label{subs-sphere}
In the first example, we study the deformation of an elastic sphere of radius $R$ under a central force of the form
\begin{equation*}
  \label{eq-central-load}
\mbs{f}(r,\theta,\phi) = - C \frac{r}{R}\, \mbs{e}_r\ ,
\end{equation*}
where $(r,\theta,\phi)$ are spherical coordinates with origin at the center of the sphere, $C$ is a known constant, and $\mbs{e}_r$ is the unit vector in the radial direction at every point. The solution to this problem is known in closed form and can be found, for example, in the classical book by Love \cite[pg. 142-143]{love1927wt}. There, it is shown that, if expressed in spherical coordinates, the radial component of the displacement field is the only one not vanishing, and has value:
\begin{equation}
u_r = - \frac{1}{10} \frac{C\,R\,r}{\lambda+2\mu}
\left(
\frac{5\lambda+6\mu}{3\lambda+2\mu} -
\frac{r^2}{R^2}
 \right)\ .
\end{equation}
Again in spherical coordinates, the only non-zero components of the strain tensor are:
\begin{equation}
\begin{split}
    \varepsilon_{rr} &=
    \frac{3 C r^2-C R^2}{10 \lambda  R+20 \mu  R}-\frac{C R}{5 (3 \lambda +2 \mu )}
    \ ,\\
    \varepsilon_{\theta\theta} &= \varepsilon_{\phi\phi} =
    -\frac{C R}{10 (\lambda +2 \mu )} \left(-\frac{4 \lambda }{3 \lambda +2 \mu }-\frac{r^2}{R^2}+3\right)\ .
\end{split}
\end{equation}

\begin{figure}[t]
  \centering
  \includegraphics[width=0.35\textwidth]{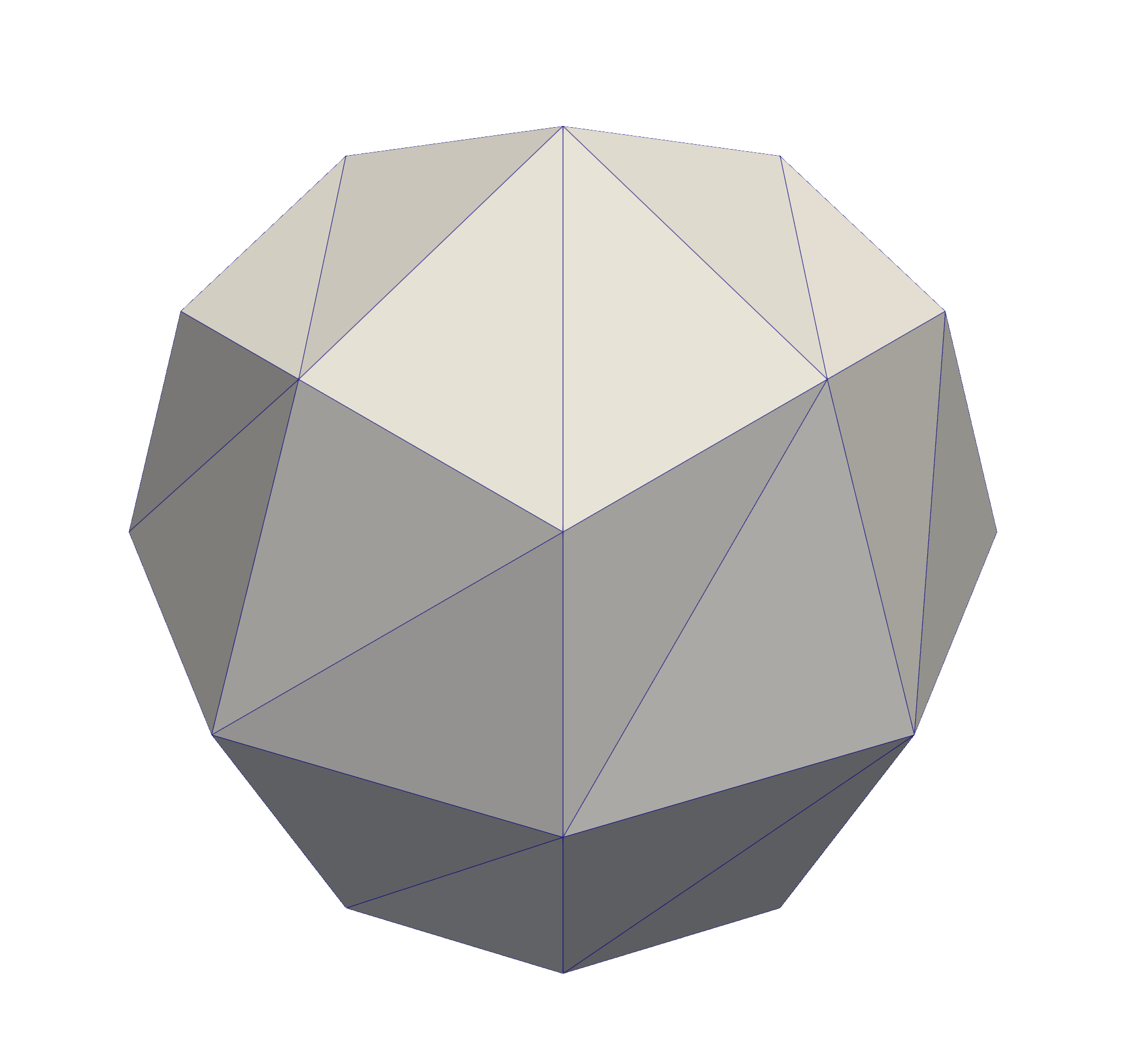}%
  \includegraphics[width=0.35\textwidth]{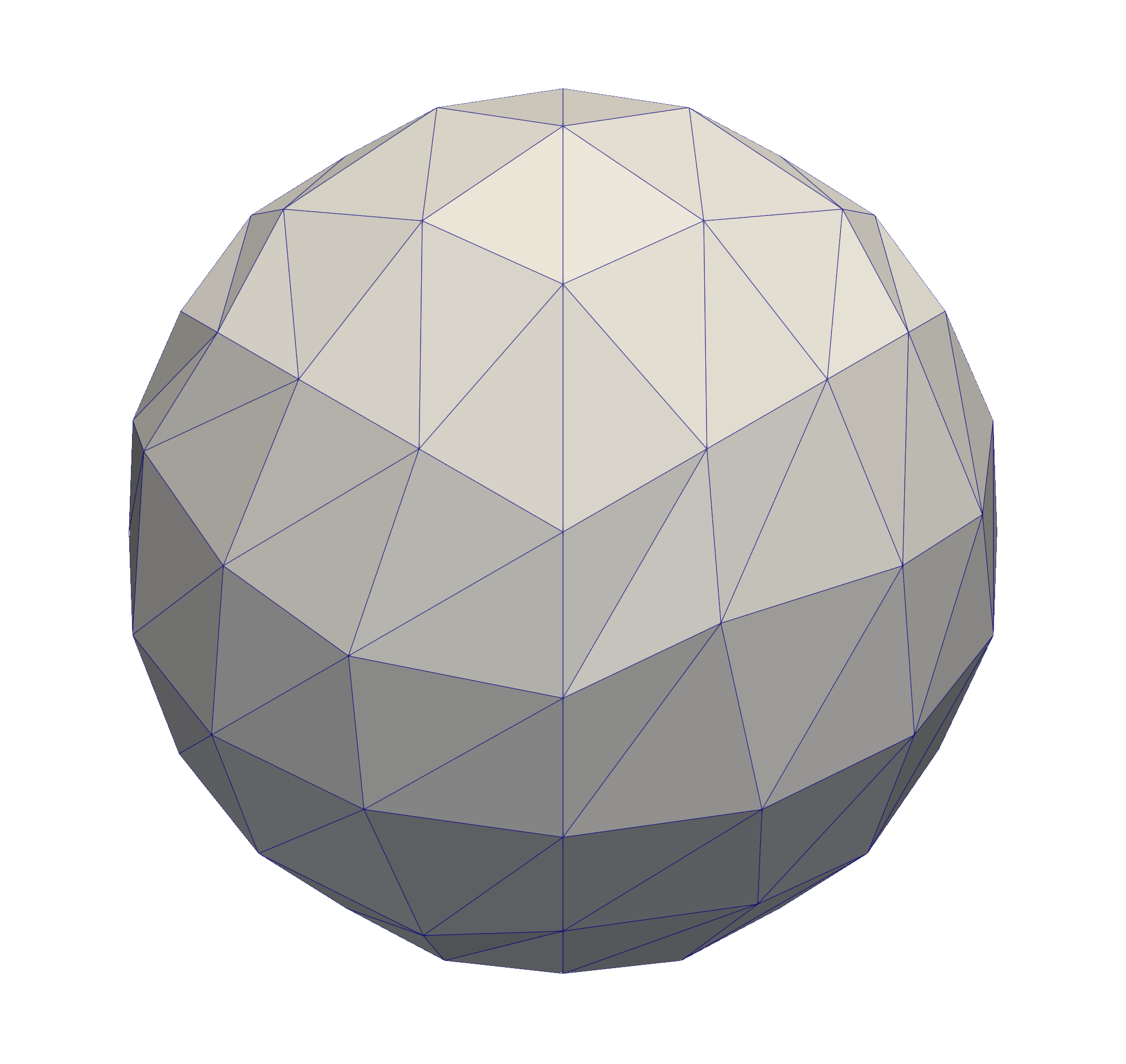}\\
  \includegraphics[width=0.35\textwidth]{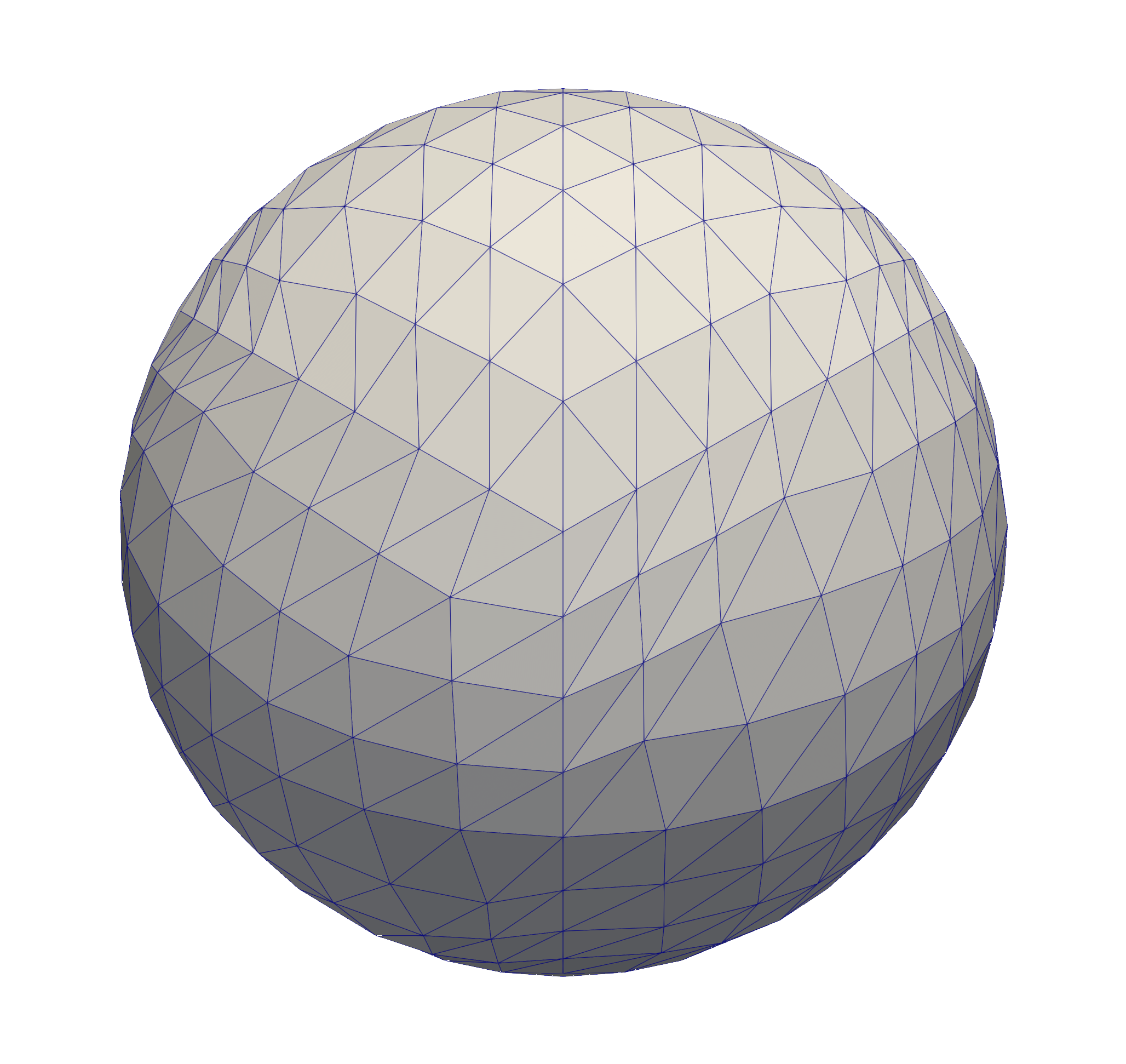}%
  \includegraphics[width=0.35\textwidth]{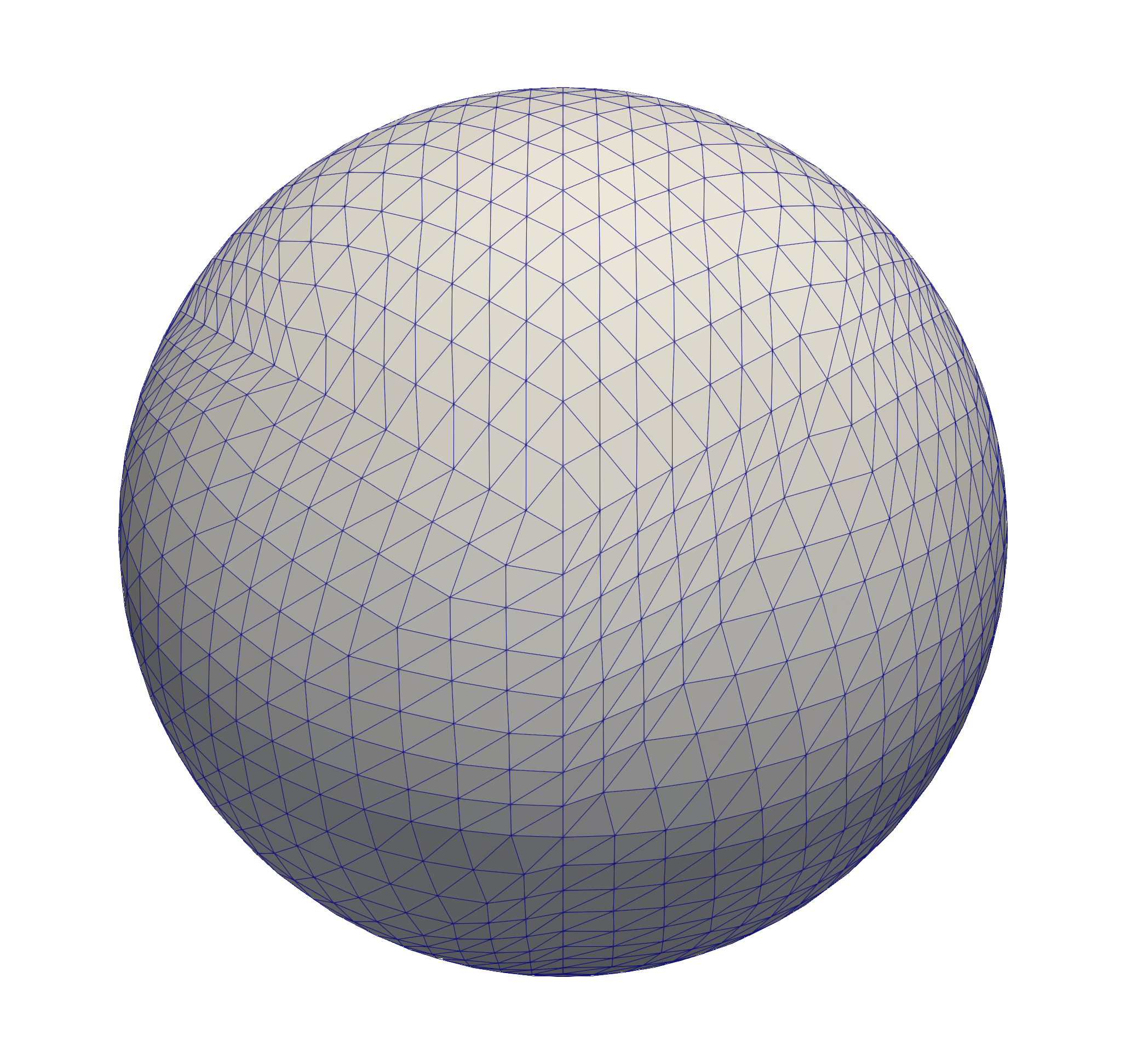}%
  \caption{Finite element models for the elastic sphere problem.
  Number of elements, from top to bottom and left to right: 
  $96\cdot8^j, j=0,\ldots,3$)}.
  \label{fig-sphere-mesh}
\end{figure}

\begin{figure}[t]
  \centering
  \includegraphics[width=0.8\textwidth]{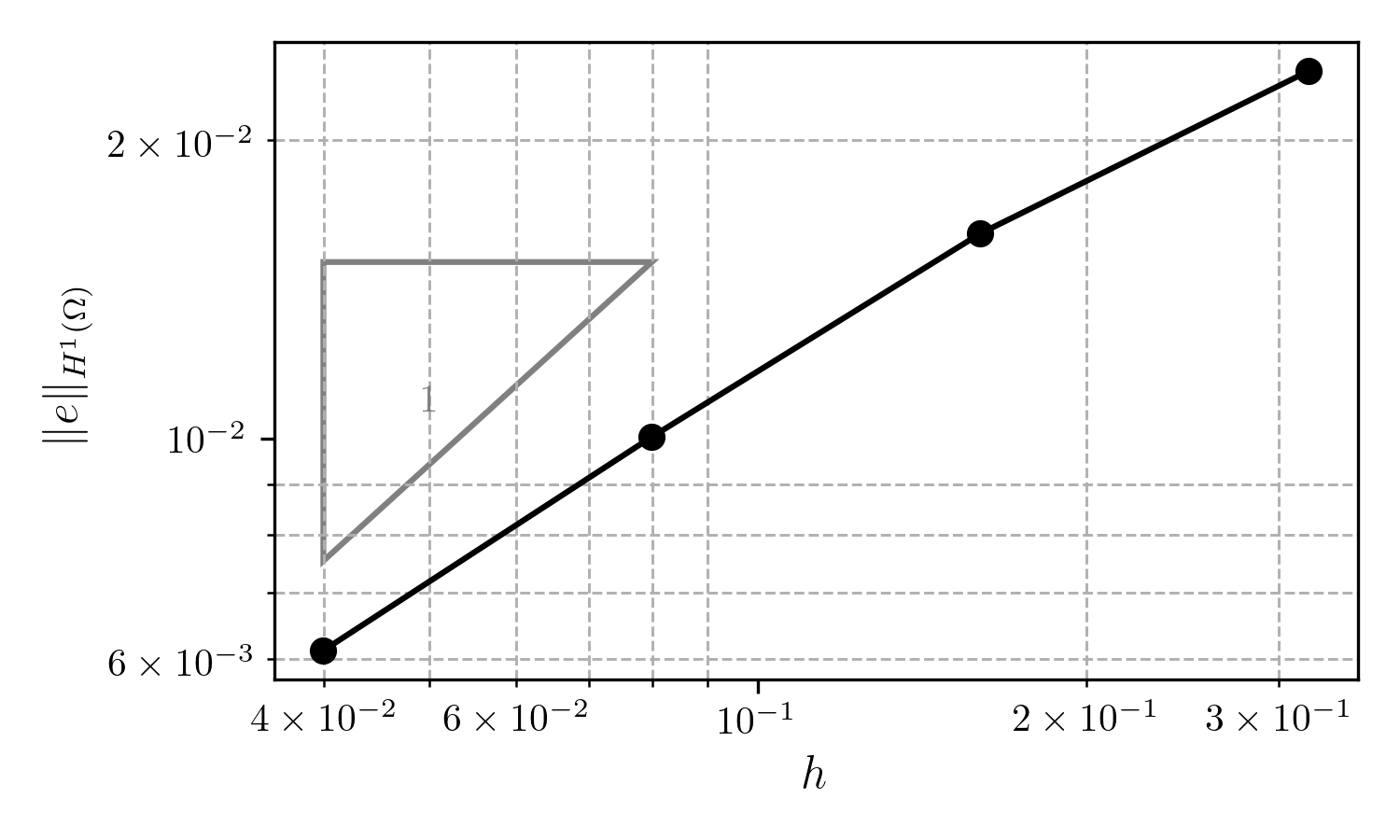}
  \includegraphics[width=0.8\textwidth]{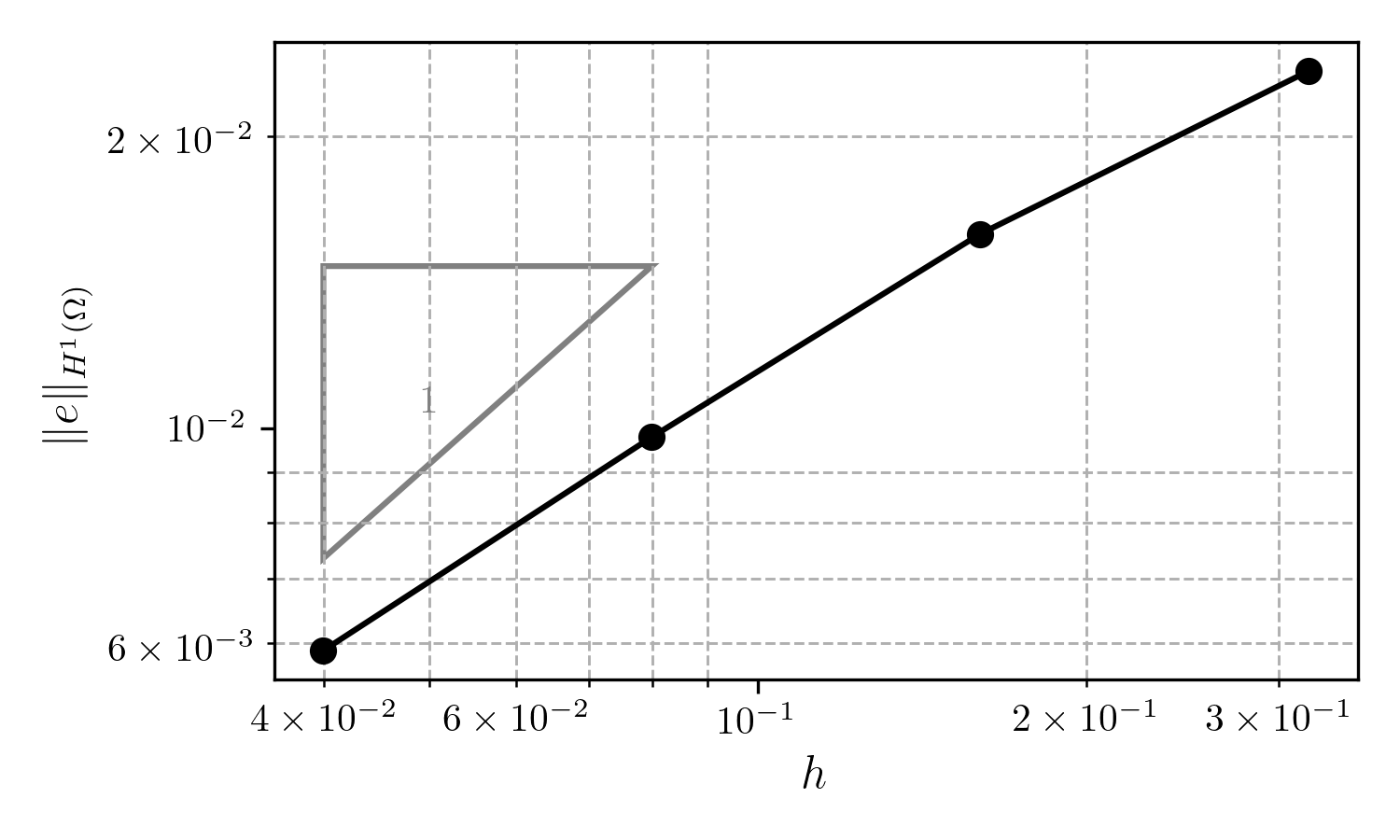}
   \caption{$H^1$ error in the numerical solution of the elastic sphere problem. Top: solutions with
    nodal constraints; bottom: solutions obtained with the regularized method.}
  \label{fig-sphere-errors}
\end{figure}

In our simulations, we select $R=1/2, C=2$ and the Lam\'e parameters correspond to a material with
Young's modulus $E=1$ and Poisson's ratio $\nu=0.3$. First, we solve this Neumann problem using
a standard finite element formulation and we use constraints to preclude rigid body motions. For
that, we introduce a system of Cartesian coordinates with its
center coinciding with the center of the sphere. Then, we create tetrahedral meshes
with an increasing number of elements ($96\cdot8^j, j=0,\ldots,3$) (see
Figure~\ref{fig-sphere-mesh}). Each of these meshes has a node at the origin and also at the intersection of the Cartesian coordinate axes with the
surface of the sphere. We hold the three displacements of the center node constraint all
translations. To preclude (infinitesimal) rotations, we hold the displacements on the $x$ and $y$
direction of the node at $(x,y,z)=(0,0,1/2)$; similarly, we constraint the displacement in the $x$
direction of the node with coordinates $(x,y,z)=(0,1/2,0)$.

Instead of constraining nodal degrees of freedom, the solution to this traction problem can be
obtained using the regularized formulation described in Section~\ref{sec-fem}, employing a
regularization parameter $\eta \propto h^{1.2}$.

Both numerical methods converge with optimal order to the exact solution of the elastic sphere as
verified in Figure~\ref{fig-sphere-errors}, the differences between
the errors being almost negligible. As claimed before, the advantage of the newly proposed
formulation lies, in this case, on the simplicity of the solution, which does not demand boundary conditions at all.

\begin{figure}[t]
  \centering
  \includegraphics[width=0.8\textwidth]{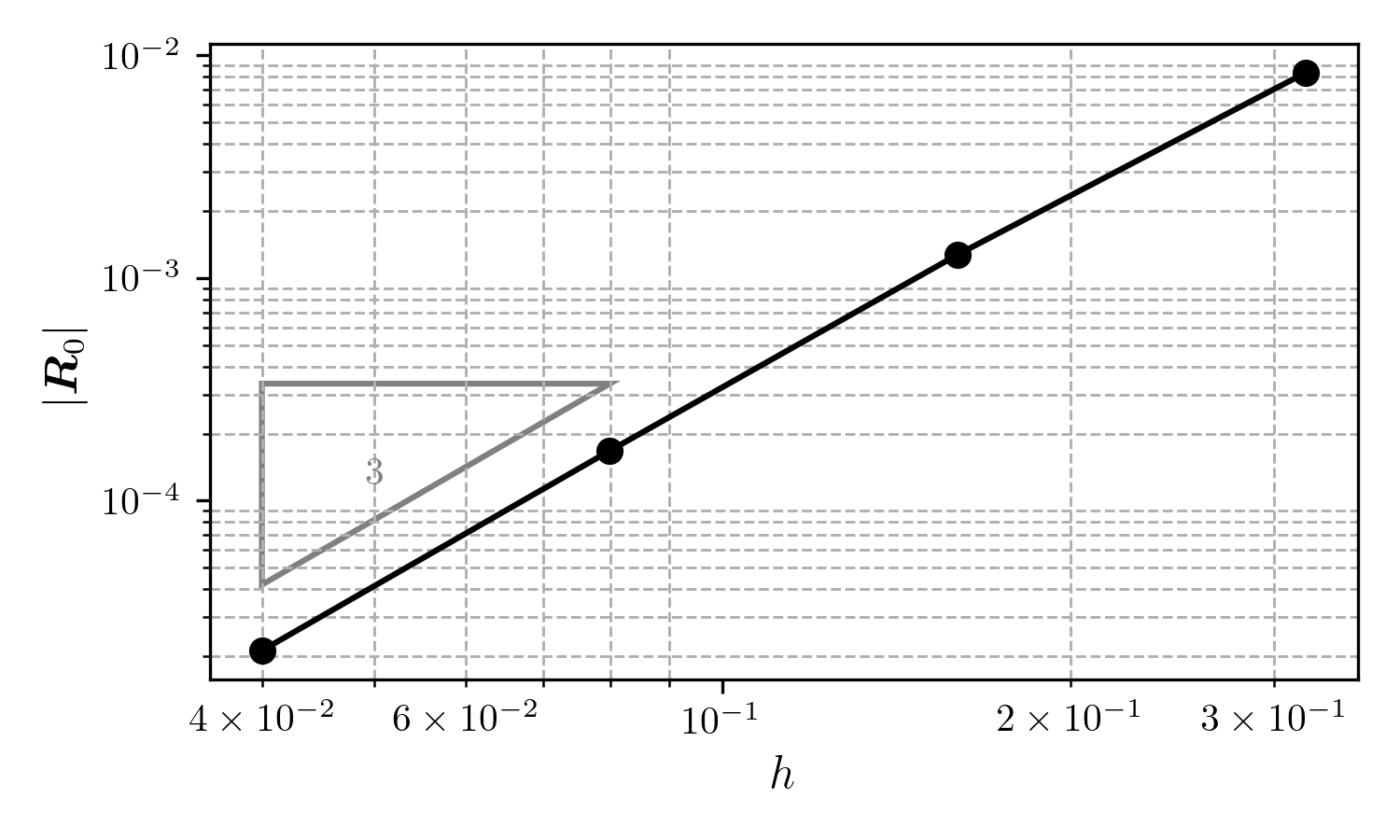}
  \caption{Elastic sphere with a perturbed central load. Norm of the nodal reactions in the
    solution obtained with constrained degrees of freedom.}
  \label{fig-windy-reac}
\end{figure}

\begin{figure}[p]
  \centering
  \includegraphics[width=0.8\textwidth]{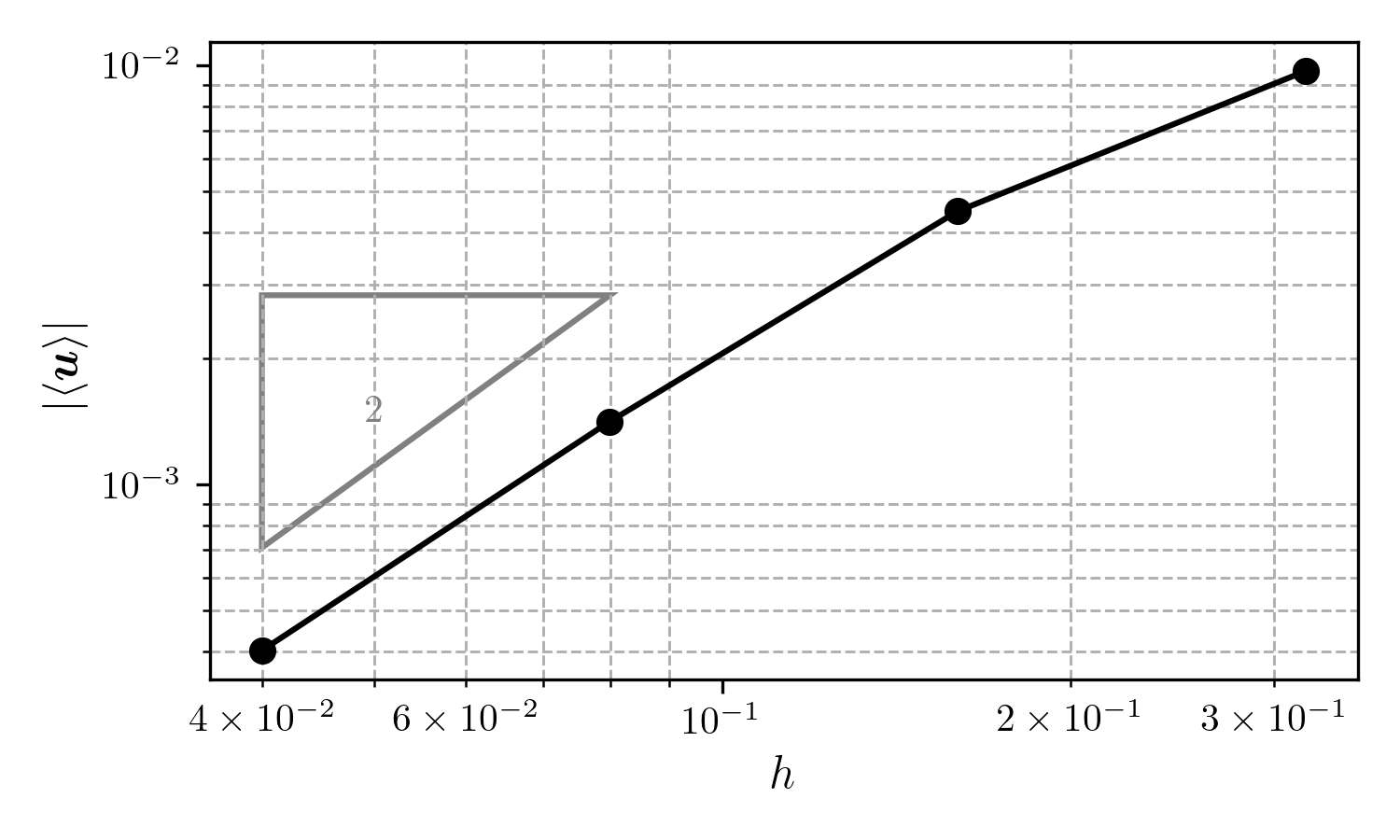}\\
  \includegraphics[width=0.8\textwidth]{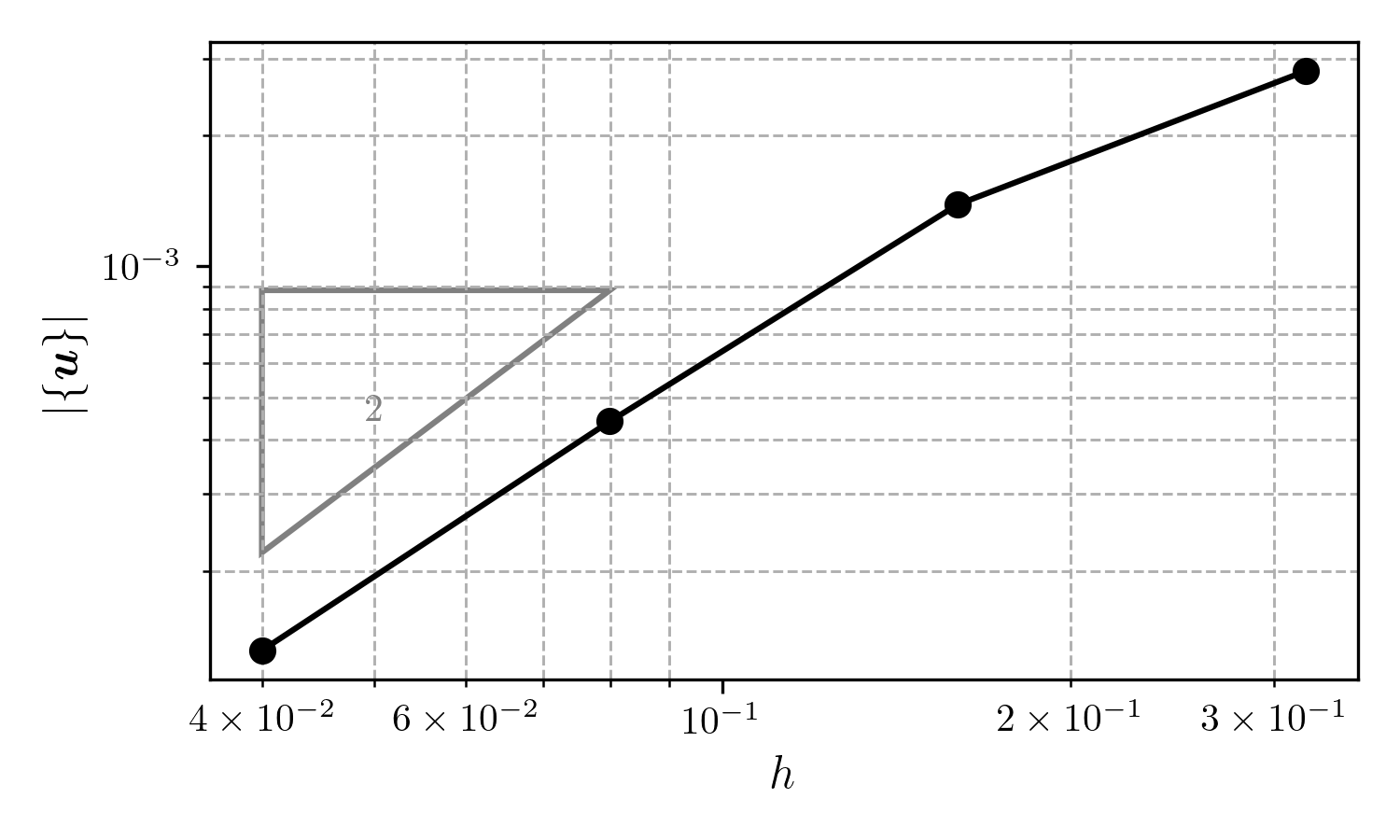}\\
  \includegraphics[width=0.8\textwidth]{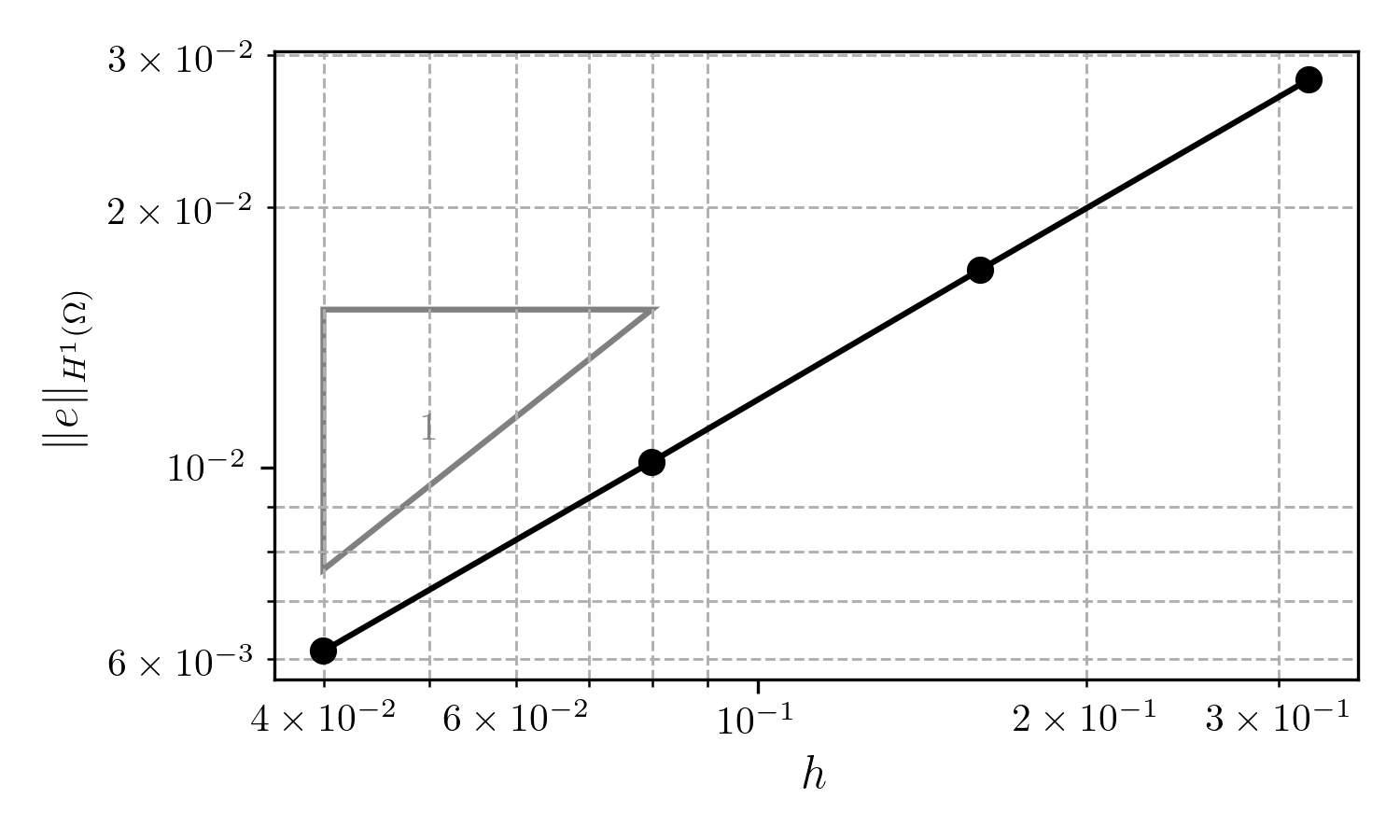}
  \caption{Elastic sphere with a perturbed central load. Solution obtained with constrained
    degrees of freedom. Top: norm of mean displacement vs. mesh size; center: 
    norm of displacement moment vs. mesh size; bottom: $H^1$ error in the solution.}
  \label{fig-windy-3}
\end{figure}

The central load~\eqref{eq-central-load} is equilibrated and the reactions in the constrained nodes of
the first solution vanish. We explore next the effects on the numerical solutions of a
(non-equilibrated) perturbation on the loading. Specifically, we consider the effect on the deformable sphere of a volume load
\begin{equation}
  \label{eq-sphere-l2}
  \tilde{\mbs{f}} = \frac{h}{R}\, \mbs{i}\ ,
\end{equation}
added to the central load~\eqref{eq-central-load}, where $\mbs{i}$ is the unit vector of the Cartesian basis in the
direction of the $x$ axis. The total load $\mbs{f}+\tilde{\mbs{f}}$ is not centered and we expect a
reaction of value $|\tilde{\mbs{f}}|\frac{4}{3}\pi R^3\sim h^3$ on the constrained nodes, which is
confirmed in the results depicted in Figure~\ref{fig-windy-reac}. Also, as a consequence of the
unequilibrated loading, the numerical solutions obtained with this method fail to be centered as
illustrated Figure~\ref{fig-windy-3}. There, it can be verified that
the numerical method is able to calculate a solution for each mesh size, but all of them naturally
fail to be centered. Notwithstanding these limitations, the constrained method yields optimally-convergent
solutions, as can be observed in Figure~\ref{fig-windy-3}.

\begin{figure}[ht]
  \centering
  \includegraphics[width=0.8\textwidth]{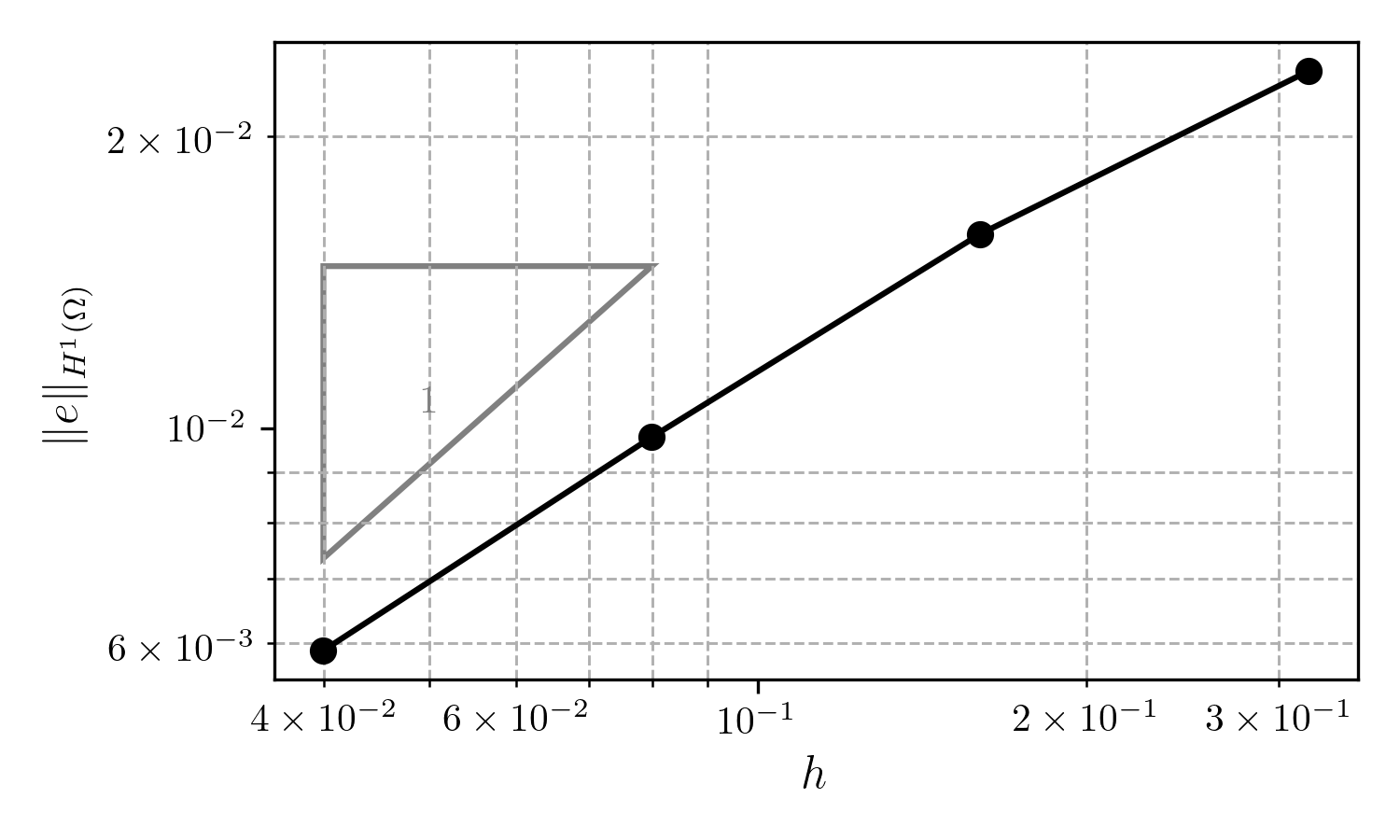}
  \caption{Elastic sphere with a perturbed central load. $H^1$ error in the solution obtained with the two-step method.}
  \label{fig-windy-error2}
\end{figure}

The method introduced in Section~\ref{subs-two-step} can deal with incompatible loads. Solutions
calculated with this two-step method are exactly centered, without requiring any nodal constraint,
and optimally convergent as seen in Figure~\ref{fig-windy-error2}.

The results shown in this section confirm that the methods introduced in Sections~\ref{sec-fem}
and~\ref{subs-two-step} can deal, respectively, with equilibrated and non-equilibrated loading
without using any nodal constraint. In both cases, the convergence of the methods is confirmed to be
optimal.

\begin{figure}[t]
  \centering
  \includegraphics[width=0.8\textwidth]{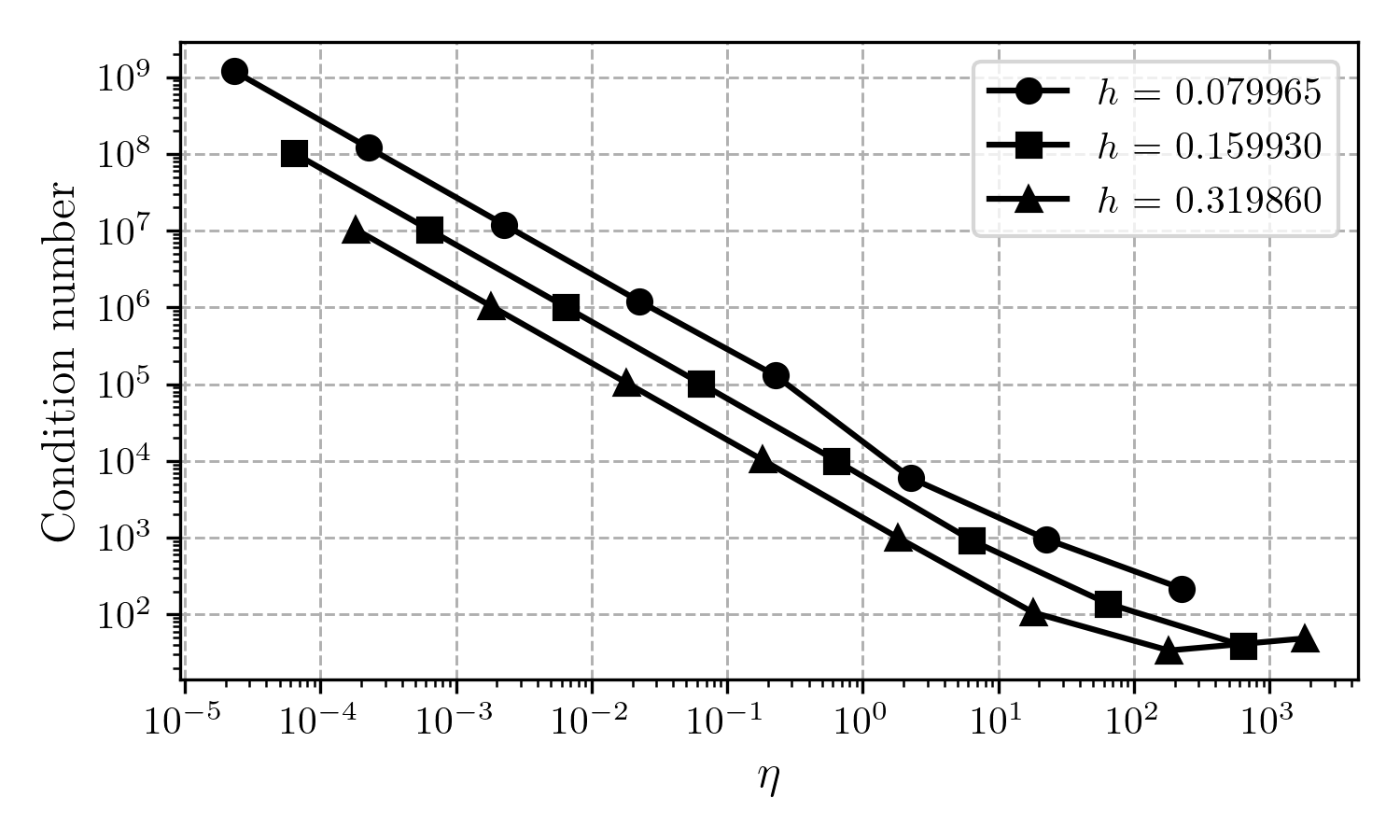}
  \caption{Condition number of the stiffness matrix as a function of the regularization parameter for the regularized solution of the deformable sphere problem.}
  \label{fig-sphere-condition}
\end{figure}

The stiffness matrix of a pure Neumann problem is singular. When the regularizing method introduced
in Section~\ref{sec-fem} is employed, this matrix becomes regular, but its condition number deteriorates with $\eta\to0$. Figure~\ref{fig-sphere-condition} depicts the condition number of the
stiffness matrix employed in the solutions studied in this section. For values of $\eta$ of the
order of~1, the
matrix is well-conditioned, yet the perturbation introduced into the Neumann problem cannot be
considered to be small and thus the solution might be far from the exact one. For $\eta \ll 1$, the effects of the regularization become less noticeable, yet the stiffness matrix has a large condition number.

From the theoretical point of view, any $\eta>0$ makes the Neumann problem well posed. Indeed, the
regularization introduces a (small) positive eigenvalue in the spectrum of the elasticity operator, enough to
render it coercive. Similarly, in the discretized Neumann problem, any positive $\eta$
introduces a positive eigenvalue in the stiffness matrix which makes it invertible, if arbitrary
precision could be employed. In practice, however, finite precision floating point calculations
render the discrete solution useless when the stiffness is ill-conditioned, both for direct solvers
as well as for iterative ones.

If the loading is equilibrated, the method of Section~\ref{subs-iterative} can be used, providing
the exact pure traction solution even when a relatively large regularized parameter is employed. As an
example, we solve the Neumann problem of the deformable sphere with $\eta=1$ and the iterative
method introduced in Section~\ref{subs-iterative}.

The Neumann solution is calculated with the finite element
discretizations depicted in Figure~\ref{fig-sphere-mesh}, with the same material properties employed before and not a regularization parameter $\eta=1$. For every discretization, the iterative scheme~\eqref{eq-iterative} is employed. 

\begin{table}[ht]
    \centering
    \begin{tabular}{l l l l l }
    \toprule
     Iteration & $j=0$ & $j=1$ & $j=2$ & $j=3$ \\
     \midrule     
$1$ & $4.4148\cdot 10^{-2}$ & $2.8104\cdot 10^{-2}$ & $1.2475\cdot 10^{-2}$ & $4.8739\cdot 10^{-3}$\\
$2$ & $4.3828\cdot 10^{-4}$ & $3.3498\cdot 10^{-4}$ & $1.5544\cdot 10^{-4}$ & $6.1022\cdot 10^{-5}$\\
$3$ & $4.2696\cdot 10^{-6}$ & $3.9575\cdot 10^{-6}$ & $1.9449\cdot 10^{-6}$ & $7.7661\cdot 10^{-7}$\\
$4$ & $4.1818\cdot 10^{-8}$ & $4.7464\cdot 10^{-8}$ & $2.4436\cdot 10^{-8}$ & $9.9100\cdot 10^{-9}$\\
$5$ & $4.3745\cdot 10^{-10}$ & $6.5080\cdot 10^{-10}$ & $3.1821\cdot 10^{-10}$ & $1.2676\cdot 10^{-10}$\\
\bottomrule
    \end{tabular}
    \caption{2-norm of the residuals in the solution of the iterative scheme~\eqref{eq-iterative} with the discretizations depicted in Figure~\ref{fig-sphere-mesh}. The number of elements in each mesh is $96\cdot8^j, j=0, \ldots, 3$.}
    \label{tab-iterations}
\end{table}

\begin{figure}[ht]
  \centering
  \includegraphics[width=0.8\textwidth]{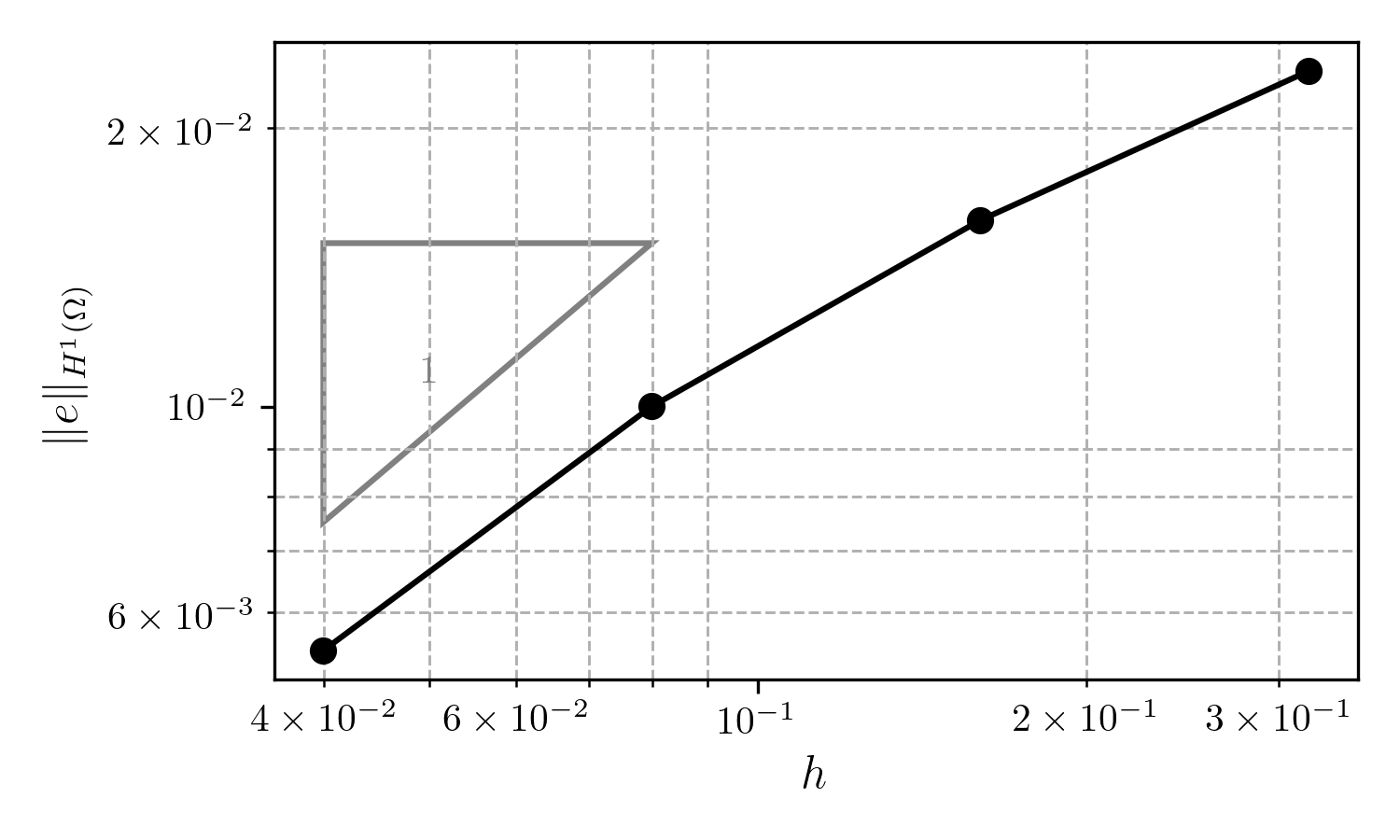}
  \caption{Elastic sphere under central load. $H^1$ norm of the errors in the solutions obtained with the iterative scheme~\eqref{eq-iterative}.}
  \label{fig-iterative-h1error}
\end{figure}

The solutions obtained with the iterative method can be compared with the exact ones. Figure~\ref{fig-iterative-h1error} shows the convergence of the solutions obtained. Notice that all of them have been obtained with a relatively large value of $\eta$ which yields, for every $h$, condition numbers of the order of $1$ (see Figure~\ref{fig-sphere-condition}). For all the meshes, moreover, the $H^1$ error is commensurate to the error in the regularized solution or the solution obtained with constrained degrees of freedom (see Figure~\ref{fig-sphere-errors}).

\subsection{A parallelepiped under wavy, equilibrated body loading}
\label{subs-para}
In the second example, we consider the solution of another simple Neumann problem with a statically equilibrated body load. To ensure that the load is exactly compatible, we start by selecting a body with the shape of a parallelepiped. Moreover, we use hexahedral elements to mesh it, exactly representing the geometry and preserving its symmetries.

\begin{figure}[ht]
  \centering
  \includegraphics[width=0.49\textwidth]{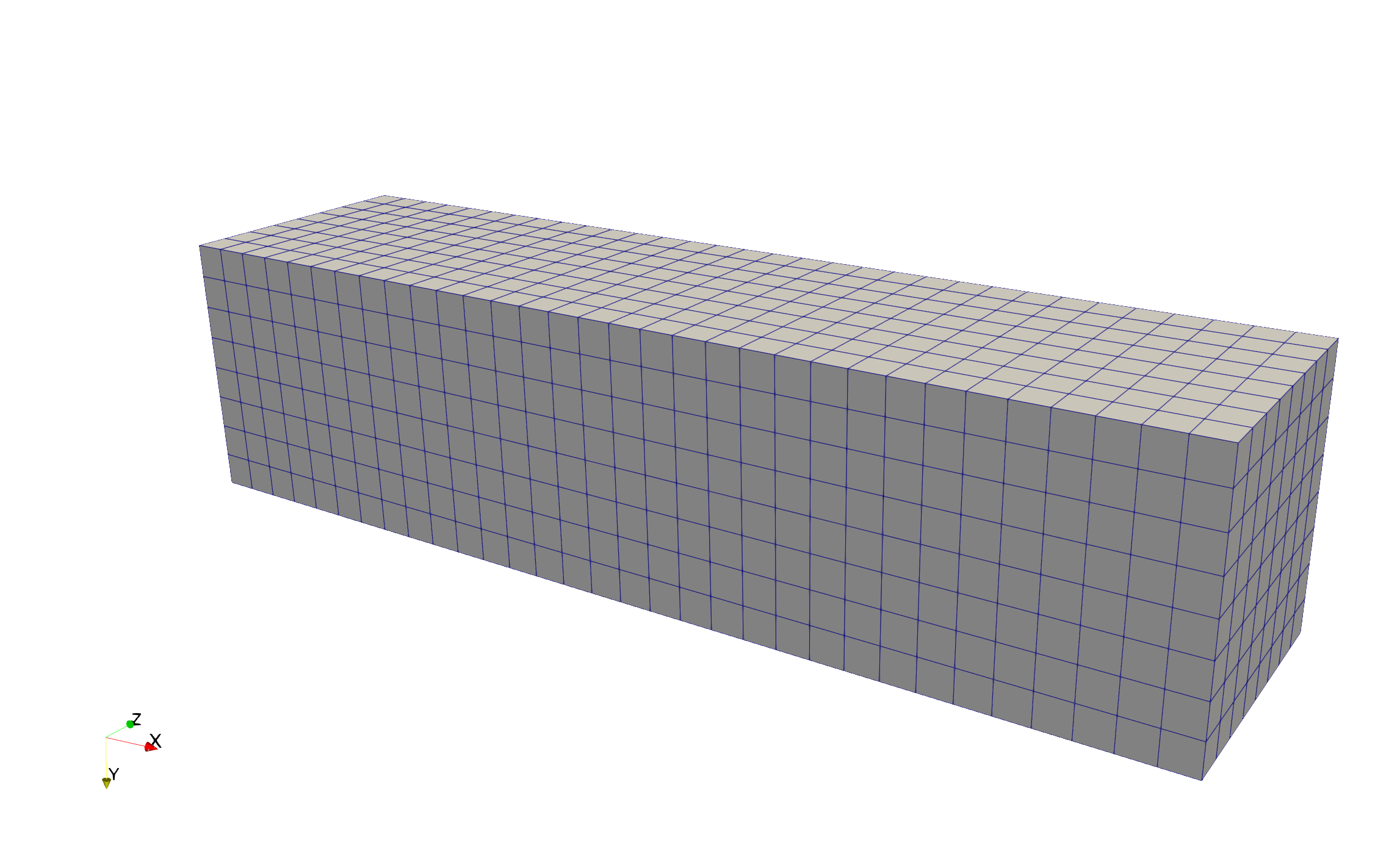}\hfill
  \includegraphics[width=0.49\textwidth]{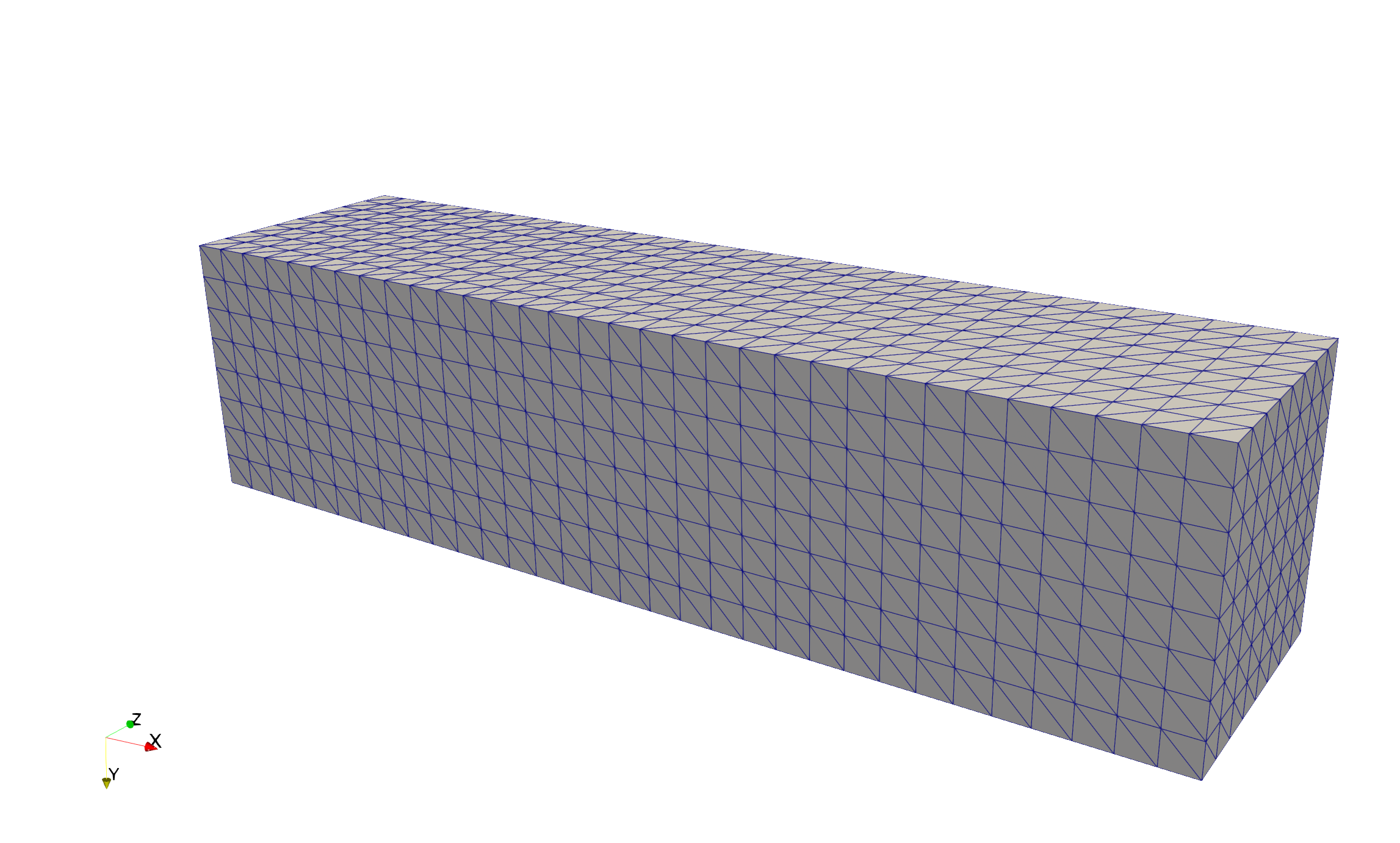}
  \caption{Parallelepiped meshed with hexahedra and tetrahedra. Notice that the element edges in the tet mesh indicate that functions in the finite element space will not be exactly symmetric with respect to the coordinate planes.}
  \label{fig-para-meshes}
\end{figure}

Specifically, we consider a parallelepiped of dimensions $L_x=4,\; L_y\equiv L_z=1$, placed with its
edges parallel to the Cartesian axes and with its center on the origin (see
Figure~\ref{fig-para-meshes}). The material is assumed to be elastic with Young's modulus $E=1$ and
Poisson's ratio $\nu=0$. A compatible body force $\mbs{f} = \sin(3\pi x/L_x)\;\mbs{i}$ is applied on
the body, with $\mbs{i}$ being, as in the first example, the unit basis vector on the $x$ direction. Since Poisson effects are eliminated, the displacement field of the body is simply
\begin{equation*}
    \mbs{u}(x,y,z) =
    \frac{L_x^2}{9\,\pi^2} \sin(3\pi x/L_x)
    \, \mbs{i}\ ,
\end{equation*}
which is a centered vector field.

\begin{figure}[t]
  \centering
  \includegraphics[width=0.8\textwidth]{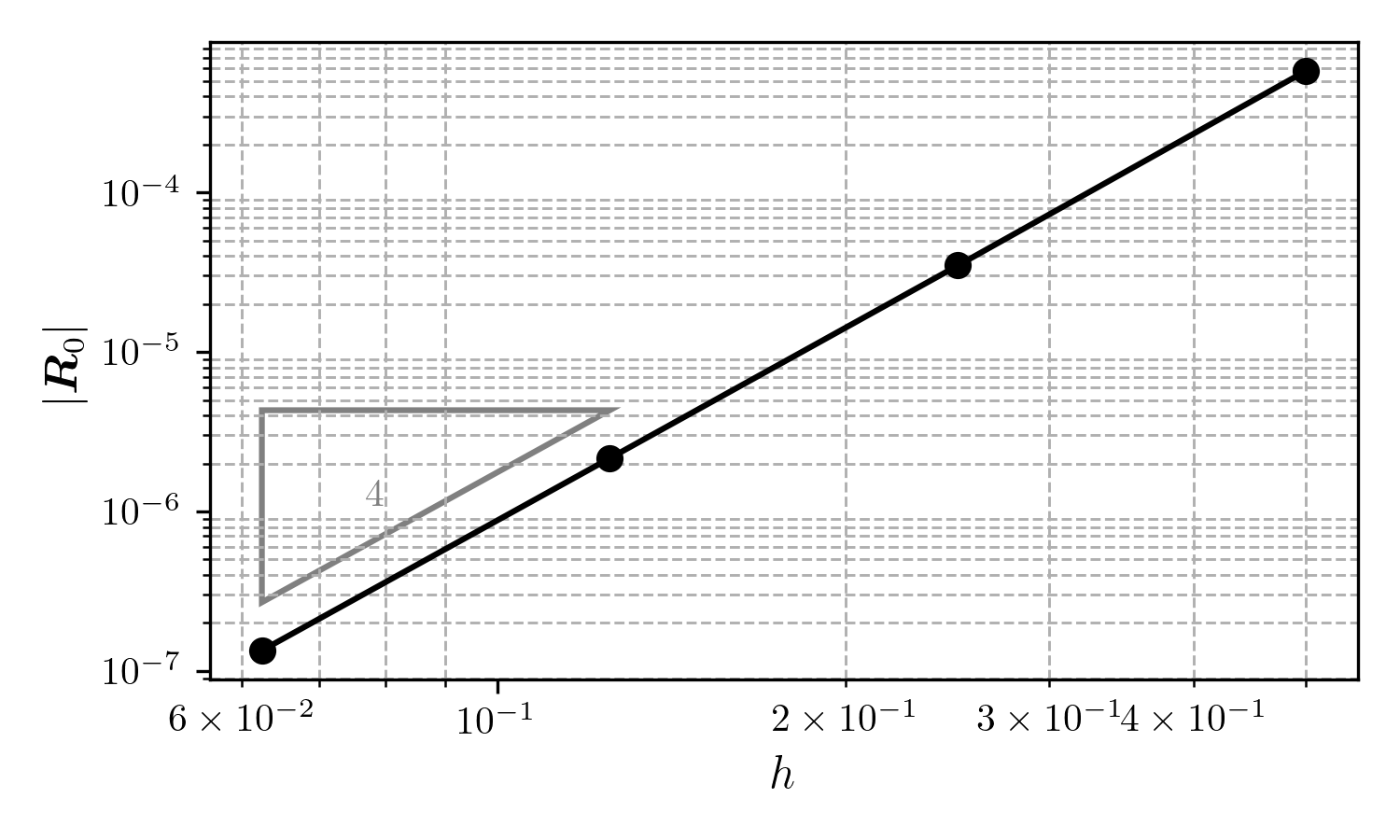}
  \includegraphics[width=0.8\textwidth]{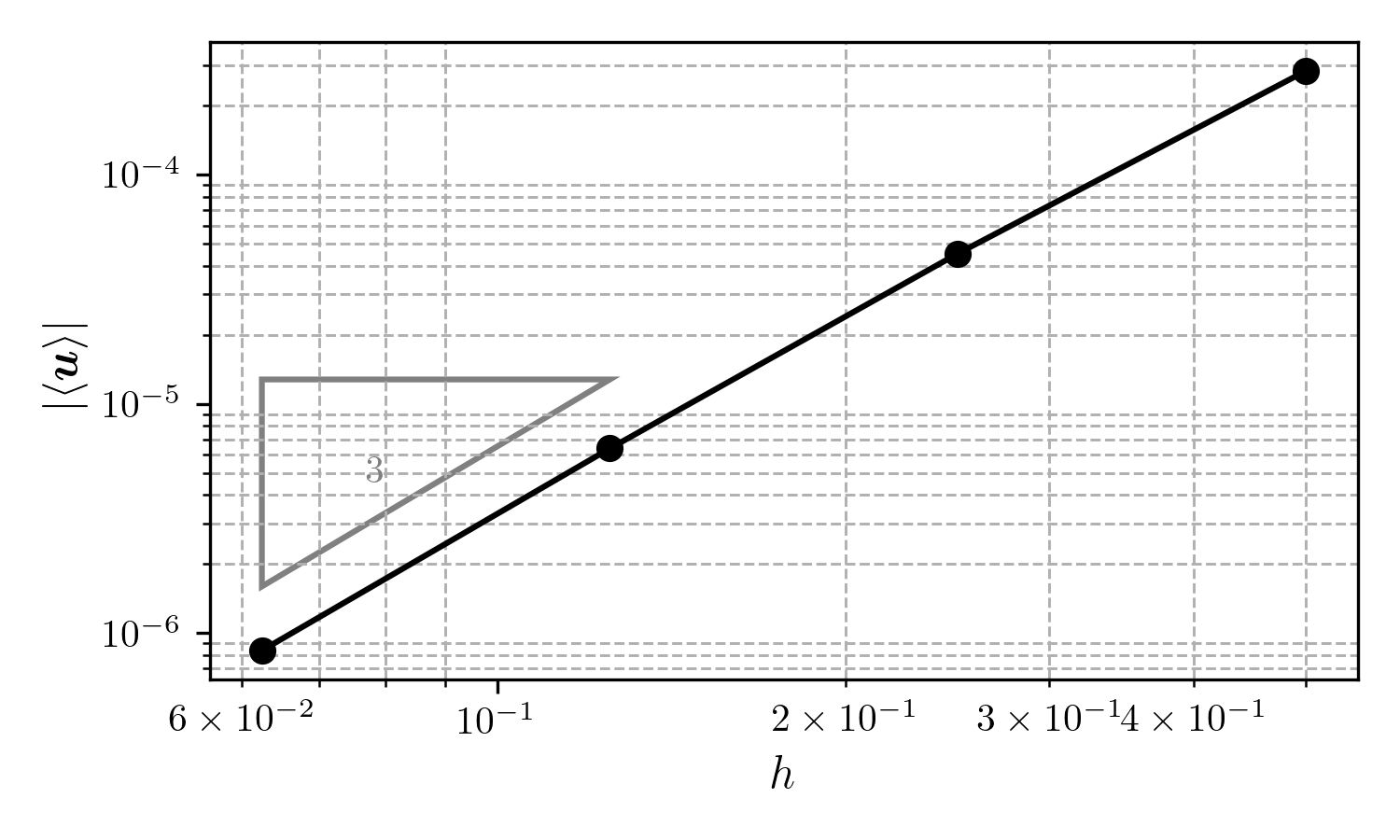}
  \includegraphics[width=0.8\textwidth]{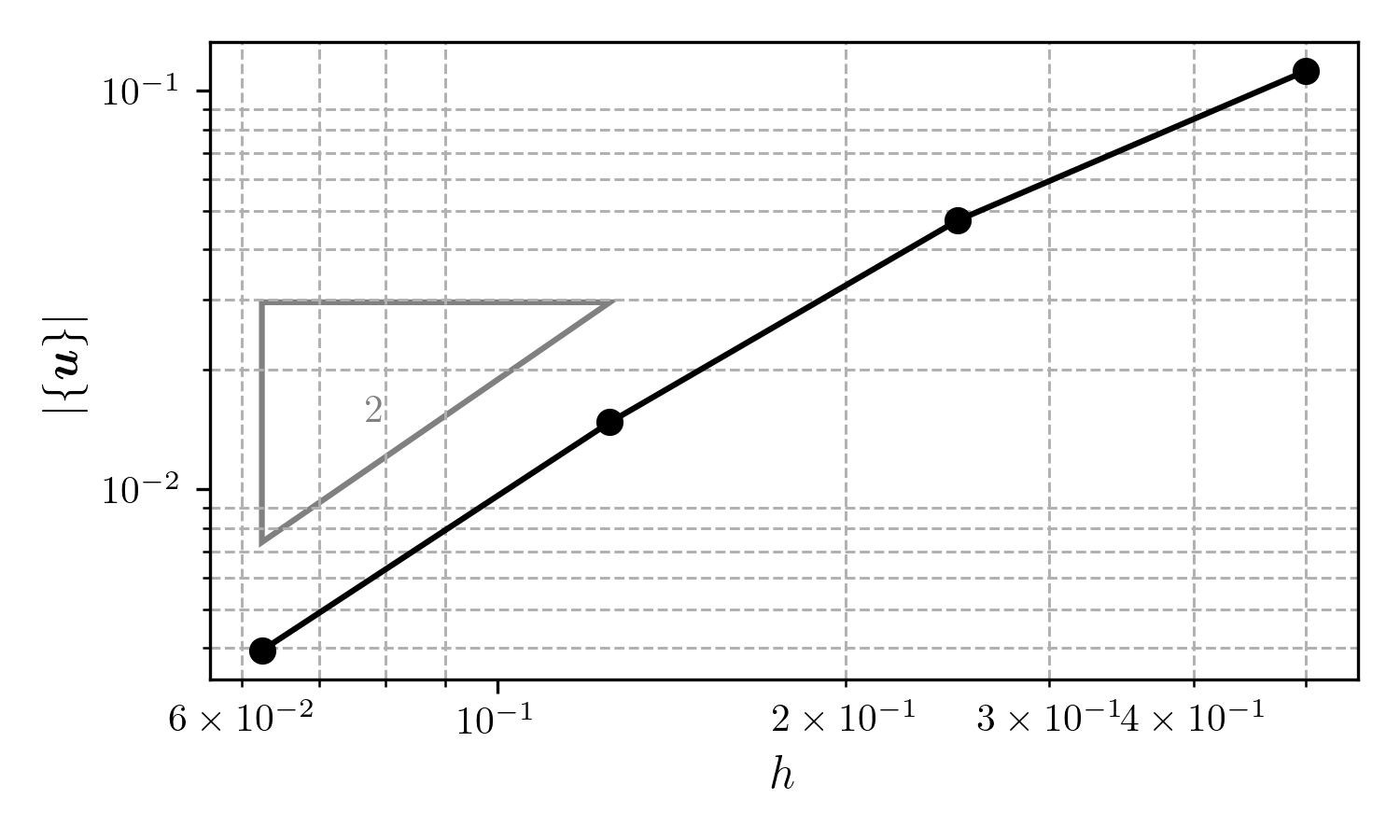}
  \caption{Parallelepiped example. 
  Solution obtained with tetrahedra and nodal constraints.
  Top: norm of the reaction at node at the origin;
  Middle: norm of mean displacements;
  Bottom: norm of the first moment of displacements.}
  \label{fig-para-constrained}
\end{figure}

This (essentially one-dimensional) Neumann problem might be solved using a standard finite element
approximation of the elasticity equations if a mesh is built that includes nodes in the origin and
on the coordinate axes. For the simple geometry under consideration this is easily verified, and
thus we can select six constraints, three at the origin and three more to preclude rigid body
motions, and obtain a unique solution.

\begin{figure}[t]
  \centering
  \includegraphics[width=0.8\textwidth]{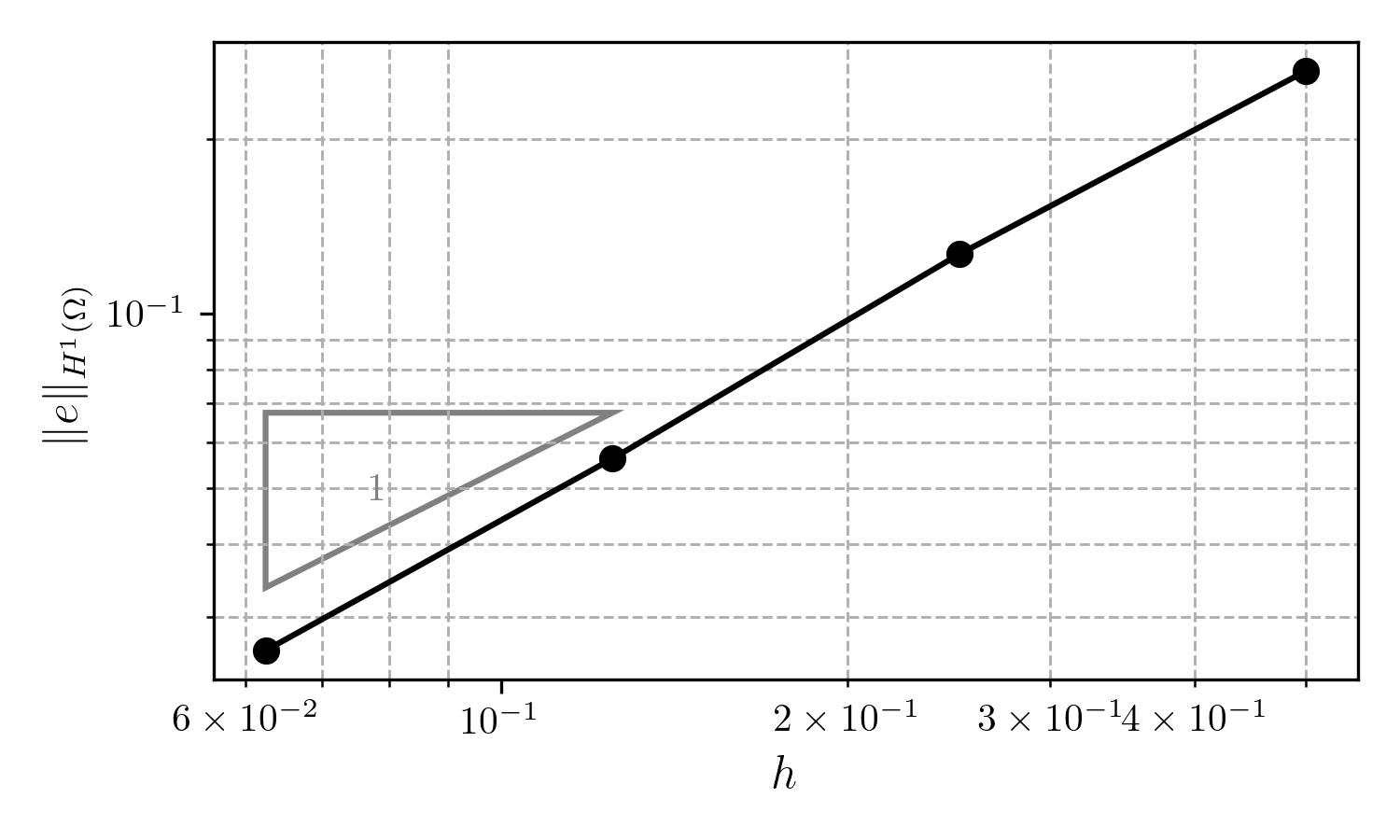}
  \includegraphics[width=0.8\textwidth]{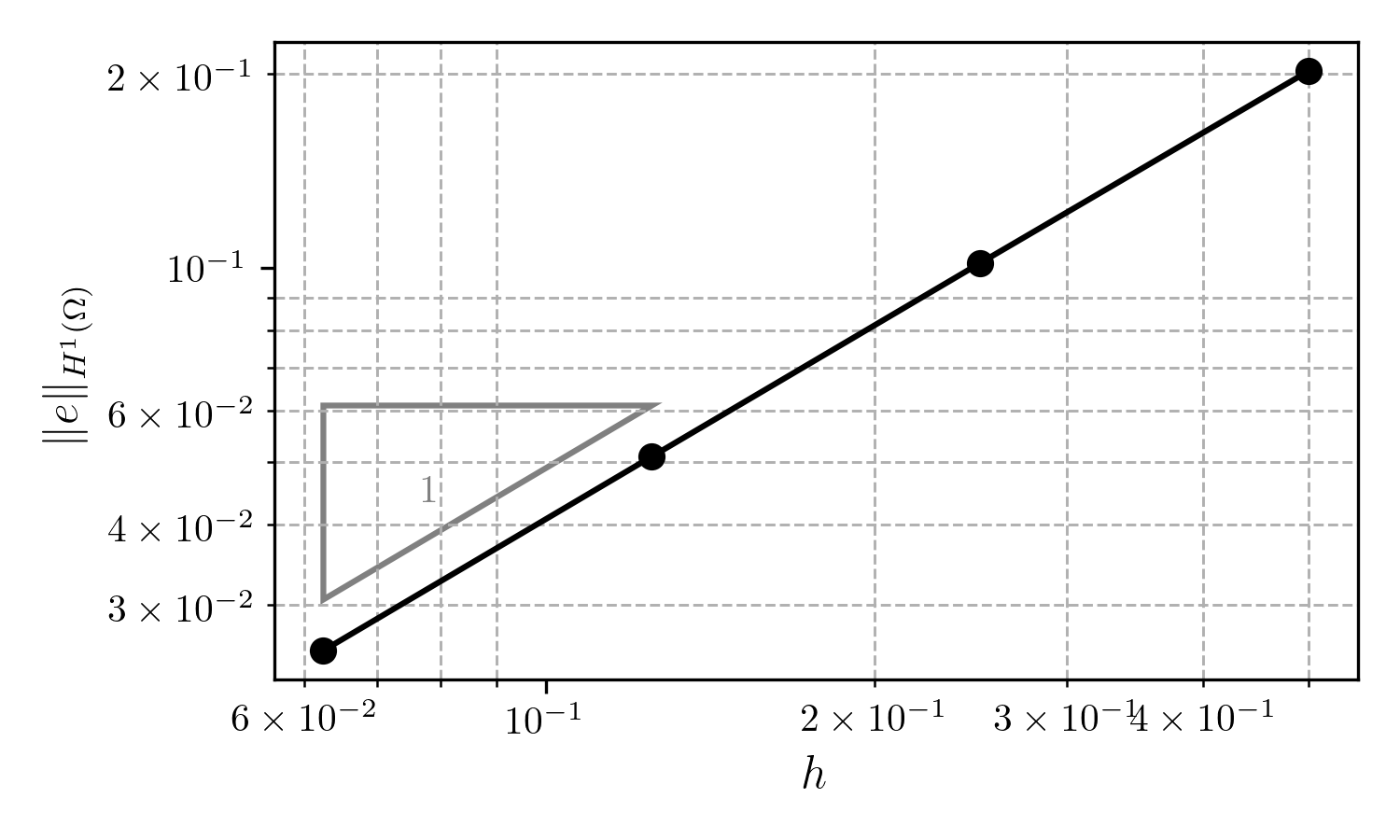}
  \caption{Parallelepiped example. 
  $H^1$ norm of the errors in the solutions obtained with
  constrained displacements.
  Top: tetrahedral mesh;
  Bottom: hexahedral mesh.}
  \label{fig-para-tqerror}
\end{figure}

We must point out, however, that even for this simple problem, constraining the
displacement field in selected nodes to preclude rigid body motions might entail errors that can
pass unnoticed unless carefully examined. Indeed, when the parallelepiped is meshed with hexahedral
elements, the discrete model preserves exactly the symmetry of the body and the load, hence the reactions on the
constrained degrees of freedom identically vanish. However, when the body is meshed with tetrahedra,
some anisotropy is induced in the solution since the discrete model is no longer exactly symmetric.
See Figure~\ref{fig-para-meshes} for an illustration of the two types of meshes employed, and the
broken symmetry due to the orientation of the tetrahedral elements.

Figure~\ref{fig-para-constrained} shows, for several mesh sizes, the norm of $\mbs{R}_{0}$, the
reaction force at the origin, when the model employs tetrahedra and all the degrees of freedom of
this node are constrained. While the size of this reaction decreases with the mesh size, it is apparent
from this figure that it never vanishes completely. Also, as a consequence of these small
incompatibilities in the loading, the solution itself is not centered. The mid and bottom plots in Figure~\ref{fig-para-constrained} depict the norms of $\langle \mbs{u} \rangle$ and $\{ \mbs{u} \}$ in the
solutions with tetrahedra, which should be zero. However, for the
unstructured mesh these two quantities are small, albeit nonzero. In contrast, when the mesh employs
regular hexahedra, the solution is exactly centered since the discretization does not spoil the
symmetry of the parallelepiped. Let us note that, as shown in Figure~\ref{fig-para-tqerror}, these constrained solutions converge optimally for both types of meshes.

\begin{figure}[ht]
  \centering
  \includegraphics[width=0.8\textwidth]{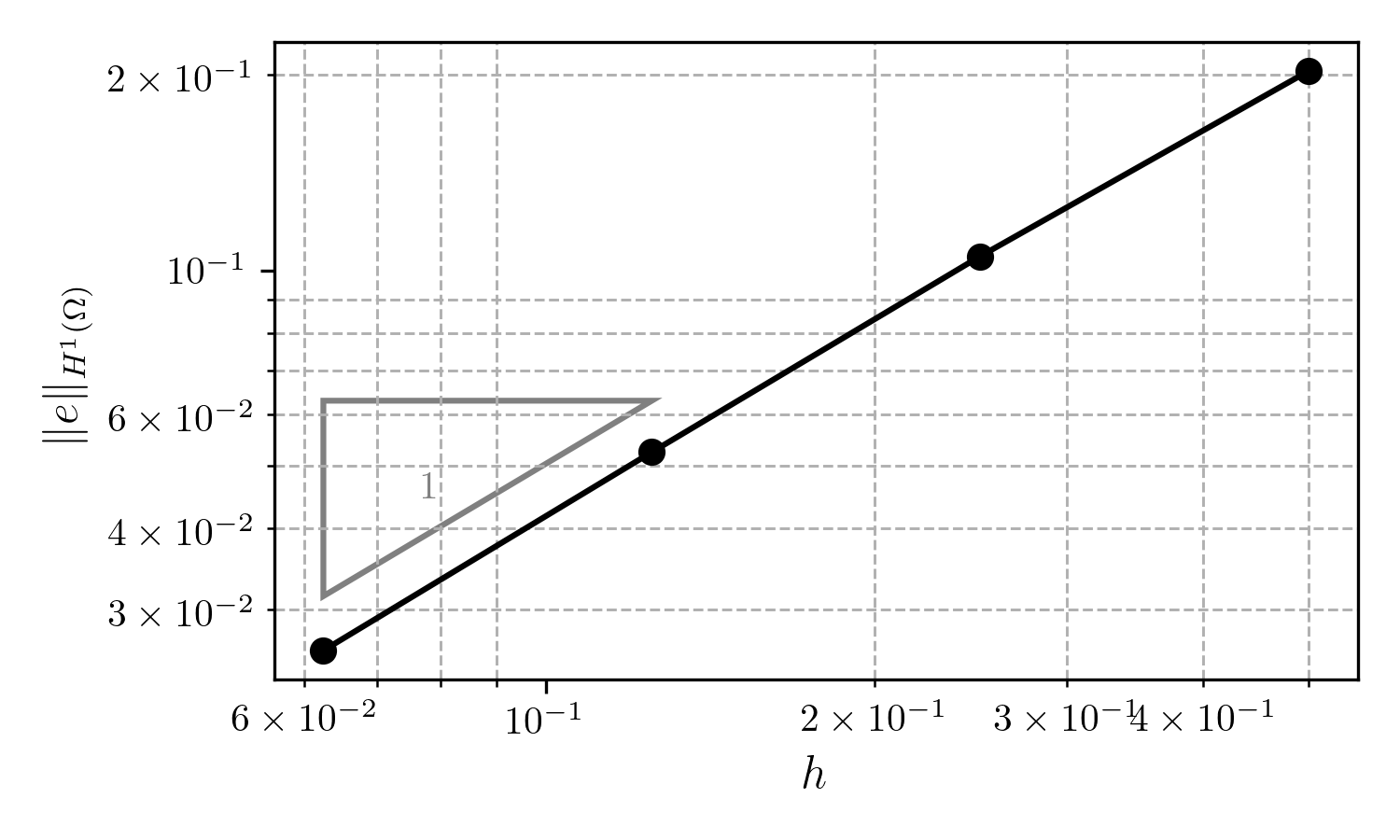}\hfill
  \caption{Parallelepiped example. $H^1$ norm of the errors in the solution obtained with
    tetrahedra and the regularized formulation.}
  \label{fig-para-serror}
\end{figure}

\begin{figure}[ht]
  \centering
  \includegraphics[width=0.8\textwidth]{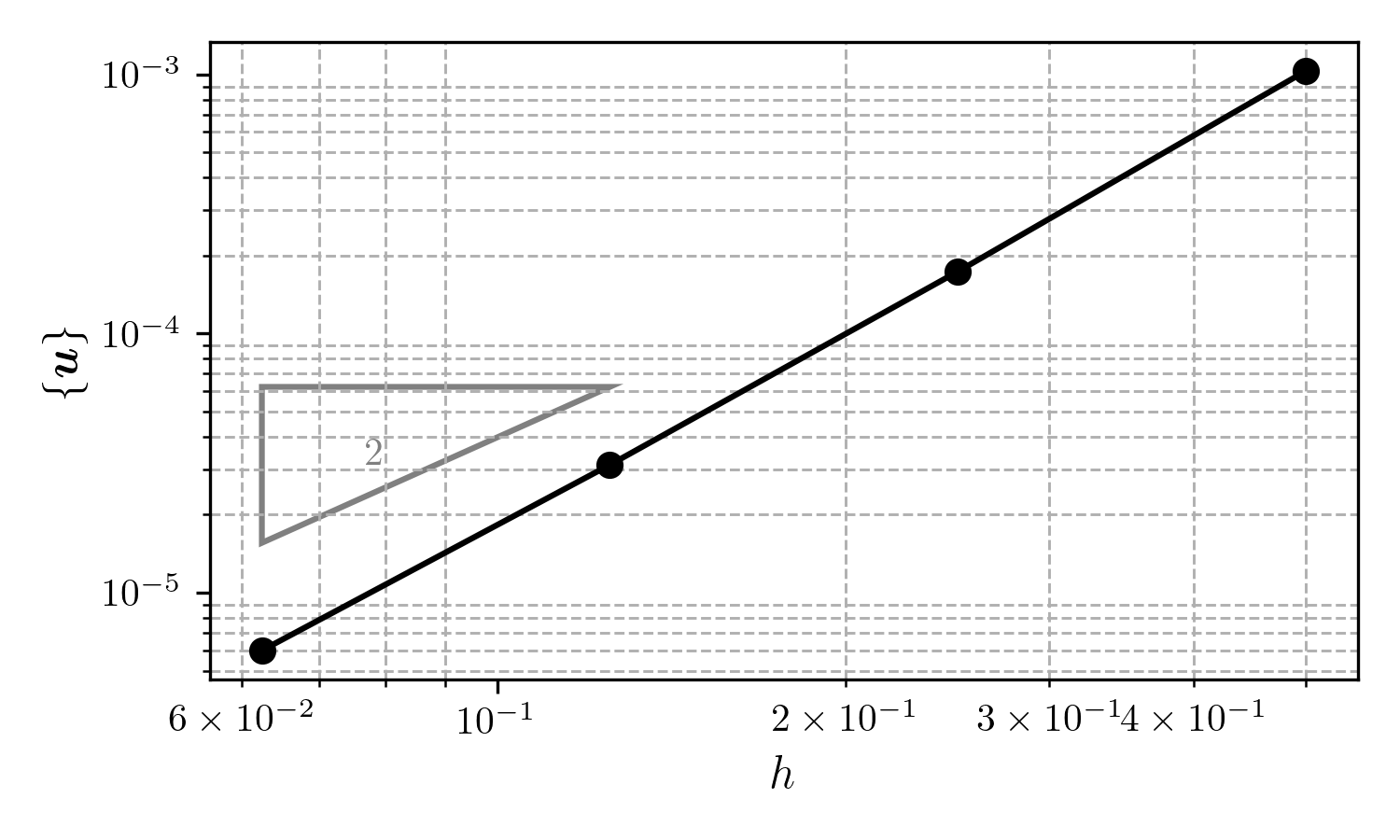}\hfill
  \caption{Parallelepiped example. Norm of the first moment of the displacement field obtained with
    the regularized formulation.}
  \label{fig-para-sjmom}
\end{figure}

\begin{figure}[ht]
  \centering
  \includegraphics[width=0.8\textwidth]{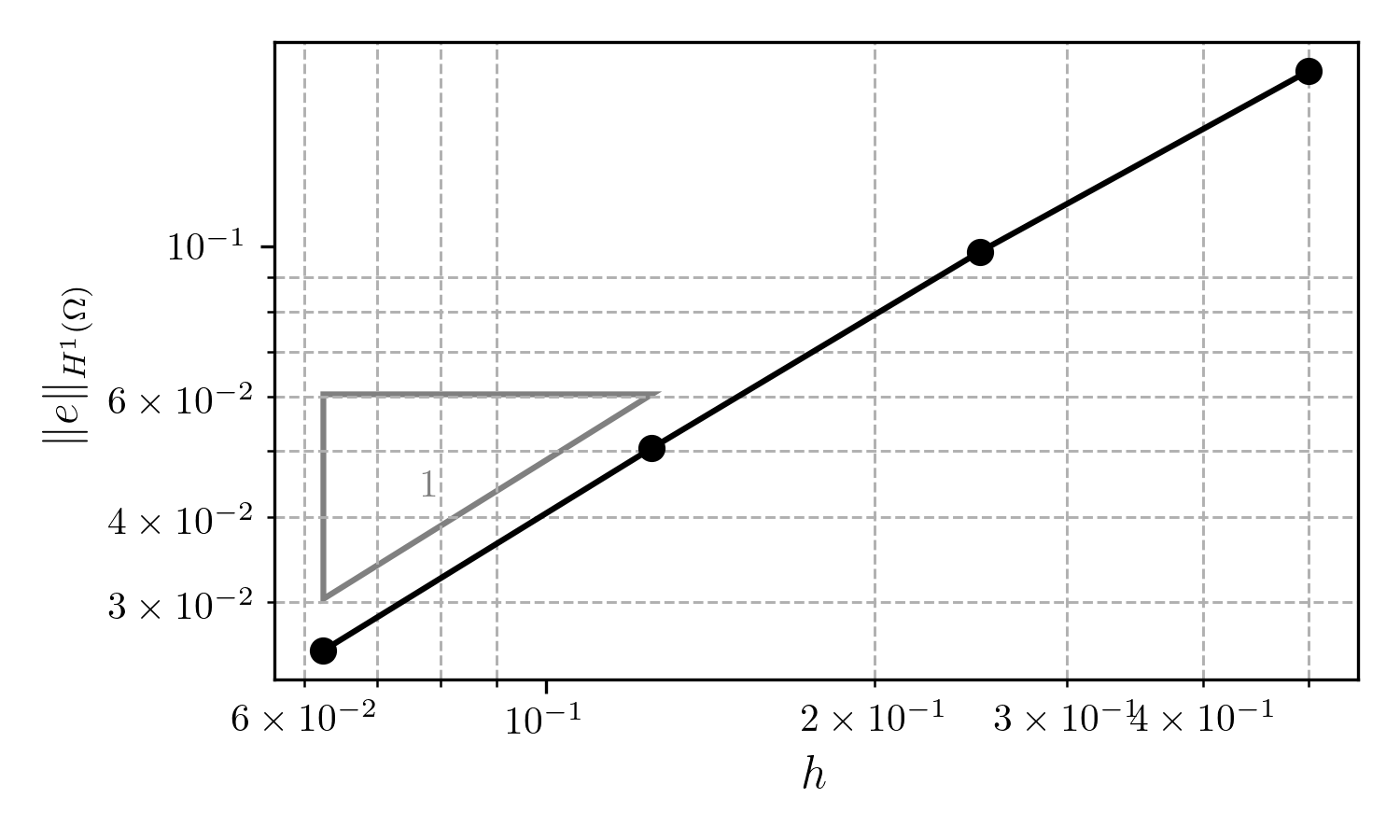}\hfill
  \caption{Parallelepiped example. $H^1$ norm of the errors in the solution obtained with
    tetrahedra and the two-step, regularized formulation.}
  \label{fig-para-rerror}
\end{figure}

The regularized method introduced in Section~\ref{sec-regularized} yields (convergent) solutions to
Neumann problems such as this one, even without any constrained degrees of freedom (see Figure~\ref{fig-para-serror}). With this method, however, the solutions obtained with the tetrahedral meshes are not centered (see Figure~\ref{fig-para-sjmom}). To remedy this undesirable effect derived from the anisotropy of the discretization, the two-step method proposed in Section~\ref{subs-two-step} can be employed. When using it, the solution obtained is convergent (see Figure~\ref{fig-para-rerror}), but also perfectly centered without constraining any degrees of freedom, and for all meshes.

\subsection{Imposing periodic boundary conditions}
\label{subs-periodic}
An import class of Neumann problems appears frequently in micromechanical computations. There, to calculate the homogenized properties of heterogeneous materials, it suffices to study representative volume elements (RVE), assuming there is a clear separation of scales, and subject them to macroscopic deformations. Moreover, to enforce the Hill-Mandell\footnote{This condition states that the macroscopic stress power must equal the volume average of the microscopic stress power over an RVE, thus warranting energy consistency across scales during computational homogenization.} conditions on these elements, it is often chosen to impose periodic boundary conditions on the latter. 

Periodic boundary conditions on a cubic domain $\Omega$ are imposed as follows. Let the boundary $\partial\Omega$ be partition into disjoint sets $\partial\Omega^+$ and $\partial\Omega^-$. For simplicity, $\partial\Omega^+$ often consists of three non-parallel faces of $\partial\Omega$, and $\partial\Omega^-$ contains the three remaining ones. For this choice, there exists a projection operator $\Pi:\partial\Omega^+\to\partial\Omega^-$ that maps every point on $\partial\Omega^+$ to a point on the opposing face. If $\mbs{u}$ denotes the displacement of the body and $\mbs{t}$ the traction on the surface, periodic boundary conditions on $\Omega$ are defined as:
\begin{equation}
\label{eq-periodic}
\mbs{u}(\mbs{x}) - (\mbs{u}\circ\Pi)(\mbs{x}) = \bar{\mbs{\varepsilon}}
(\mbs{x}-\Pi(\mbs{x}))\ ,
\qquad
\mbs{t}(\mbs{x}) + (\mbs{t}\circ\Pi)(\mbs{x}) = \mbs{0}\ ,
\end{equation}
for all $\mbs{x}\in\partial\Omega^+$ and some macroscopic strain~$\bar{\mbs{\varepsilon}}$.
Notice that the elasticity problem with periodic boundary conditions is of pure Neumann type and
therefore its numerical solution demands one of the strategies identified in Section~\ref{sec-intro}.

\begin{figure}[p]
  \centering
  \includegraphics[width=0.49\textwidth]{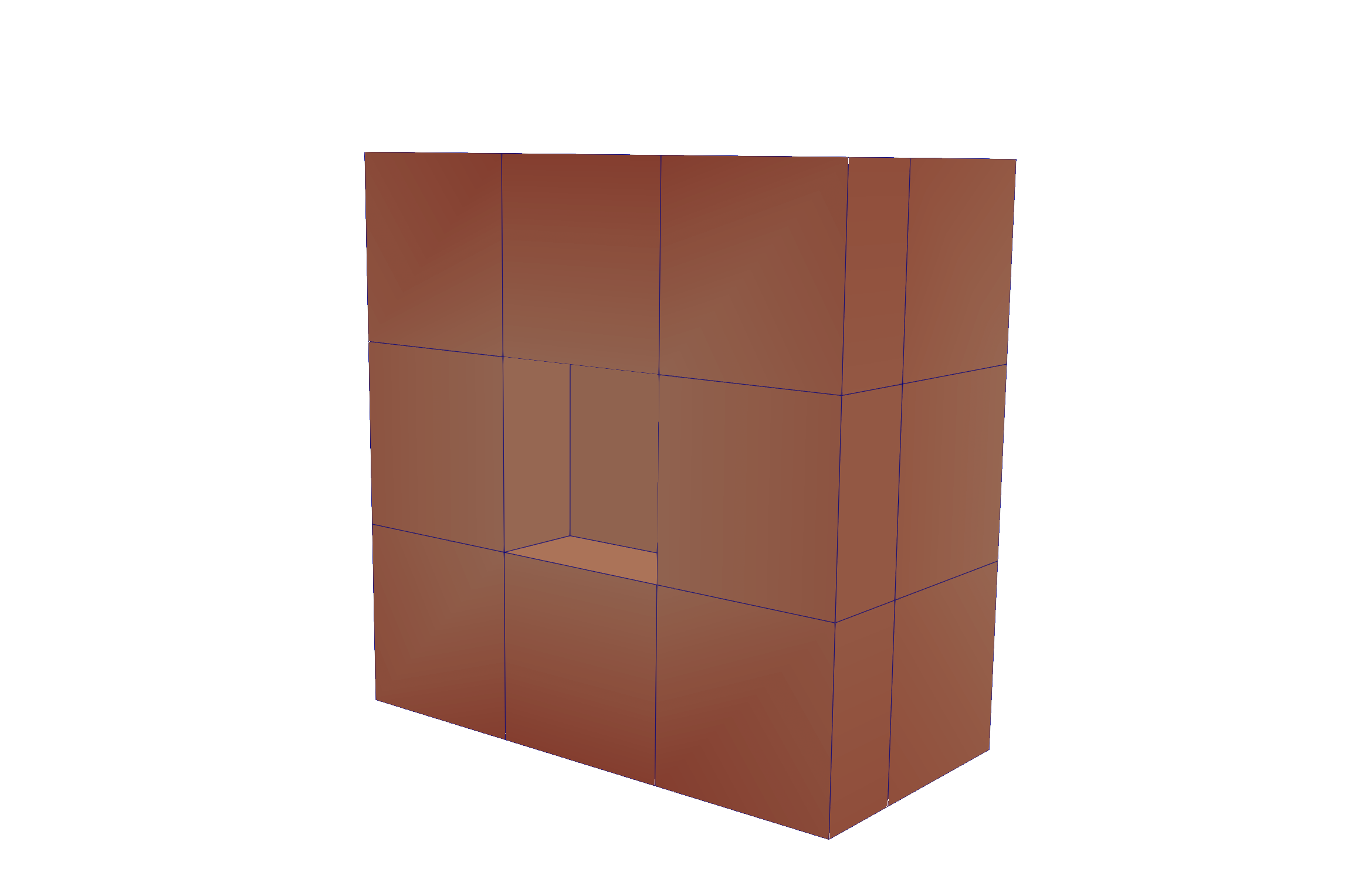}
  \includegraphics[width=0.49\textwidth]{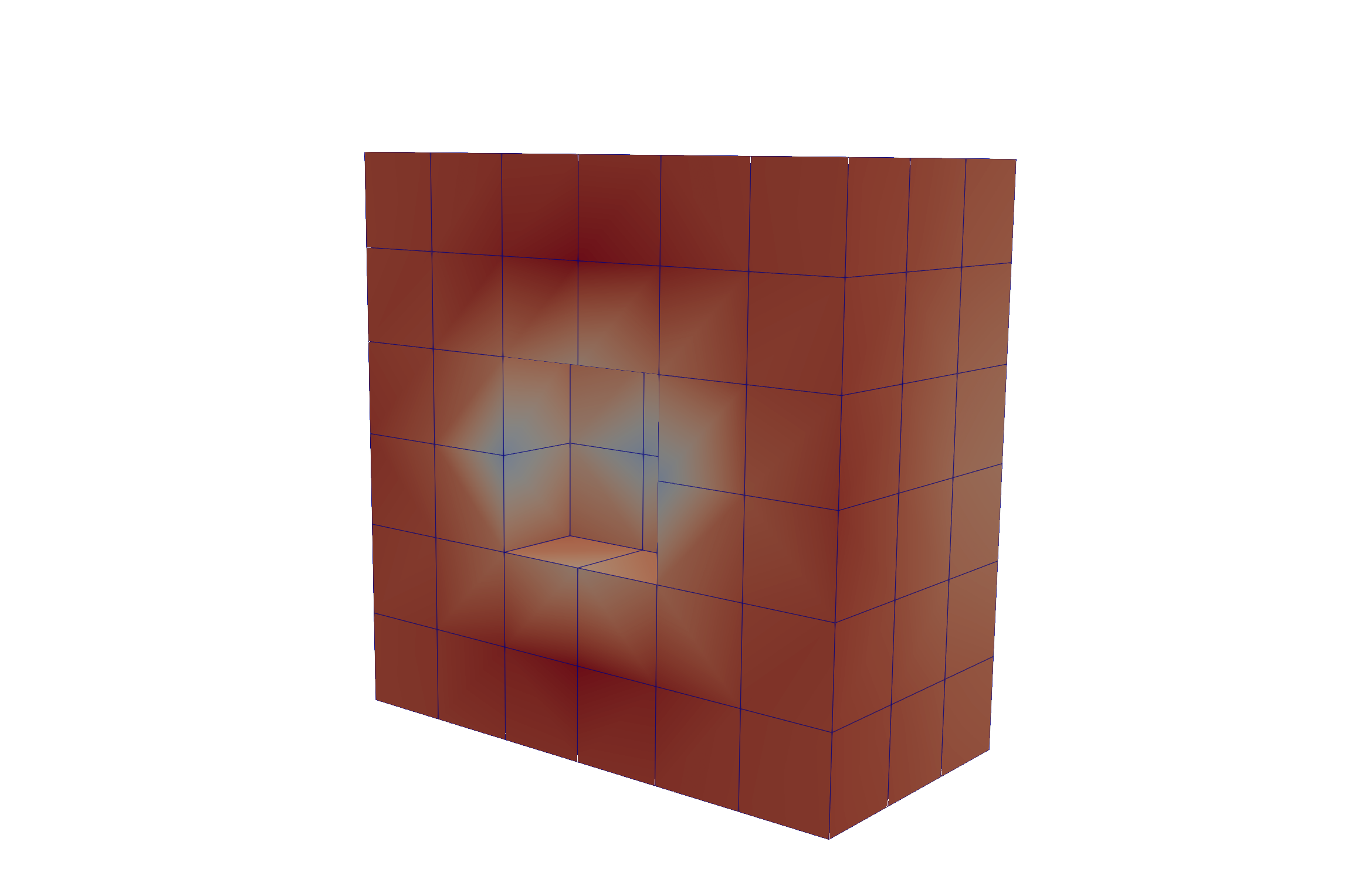}
  \includegraphics[width=0.49\textwidth]{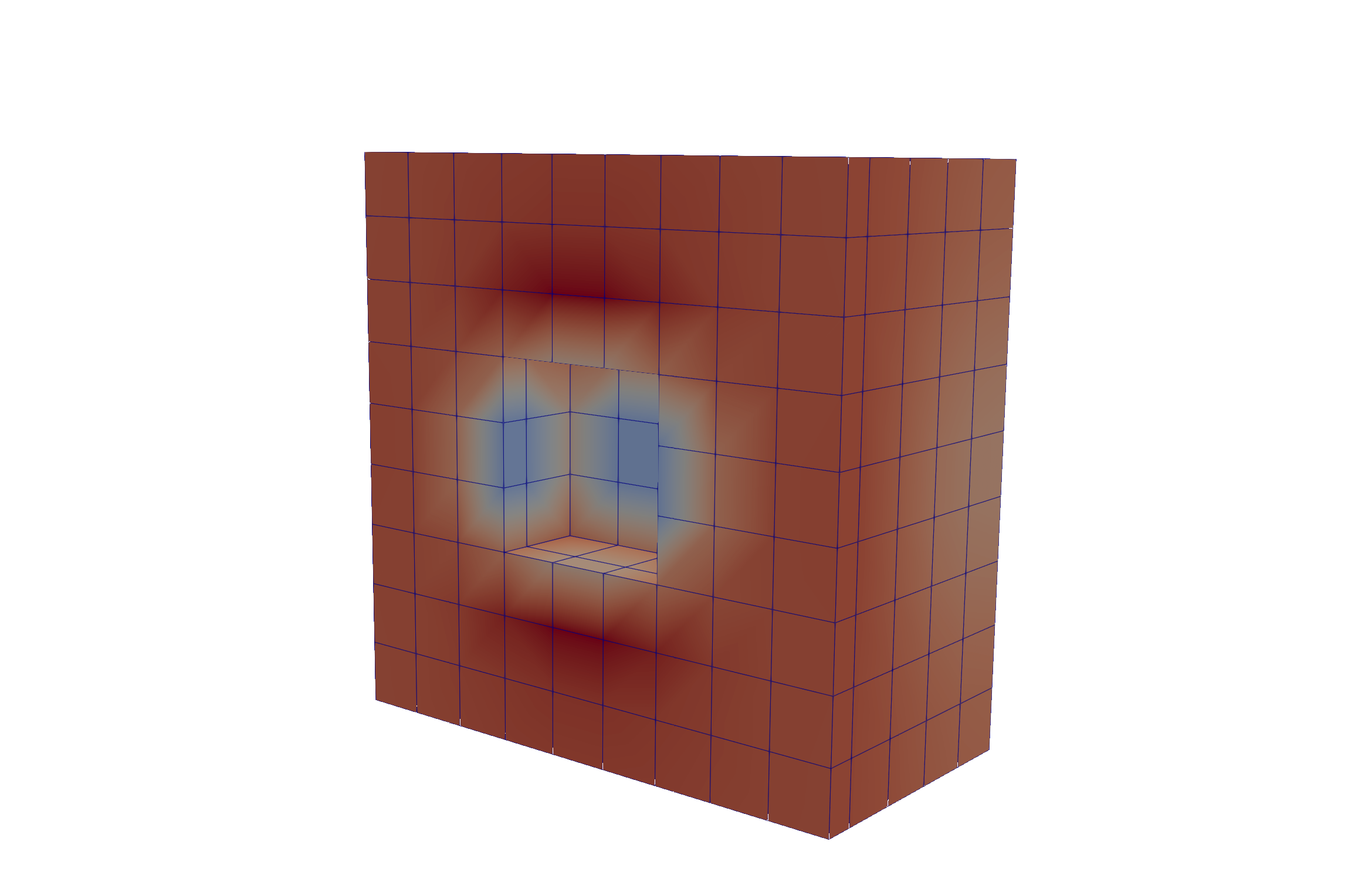}
  \includegraphics[width=0.49\textwidth]{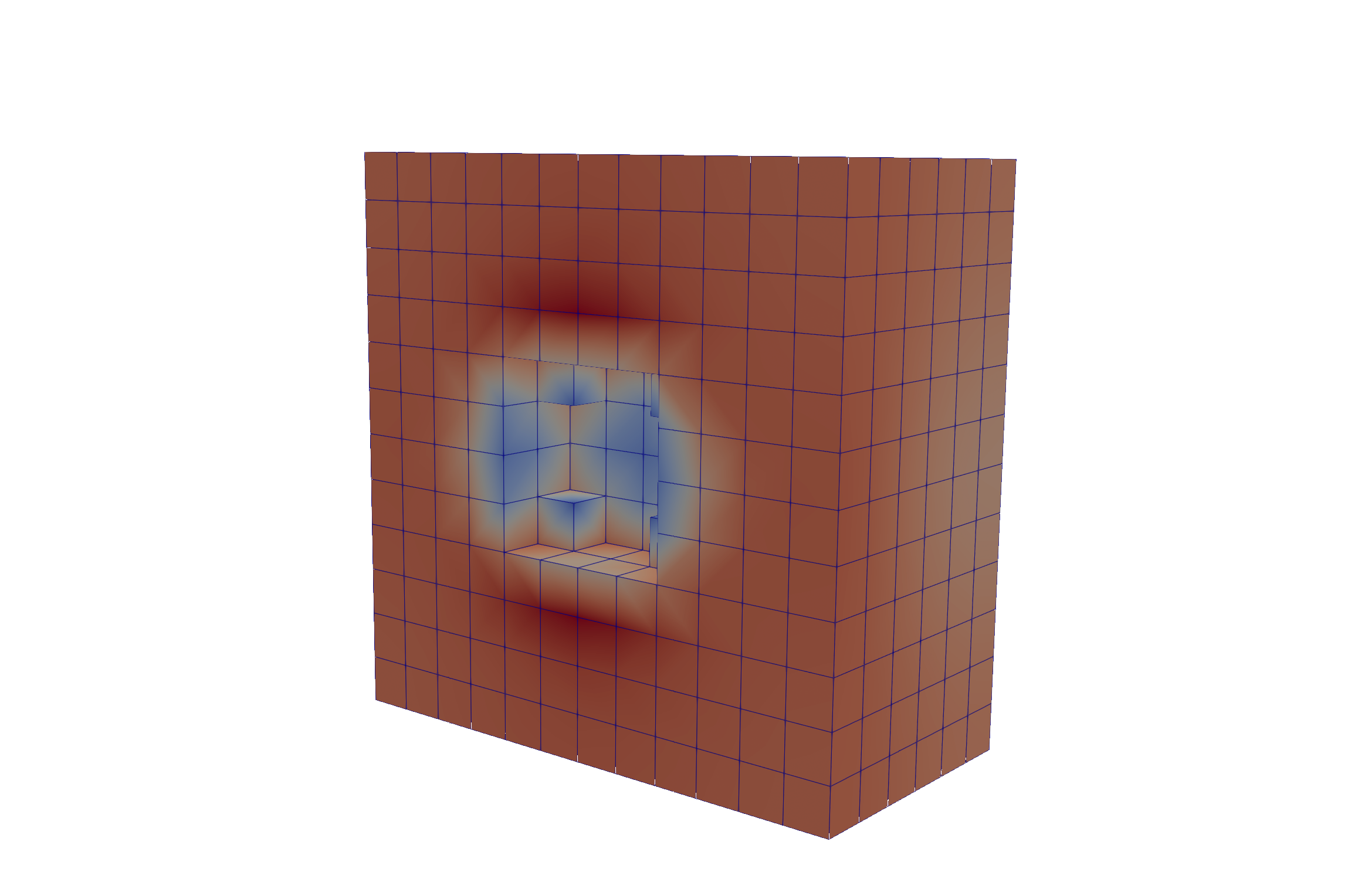}
  \includegraphics[width=0.49\textwidth]{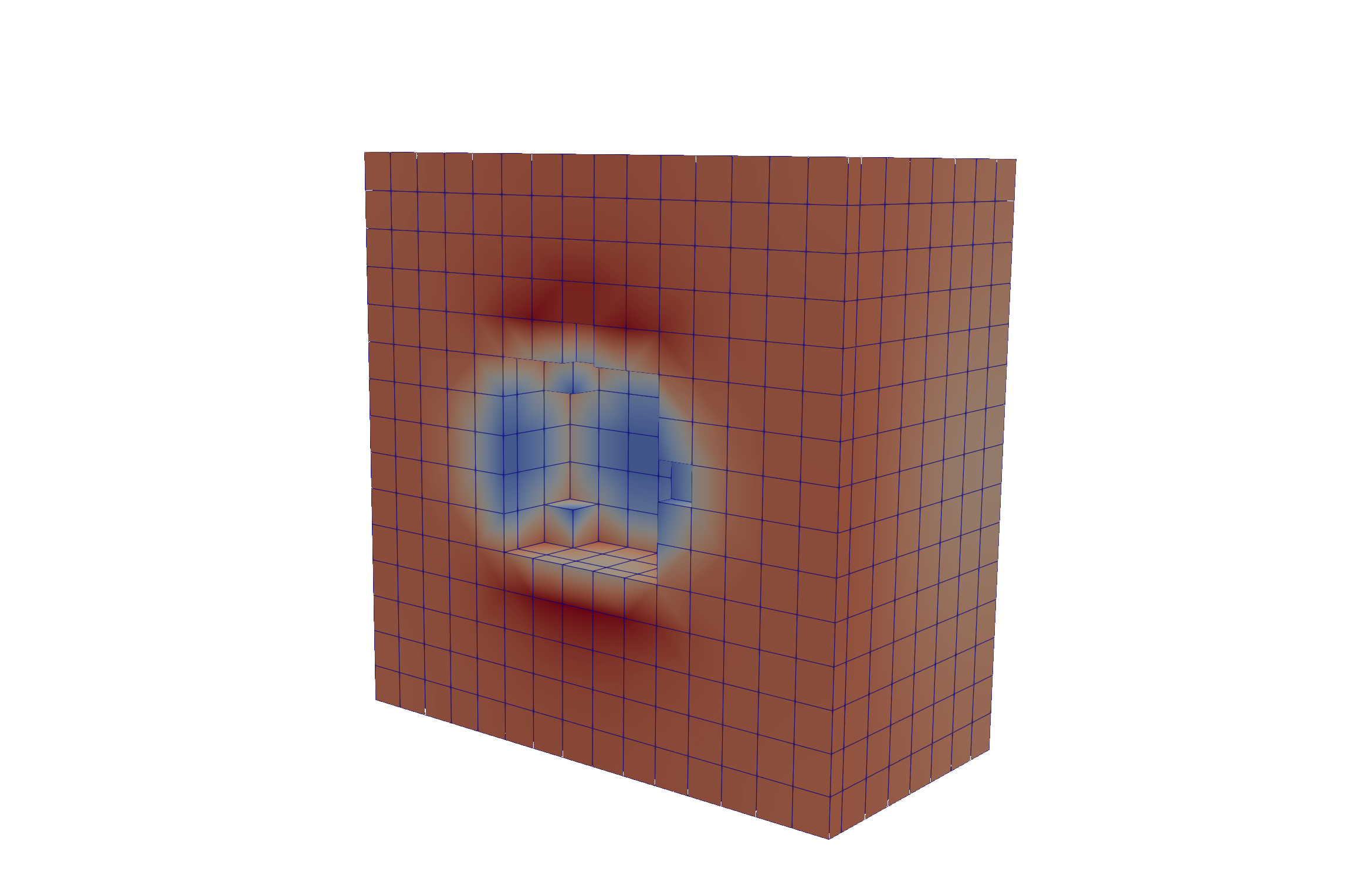}
  \includegraphics[width=0.49\textwidth]{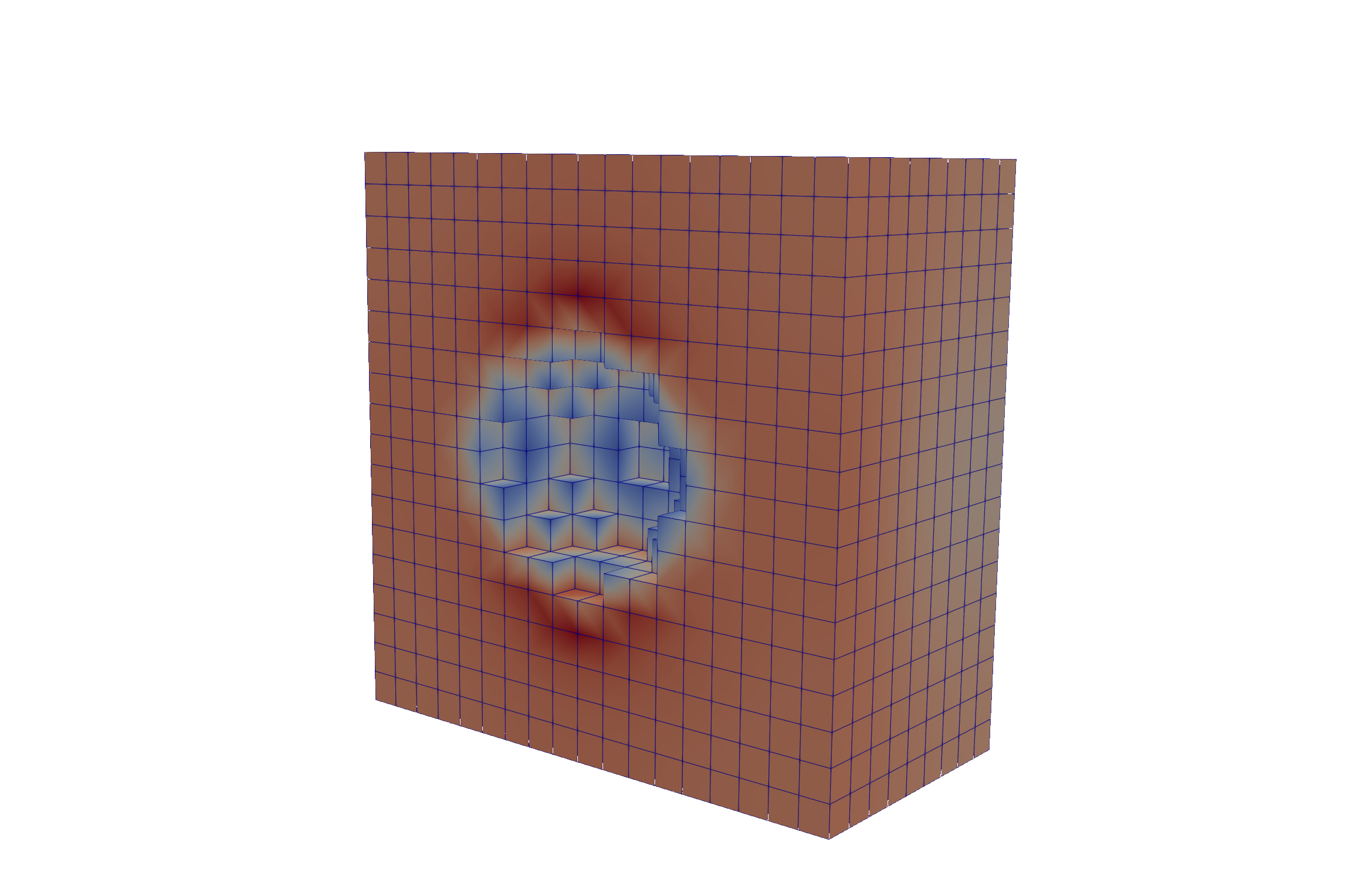}
  \includegraphics[width=0.8\textwidth]{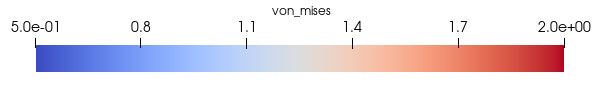}
  \caption{Cross section of RVEs with soft spherical inclusion under imposed shear strain and
    periodic boundary conditions. Color indicates von Mises stress. From left to right, top to bottom,
    meshes have size $(3\,k)^3,\ k=1,2\ldots,6$.}
  \label{fig-periodic}
\end{figure}

Several strategies have been proposed in the literature that can be used to implement periodic
boundary conditions on finite element discretization, and all must remove the (infinitesimal) rigid
body motions of the Neumann problem from the periodic solution (see, i.e.,
\cite{michel1999es,miehe2002zs,nguyen2011wd,henys2019yn}). Here, we will study an RVE under periodic boundary conditions implementing Eqs.~\eqref{eq-periodic} and removing the Neumann singularity by employing the regularization proposed in Section~\ref{sec-regularized}.

We consider now the computation of the shear modulus of a cubic elastic RVE with side~$a=1$ inside
of which a spherical inclusion of radius $r=a/5$ is located. To make the results reproducible, we
use finite element discretizations of the voxelized cube, using an increasing number of subdivisions
in the domain. For the matrix and the inclusion we employ elastic materials with Lam\'e parameters
$\lambda=\mu=1$ and $\lambda=\mu=10^{-4}$, respectively. A classical estimate~\cite{mackenzie1950aa}
of the shear stiffness for the homogenized material is
\begin{equation*}
  \label{eq-periodicC}
  C = \mu \left( 1 - \frac{15 (1-\nu)}{7-5\nu} p \right)\ ,
\end{equation*}
where $\nu=\lambda/2/(\lambda+\mu)=0.25$ is the Poisson ratio of the material and
\begin{equation}
  p = \frac{4}{3}\frac{\pi r^3}{a^3}
  \end{equation}
  is the porosity of the RVE. Eq.~\eqref{eq-periodic} is a first order approximation of the change
  of shear stiffness in terms of the porosity which, for the selected properties of the materials and inclusion, evaluates to $C=0.934436327$.

To verify the use of the proposed formulation, we use RVEs with $(3\,k)^3$ elements, $k=1,2,\ldots,6$.
For each mesh, we impose periodic boundary conditions as indicated above, using only
$\epsilon_{xy}=1$, and we eliminate the rigid body motions by regularizing the elasticity operator
as proposed in Section~\ref{sec-regularized} with a parameter $\eta=10^{-2}\,h^{1.5}$. Then, we
calculate the sum of the reactions on a face normal to the $x$ direction and project it onto the $y$
direction to obtain an approximation to $C$. Table~\ref{tab-periodic} summarizes the results of the
computation, and we observe that the numerically obtained value converges to the
estimate~\eqref{eq-periodic}. Since the latter is only an approximation of the true value of the
shear stiffness, a detailed analysis of the convergence rate of these computations is not relevant. At most, one can expect that the computed stiffness remains close to the estimate~\eqref{eq-periodic}.

\begin{table}[ht]
  \begin{center}
    \begin{tabular}{cll}
      \toprule
      $k$ & \multicolumn{1}{c}{$C_h$} & \multicolumn{1}{c}{Error (\%)} \\
      \midrule
1 & 0.94890  & 1.54785 \\ 
2 & 0.93130  & 0.33547 \\
3 & 0.92409  & 1.10714 \\ 
4 & 0.93455  & 0.01260 \\ 
5 & 0.92549  & 0.95749 \\ 
6 & 0.92889  & 0.59318 \\ 
      \bottomrule
\end{tabular}
  \end{center}
  \caption{Approximate shear stiffness $C_h$ obtained with regular meshes of size $(3\,k)^3$ and periodic boundary conditions. Error with respect
  to $C$ provided in Eq.~\eqref{eq-periodic}.}
  \label{tab-periodic}
\end{table}

\subsection{Body with complex geometry under imposed temperature field}
\label{subs-bunny}
In this last example, we study a pure traction problem in
a solid with a complex geometry. Specifically, let us consider
an isotropic thermoelastic solid with the geometry of the
\emph{Stanford bunny} \cite{bunny}, a model often employed for evaluating
the performance of algorithms for computer graphics.
We will generate three different meshes and impose a linear temperature
field on the body. Due to thermoelastic effects, the bunny will
deform but no stresses should appear.

The material that we have employed has Young's modulus $E=1$, Poisson's
ratio $\nu=0.3$, and thermal expansion coefficient $\alpha=1$. The imposed
thermal increment is $\vartheta=x+y+z$, where the latter refer to the
Cartesian coordinates employed in the original bunny model.

\begin{figure}[ht]
  \centering
  \includegraphics[width=0.32\textwidth]{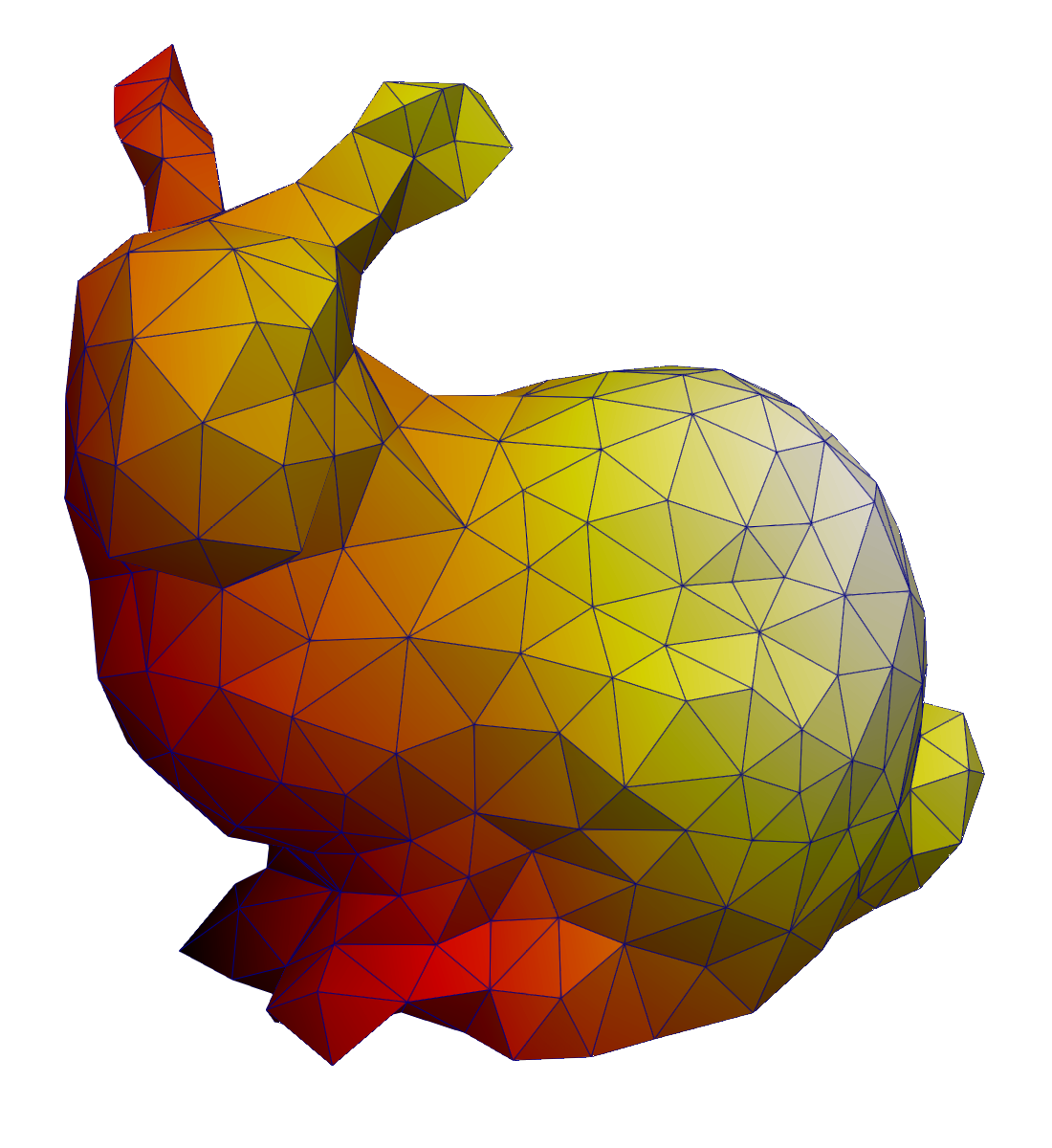}
  \includegraphics[width=0.32\textwidth]{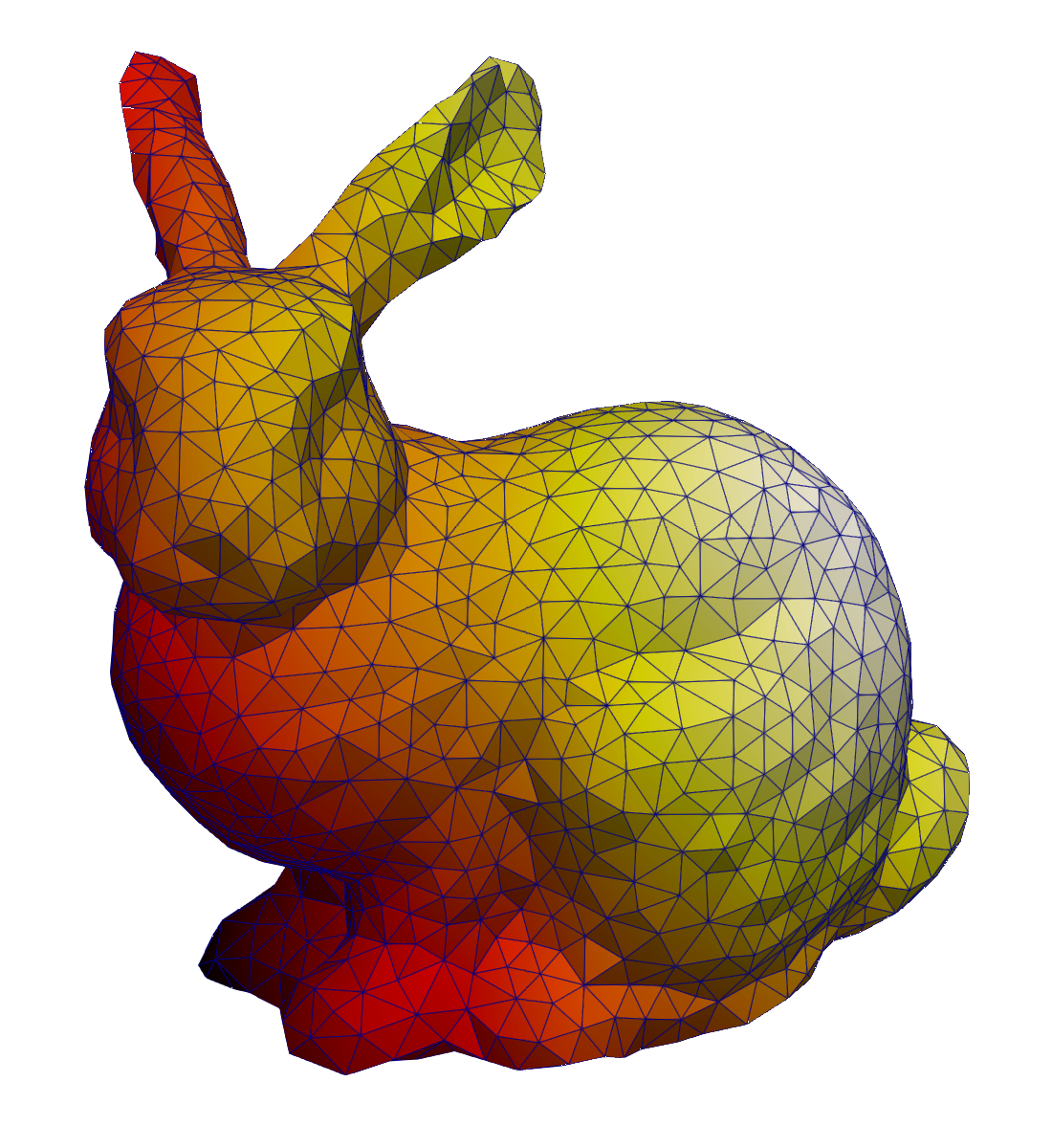}
  \includegraphics[width=0.32\textwidth]{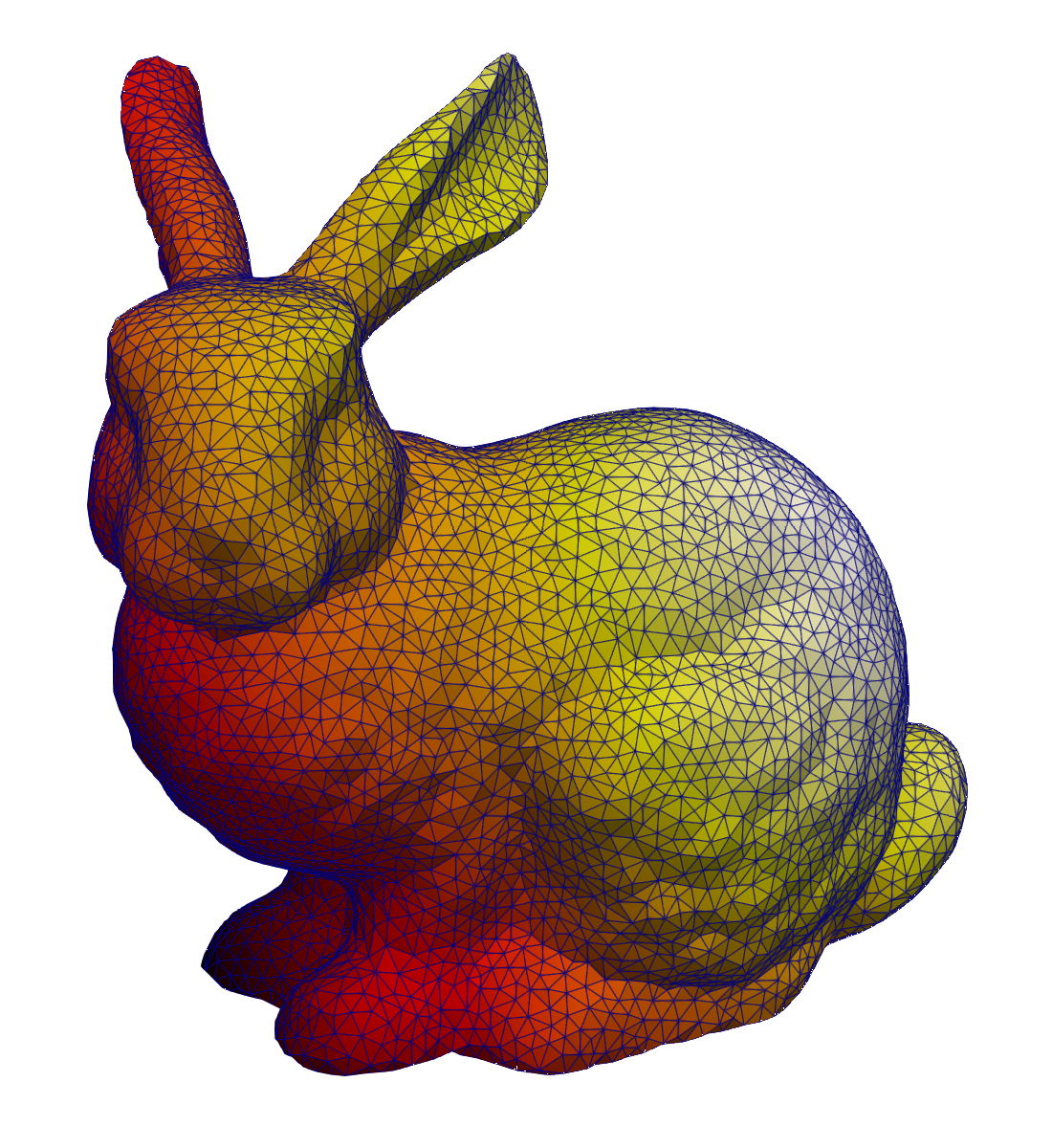}
  \includegraphics[width=0.75\textwidth]{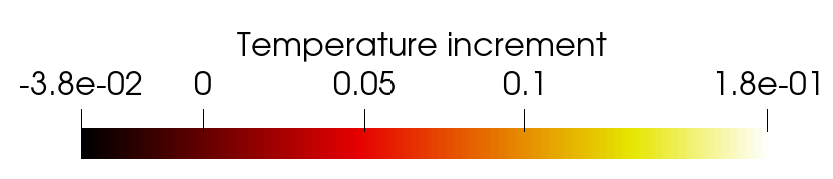}
  \caption{Temperature field of the form $\vartheta=x+y+z$ imposed
    on the coarse, intermediate, and fine meshes of the Stanford
  bunny.}
  \label{fig-bunnies}
\end{figure}

\begin{figure}[p]
  \centering
  \includegraphics[width=0.32\textwidth]{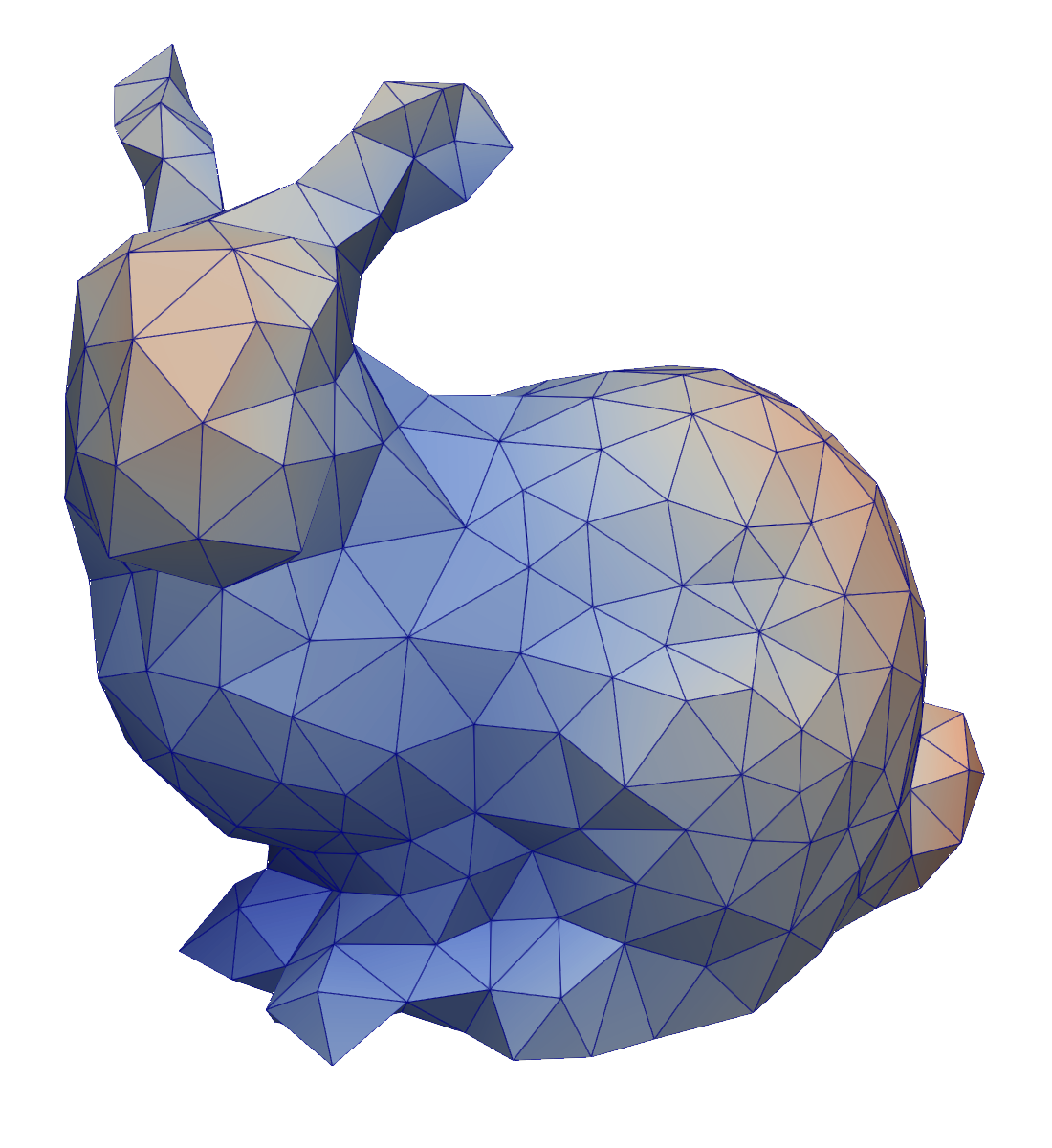}
  \includegraphics[width=0.32\textwidth]{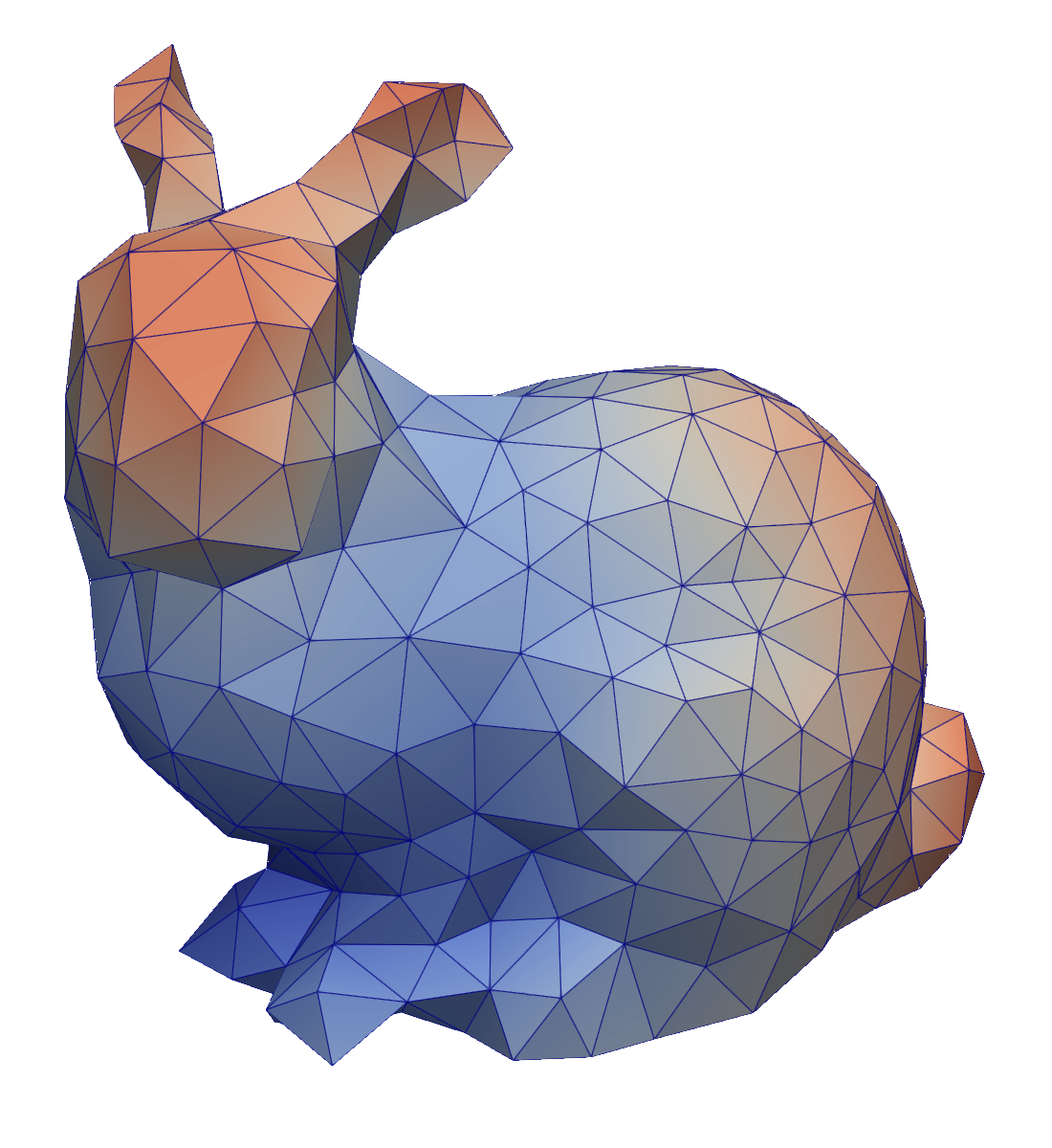}
  \includegraphics[width=0.32\textwidth]{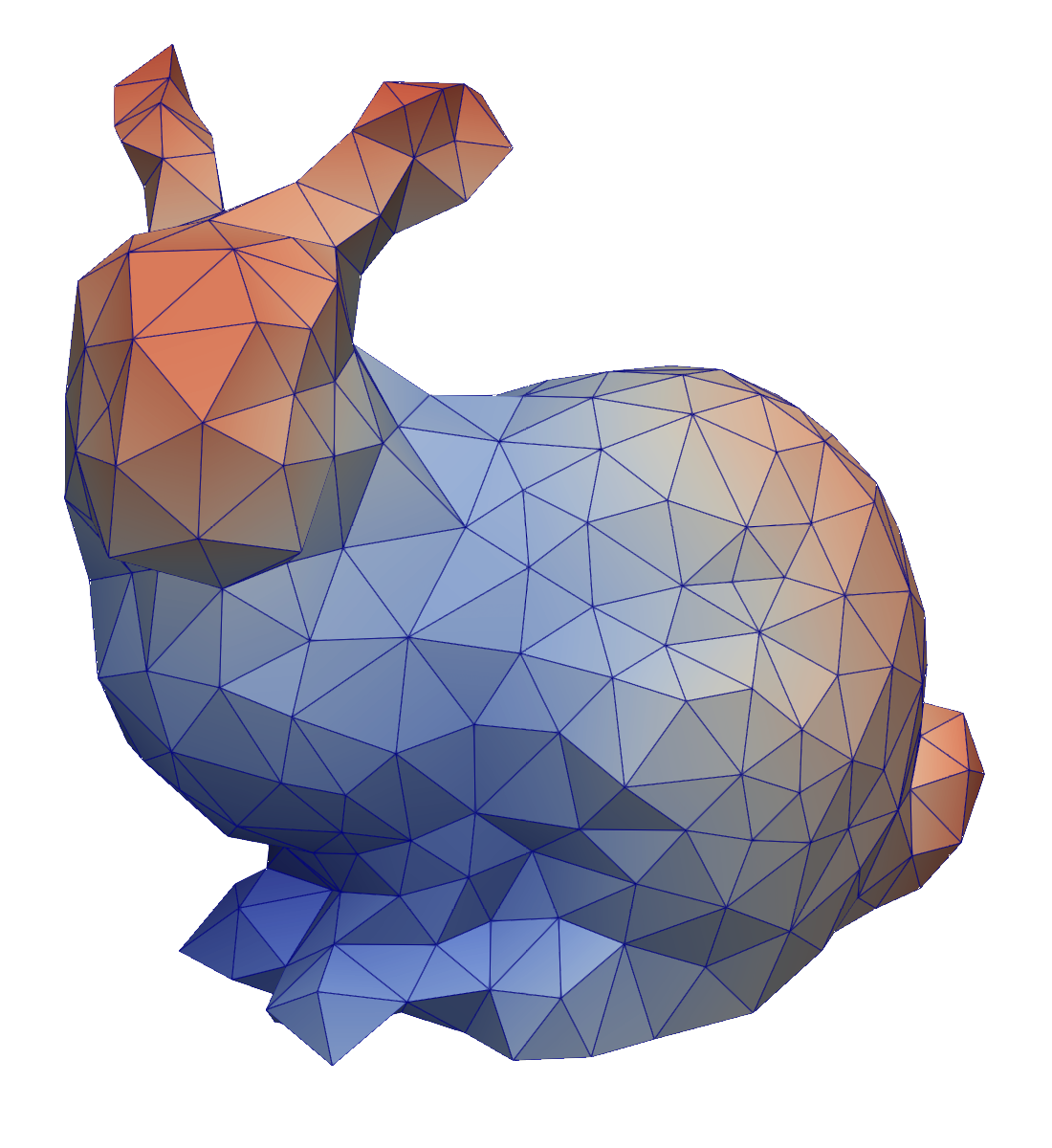}
  \includegraphics[width=0.32\textwidth]{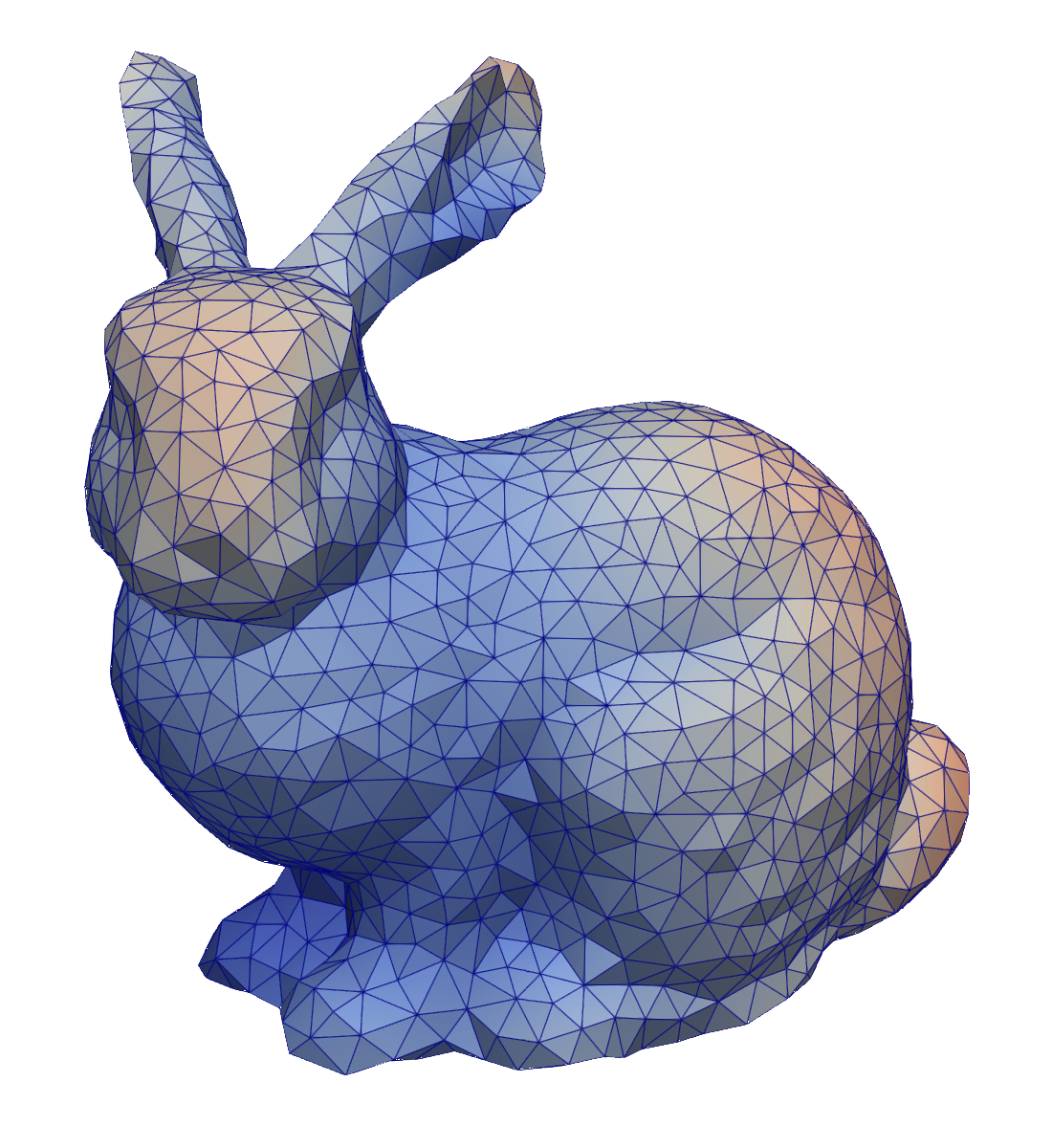}
  \includegraphics[width=0.32\textwidth]{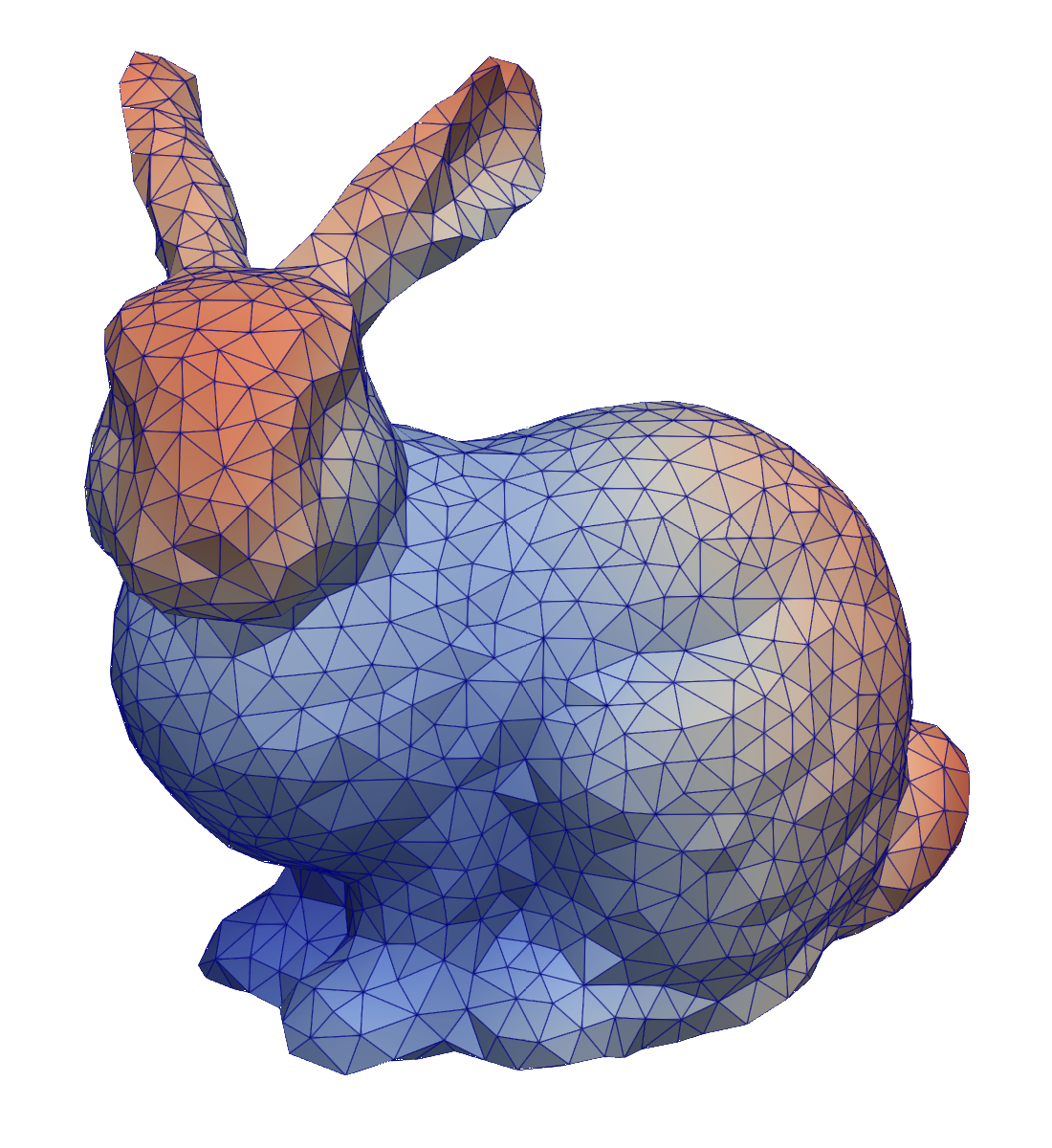}
  \includegraphics[width=0.32\textwidth]{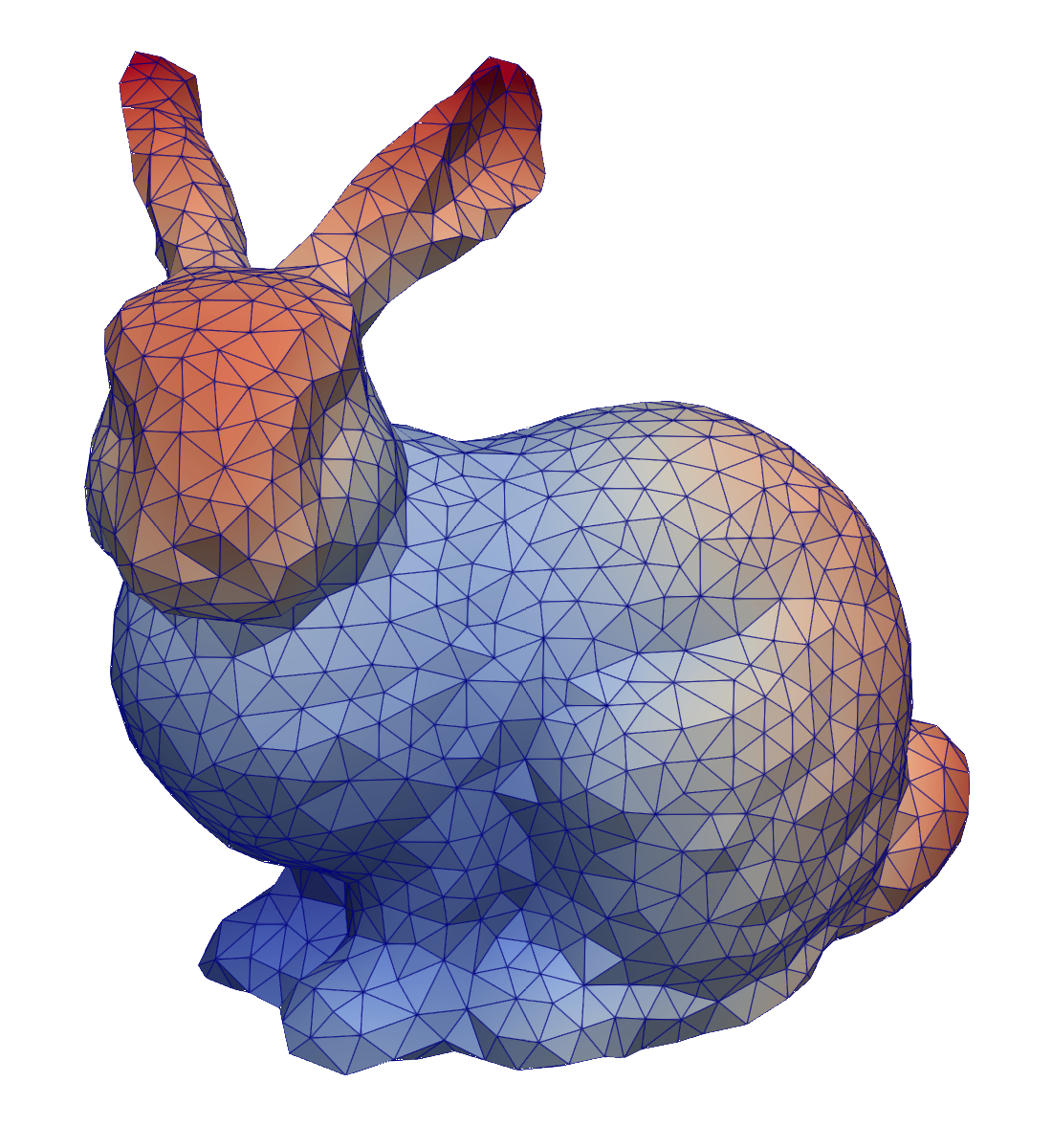}
  \includegraphics[width=0.32\textwidth]{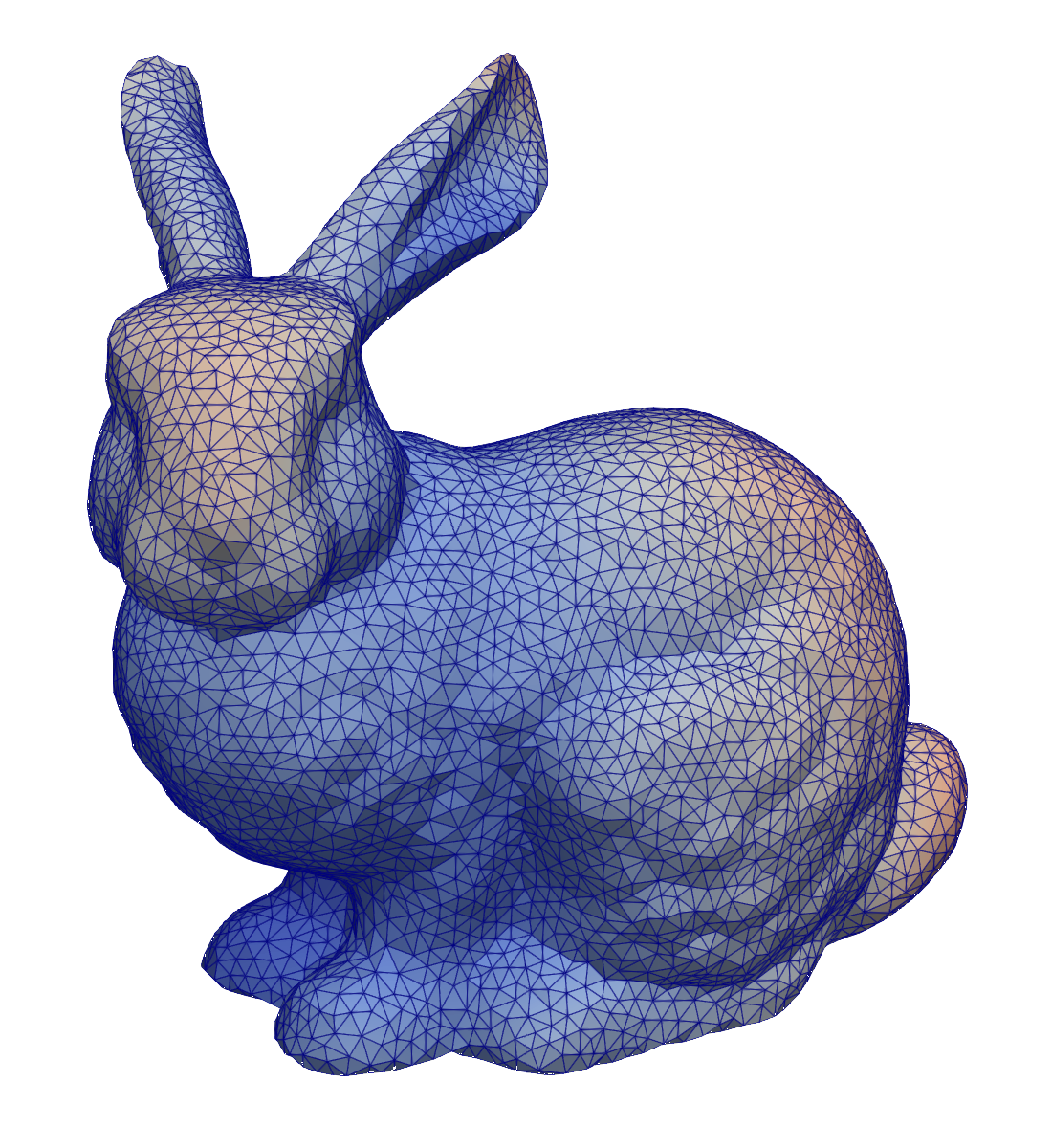}
  \includegraphics[width=0.32\textwidth]{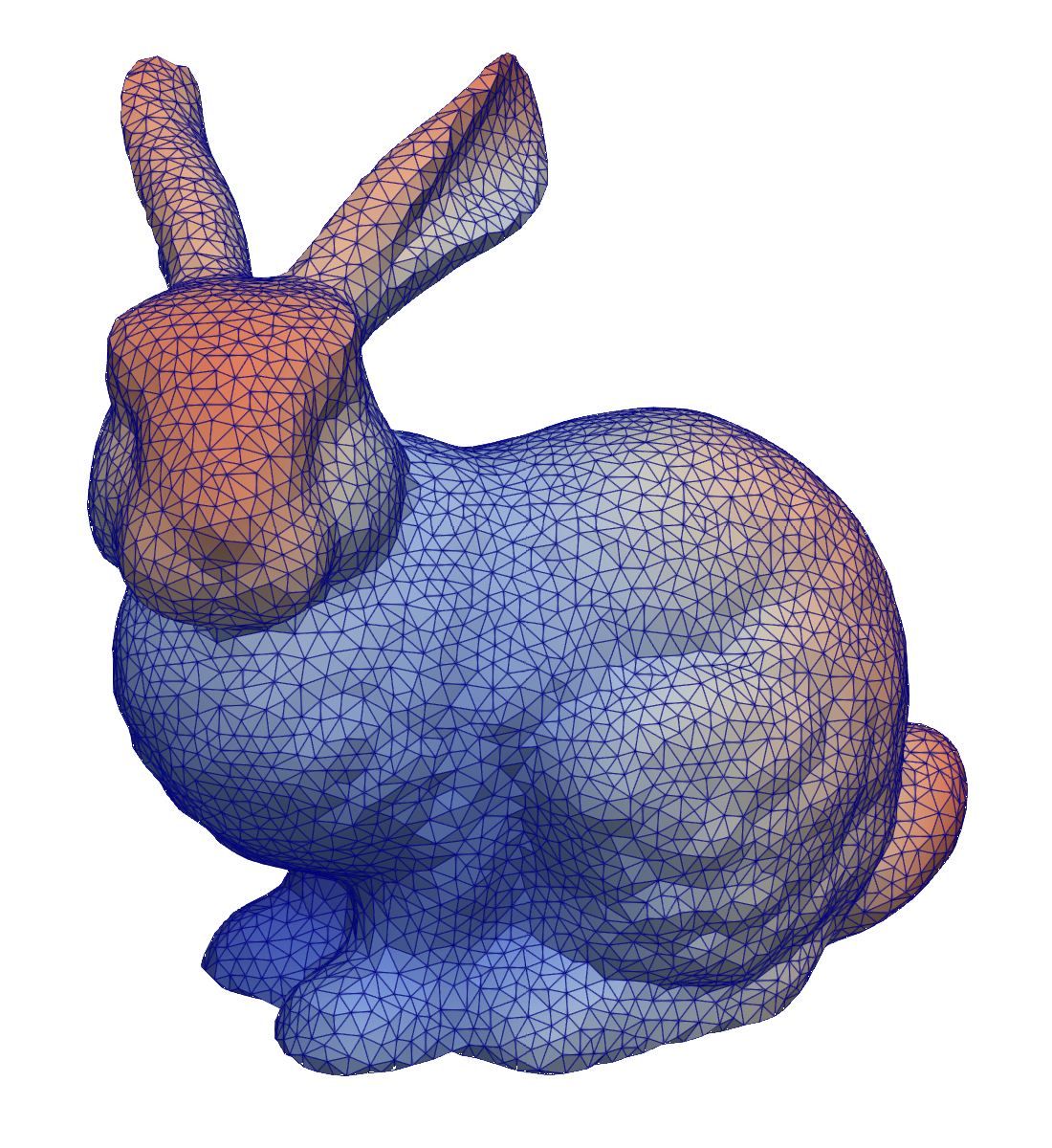}
  \includegraphics[width=0.32\textwidth]{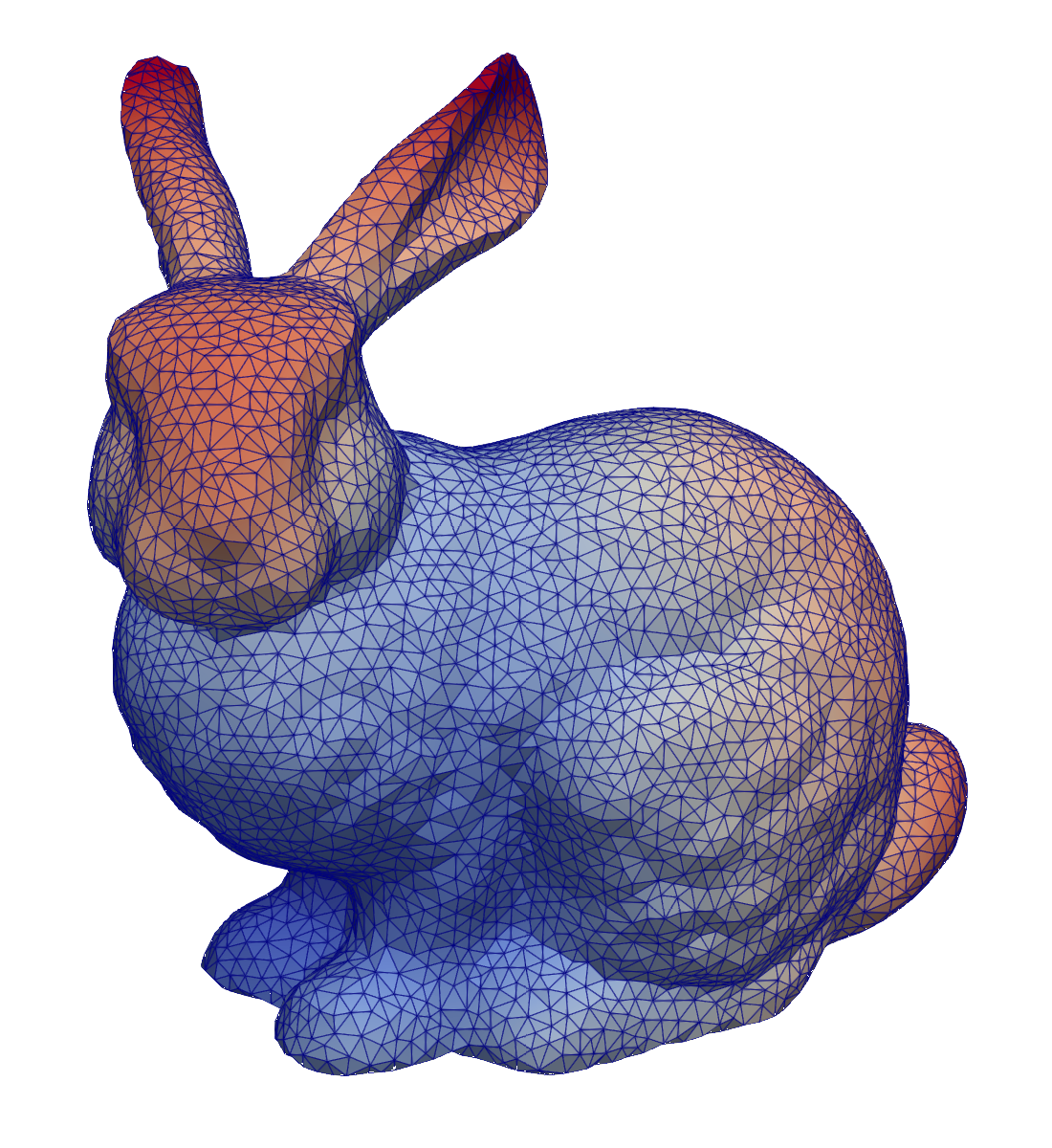}
  \includegraphics[width=0.95\textwidth]{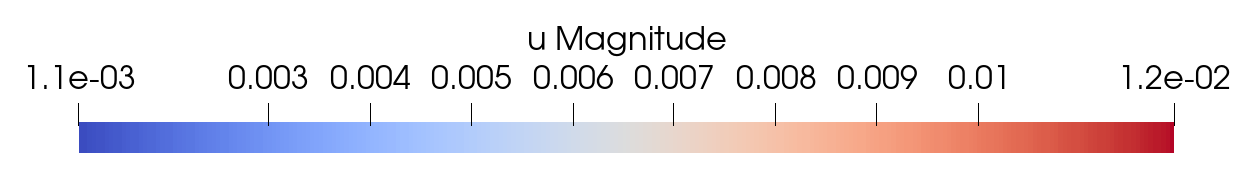}
  \caption{Deformation fields on the thermoelastic bunny-like solid.
    Top row: coarse mesh; center row: intermediate mesh; bottom row: fine mesh.
  Left column: $\bar{\eta}=1$; center column: $\bar{\eta}=0.1$; right column: $\bar{\eta}=0.01$.}
  \label{fig-bunny-u}
\end{figure}

Figure~\ref{fig-bunnies} shows the imposed temperature field
on each of the three meshes considered. The small, medium,
and large meshes have, respectively, 2319, 19138, and 152706 elements.
The problem has no Dirichlet boundary condition, and, given
the complexity of the model geometry, it is a
clear example where imposing constraints to preclude rigid body
motions, \emph{in all meshes}, is non-trivial.

\begin{figure}[ht]
  \centering
  \includegraphics[width=0.9\textwidth]{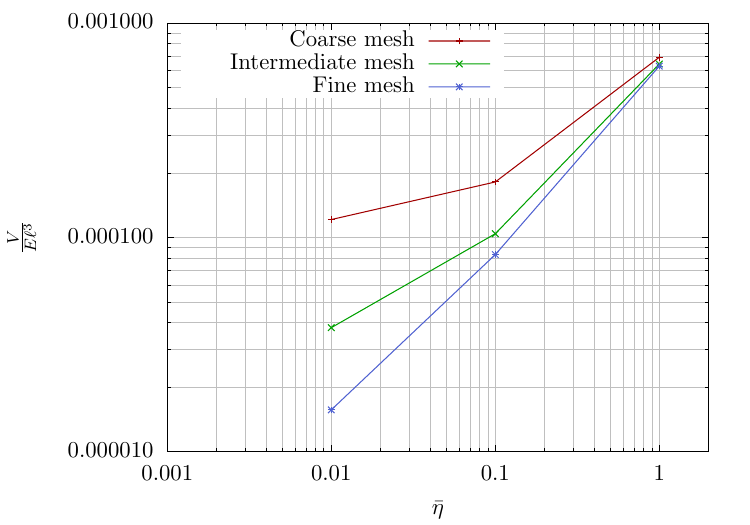}
  \caption{Bunny-like thermoelastic solid. Normalized potential energy of the solution in the
    regularized method as a function of the mesh size
  and the regularizing parameter. The exact solution has zero potential energy.}
  \label{fig-bunny-conv}
\end{figure}

We solve the Neumann problem for the three meshes using the regularized formulation presented in this work, and we study the
effects of mesh size and the regularizing parameter.
Figure~\ref{fig-bunny-u} shows the displacement
field for three different meshes and $\bar{\eta}=1,0.1$, and $0.01$, where we
have selected $\ell=$0.1 as the characteristic length of the solid.
This figure shows that as the regularizing parameter is reduced,
the displacement on the body is less constrained, thus the
maximum values increase.

As advanced, the theory predicts that a linear temperature field
should create no stresses on an isotropic thermoelastic solid if it
has no Dirichlet boundary conditions. The regularization that we have
introduced spoils this property, and thus some stored energy
is to be expected. Nevertheless, this energy should vanish as $\bar{\eta}\to0$. Figure~\ref{fig-bunny-conv}
shows the value of the (normalized) potential energy of the solid as
the regularizing parameter is reduced, for the three meshes employed in
Figures~\ref{fig-bunnies} and~\ref{fig-bunny-u}. This energy monotonically
decreases when either the mesh size or $\bar{\eta}$ goes to zero, indicating
that the method recovers the exact, stress-free, configuration in the limit.

\section{Summary}
\label{sec-summary}
The Neumann problem of elasticity appears often in Engineering practice. Assuming
that the loading satisfies certain compatibility conditions, there are infinite
rigid body equivalent displacement fields that solve this boundary value problem, which
causes troubles with its numerical approximation. In finite element analysis, standard fixes to this
lack of uniqueness are computationally costly or not robust.

In this work, we have introduced a perturbed Neumann problem whose solution is unique and
arbitrarily close to the solution of the original problem, more precisely, to the one of minimal
$L_2$ norm. This result opens the door to finite element approximations of the Neumann problem that
are well-posed, robust, and can be obtained by adding a negligible additional cost to the solution
of a classical Dirichlet- or mixed-boundary value problem.

We have presented rigorous bounds for the $H^1$ error between the solutions of the original Neumann
problem (or rather, the solution with minimal $L^2$ norm) and the perturbed solution. These bounds
show that this error is proportional to the regularization parameter and, therefore, arbitrarily
closed. Also, we have shown that the finite element approximations of the perturbed problem
converge to the exact solution when both the mesh size and the perturbation parameter to zero.

Neumann problems with non-equilibrated loads have no solution. It happens, however, that the
discretization of a pure Neumann problem introduces perturbations that affect critically its
numerical solution. If a constrained formulation is employed, non-zero reactions are to be expected
at the constrained degrees of freedom and the obtained solution will not be centered; if a
regularized method is employed, its solution will not be centered either and the error bound
deteriorate. To solve these problems, a predictor/corrector regularized method has been proposed that provides optimal,
centered solutions without the need to use constraints, even if the load is not equilibrated.

The main idea of this work has broader implications. In fact, one could prove that given an elliptic
problem whose solution is in a closed subspace of a larger linear space, one can perturb the
energy functional and show the convergence of the regularized solution, and its finite element
approximation, to the original one. In the linear setting, this statement can be proved using
similar arguments to the one employed in this work. In general, $\Gamma-$convergence arguments can
be used to prove these properties. The aforementioned generalizations can be found in a sequel to
the current article \cite{gebhardt2025ub}.

%--------------------------------------------------------------------------------------------------
%                                       End of the document
%--------------------------------------------------------------------------------------------------

\begin{acknowledgements}
  \myack
\end{acknowledgements}

\appendix
\section{Implementation details}
The regularized solution to the pure traction problem has a simple implementation, and it does not add
any extra degree of freedom to a finite element solution of small strain mechanics. For completeness, we add a condensed summary of the details required to implement the proposed formulation in the context of a finite element solution. We consider thus a finite element mesh with node set $\mathcal{N}$. Denoting as $N^a:\Omega\to\mathbb{R}$ the shape function associated to node $a\in\mathcal{N}$, the residual vector associated to this node is
\begin{equation}
\begin{split}
\mbs{R}^a =& \int_\Omega
\left[
\mbs{\sigma}(\mbs{x})\nabla N^a(\mbs{x})
+
\eta\, \mbs{u}^h(\mbs{x})\,N^a(\mbs{x})
- \mbs{f}(\mbs{x})\, N^a(\mbs{x})
\right]\,dV\\
&- \int_{\partial\Omega} \mbs{t}(\mbs{x}) N^a(\mbs{x})\,dA .
\end{split}
\end{equation}
where $\mbs{\sigma}=\lambda\, \mathrm{tr}[\mbs{\varepsilon}]\mbs{I} + 2\mu\,\mbs{\varepsilon}$ is the stress tensor and $\mbs{\varepsilon}=\nabla^s \mbs{u}^h$. The block of the stiffness matrix corresponding to nodes $a,b\in \mathcal{N}$ evaluates to
\begin{equation}
\begin{split}
\mbs{k}^{ab} =&
                \int_\Omega \lambda \nabla N^a(\mbs{x})\otimes \nabla N^b(\mbs{x}) \; dV
  \\
              &+
                \int_{\Omega}
\mu\left( \nabla N^b(\mbs{x})\otimes\nabla N^a(\mbs{x}) +
(\nabla N^b(\mbs{x})\cdot\nabla N^a(\mbs{x})) \mbs{I} \right) \,dV
\\
              &+
                \int_\Omega \eta\,N^a(\mbs{x})\,N^b(\mbs{x})\,\mbs{I}
\,dV
\end{split}
\end{equation}
where $\mbs{I}$ is the identity tensor and $\otimes$ is the dyadic product of vectors. No additional
residual equations are required for the regularized form since the formulation does not introduce
any degrees of freedom. The modified stiffness matrix includes one extra term on the diagonal of
each block $\mbs{k}^{ab}$ but does not modify the profile of the global tangent matrix. The
additional cost of the regularized solution is, therefore, negligible.

\bibliographystyle{\mybibstyle}
\bibliography{biblio}

\end{document}